\documentclass[11pt]{report}
\newcommand{\filename}{Mikhailov-JMAA2010-6.tex}
\usepackage{amsthm}
\usepackage{amsfonts}
\usepackage{amssymb}
\usepackage{amsmath}
\usepackage{epsfig}
\usepackage[]{graphicx}
\graphicspath{{./figures/}{../figures/}{../../figures/}}
\usepackage{color}
\usepackage{dsfont} 
\newtheorem{lemma}{LEMMA}[section]
\newtheorem{theorem}[lemma]{THEOREM}
\newtheorem{definition}[lemma]{DEFINITION}
\newtheorem{corollary}[lemma]{COROLLARY}
\newtheorem{remark}[lemma]{REMARK}

\def\supp{\mathop{\mbox{\rm supp}}\nolimits}
\newcommand{\be}{\begin{equation}}
\newcommand{\ee}{\end{equation}}
\newcommand{\bes}{\begin{equation*}}
\newcommand{\ees}{\end{equation*}}
\newcommand{\bea}{\begin{eqnarray}}
\newcommand{\eea}{\end{eqnarray}}
\newcommand{\beas}{\begin{eqnarray*}}
\newcommand{\eeas}{\end{eqnarray*}}
 \newcommand{\nc}{\newcommand}
 \nc{\ha}{\frac{1}{2}}
 \nc{\tha}{\frac{3}{2}}
 \nc{\s}{\widetilde}

 \nc{\ov}{\overline}
 \nc{\pa}{\partial}
 \nc{\pO}{{\partial\Omega}}

\newcommand{\RE}{\text{Re}}

\nc{\C}{{\mathbb{C}}}
\newcommand{\D}{{\cal D}}
\newcommand{\E}{{\cal E}}
\newcommand{\F}{{\cal F}}

\newcommand{\J}{{\cal J}}

 \nc{\N}{{\mathbb{N}}}
\renewcommand{\P}{{\cal P}}

\newcommand{\R}{{\mathds R}}

\newcommand{\V}{{\cal V}}

\newcommand{\ve}{{\varepsilon}}

\makeatletter \@addtoreset{equation}{section}\makeatother

 \makeatletter
\renewcommand{\@oddhead}{\vbox{\hbox to\textwidth{\scriptsize %
 JMAA, {\bf 378}, 2011, 324-342\hfill S.E.Mikhailov
 }\hrule
 }}
 \renewcommand{\@evenhead}{\vbox{\hbox to\textwidth{\scriptsize %
 \filename\hfill\today
 }\hrule
 }}
 \makeatother
 \allowdisplaybreaks
 \numberwithin{equation}{section}
\begin{document}

\title 
{\bf Traces, Extensions, Co-normal Derivatives and Solution Regularity of Elliptic Systems with Smooth and Non-smooth Coefficients}

\author
{Sergey E. Mikhailov\\
     Brunel University London,
     Department of Mathematics,\\
     Uxbridge, UB8 3PH, UK}

\date{}

 \maketitle
 
\chapter{Traces, extensions and co-normal derivatives for elliptic systems on Lipschitz domains}

\begin{center}{Sergey E. Mikhailov\footnote{Corresponding author:
e-mail: {\sf sergey.mikhailov@brunel.ac.uk}, Phone: +44\,189\,267361,
     Fax: +44\,189\,5269732}\\
     Brunel University London,
     Department of Mathematics,\\
     Uxbridge, UB8 3PH, UK}
\end{center}

 \noindent{Published in: {\em J. Math. Analysis Appl.} {\bf 378}, 2011, 324-342 
 \hfill\phantom{*} 
}\\

{\bf Abstract}\\
For functions from the Sobolev space $H^s(\Omega)$, $\ha<s<\tha$,
definitions of non-unique generalized  and unique canonical
co-normal derivative are considered, which are related to possible
extensions of a partial differential operator and its right hand
side from the domain $\Omega$, where they are prescribed, to the
domain boundary, where they are not. Revision of  the boundary value problem settings, which makes them insensitive to the generalized co-normal derivative inherent non-uniqueness are given. It is shown, that the canonical co-normal derivatives, although defined on a more narrow function class than the generalized ones, are continuous extensions of the classical co-normal derivatives. Some new results about trace operator estimates and Sobolev spaces characterizations, are also presented.

\noindent{\bf Keywords}. Partial differential equation systems,  Sobolev spaces, Classical, generalized and canonical co-normal derivatives, Weak BVP settings.


\section{Introduction}

While considering a second order partial differential equation
for a function from the Sobolev space $H^{s}(\Omega)$,
$\ha<s<\tha$, with a  right-hand side from $H^{s-2}(\Omega)$, the {\em strong}  co-normal derivative of $u$ defined on  the boundary  in the trace sense, does not generally exist. Instead, a
{\em generalized}  co-normal derivative operator can be defined by the first Green identity. However this
definition is related to an extension of the PDE operator and its right
hand side from the domain $\Omega$, where they are prescribed, to
the domain boundary, where they are not. Since the extensions are
non-unique, the generalized co-normal derivative operator appears to be
non-unique and non-linear unless a
linear relation between the PDE solution and the extension of its right hand side
is enforced. This leads to the need of a revision of  the boundary value problem settings, which makes them insensitive to the co-normal derivative inherent non-uniqueness.
For functions $u$ from a subspace of $H^{s}(\Omega)$, $\ha<s<\tha$,
which can be mapped by the PDE operator into the space
$\widetilde{H}^{t}(\Omega)$, $t\ge -\ha$,  one can still define a {\em
canonical} co-normal derivative, which is unique, linear
in $u$ and coincides with the co-normal derivative in the trace
sense if the latter does exist.

These notions were developed, to some extent, in \cite{Liverpool2005UKBIM,
MikMMAS2006} for a PDE with an infinitely smooth coefficient on a
domain with an infinitely smooth boundary, and a right hand side
from ${H}^{s-2}(\Omega)$, $1\le s<\tha$, or extendable to
$\widetilde{H}^{t}(\Omega)$, $t\ge -1/2$. In \cite{MikhailovIMSE06}  the
analysis was generalized to the co-normal derivative operators
for some scalar PDE with a H\"older coefficient and right hand side
from ${H}^{s-2}(\Omega)$, $\ha<s<\tha$, on a bounded Lipschitz
domain $\Omega$.

In this paper updating \cite{MikhailovArXiv2009}, we extend the previous results on the co-normal derivatives to strongly elliptic second order PDE systems on bounded or unbounded Lipschitz domains with infinitely smooth  coefficients, with complete proofs. We also give the week BVP settings invariant to the generalized co-normal derivatives non-uniqueness.  To obtain these results, some new facts about trace operator estimates and Sobolev spaces characterizations are also proved in the paper.

The paper is arranged as follows.  Section~\ref{S2} provides a number of
auxiliary facts on Sobolev spaces, traces and extensions, some of which might
be new for Lipschitz domains. Particularly, we proved Lemma
\ref{tauunb} on two-side estimates of the trace operator, Lemma
\ref{LipExt} on boundedness of extension operators from boundary
to the domain for a wider range of spaces, Theorem \ref{McL3.40m-2} on characterization of the
Sobolev space $H_0^s(\Omega)=\s{H}^s(\Omega)$ on the (larger than
usual) interval $\ha<s<\tha$,  Theorem \ref{H_F=0} on
characterization of the space $H^t_{\pO}$, $t>-\tha$, Theorem~\ref{H0=H} on equivalence of $H_0^s(\Omega)$ and ${H}^s(\Omega)$ for $s\le\ha$, Theorem~\ref{Trasenonex} on non-existence of the trace operator, Lemma~\ref{ExtH} and  Theorem~\ref{ExtOper} on extension of
$H^s(\Omega)$ to $\s{H}^s(\Omega)$ for all $s<\ha$, $s\not=\ha-k$.

The results of Section~\ref{S2} are  applied
in Section \ref{S3sm} to
introduce and analyze the generalized and canonical co-normal derivative operators on bounded and unbounded Lipschitz
domains, associated with strongly elliptic systems of second order PDEs with infinitely smooth coefficients and right hand side
from ${H}^{s-2}(\Omega)$, $\ha<s<\tha$. The weak settings of Dirichlet, Neumann and mixed problems (revised versions for the latter two) are considered and it is shown that they are well posed in spite of the inherent non-uniqueness of the generalized co-normal derivatives. It is proved that the canonical co-normal derivative coincides with the classical (strong) one for the cases when they both do exist.

The results of Section~\ref{S3sm} are generalized to H\"older-Lipschitz coefficients in \cite{JCE-H}, see also \cite{MikhailovArXiv2009}.

%
\section{Sobolev spaces, trace operators and extensions}\label{S2}
\subsection{Notations}\label{S2.1}
Suppose $\Omega=\Omega^+$ is a bounded or unbounded open domain of
$\R^n$, which boundary $\pa \Omega$ is a simply connected, closed,
Lipschitz $(n-1)-$dimensional set. Let $\ov{\Omega}$ denote the closure of $\Omega$ and $\Omega^-=\R^n\backslash \ov{\Omega}$ its complement. In what follows  ${\cal
D}(\Omega)=C^\infty_{comp}(\Omega)$ denotes the space of Schwartz
test functions, and ${\cal D}^*(\Omega)$ denotes the space of
Schwartz distributions; $ H^s(\R^n)= H^s_2(\R^n)$, $
H^s(\pO)=H^s_2(\pO)$ are the Sobolev  (Bessel potential) spaces, where $s\in
\R$ is an arbitrary real number (see, e.g., \cite{LiMa1}).

 We denote by $\s{H}^s (\Omega)$ the closure of $\D(\Omega)$ in  $ H^s(\R^n)$, which can be characterized as
 $
\s{H}^s (\Omega)=\{g:\;g\in H^s  (\R^n),\; \supp \,g
\subset\ov{\Omega}\}
 $, see e.g. \cite[Theorem 3.29]{McLean2000}.
The space ${H}^s (\Omega)$ consists of restrictions on
$\Omega$ of distributions  from ${H}^s  (\R^n)$,
 ${H}^s (\Omega):=\{g|_{_{\Omega}}:\;g\in{H}^s (\R^n)\}$,
and $H_0^s(\Omega)$ is closure of ${\cal D}(\Omega)$ in
$H^s(\Omega)$. We recall that $H^s(\Omega)$ coincide with the
Sobolev--Slobodetski spaces $W^s_2(\Omega)$ for any non-negative $s$.
We denote ${H}^s_{loc}(\Omega):=\{g : \varphi g\in{H}^s(\Omega)\ \forall\varphi\in\D(\Omega)\}$. For infinite (unbounded) domains $\Omega$ we will use also the notation $H^{s}_{loc}(\ov\Omega):=\{g : \varphi g\in{H}^s(\Omega)\ \forall\varphi\in\D(\ov\Omega)\}$ (for bounded domains $H^{s}_{loc}(\ov\Omega)=H^{s}(\Omega)$).

Note that distributions from ${H}^s (\Omega)$ and $H_0^s(\Omega)$ are defined only in $\Omega$, while distributions from $\s{H}^s (\Omega)$ are defined in $\R^n$ and particularly on the boundary $\pO$. For  $s\ge 0$, we can identify $\s{H}^s (\Omega)$ with the subset of functions from ${H}^s (\Omega)$, whose extensions by zero outside $\Omega$ belong to ${H}^s (\R^n)$, cf. \cite[Theorem 3.33]{McLean2000}, i.e., identify functions $u\in \s{H}^s (\Omega)$ with their restrictions, $u|_\Omega\in {H}^s (\Omega)$. However generally we will distinguish distributions $u\in\s{H}^s (\Omega)$ and $u|_\Omega\in {H}^s (\Omega)$, especially for $s\le-\ha$.

We denote by ${H}^s_{_\pO}$ the subspace of ${H}^s (\R^n)$ (and of
$\s{H}^s (\Omega)$), which elements are supported on $\pO$, i.e.,
 $
{H}^s_{_\pO}:=\{g:\;g\in H^s  (\R^n),\; \supp \,g \subset{\pO}\}.
 $
To simplify notations for vector-valued functions, $u:\Omega\to\C^m$, we will often write $u\in H^s(\Omega)$ instead of $u\in H^s(\Omega)^m=H^s(\Omega;\C^m)$, etc.

As usual (see e.g. \cite{LiMa1, McLean2000}), for two  elements from dual complex Sobolev spaces the bilinear dual product $\langle \cdot,\cdot\rangle_\Omega$ associated with the sesquilinear inner product  $(\cdot,\cdot)_\Omega:=(\cdot,\cdot)_{L_2(\Omega)}$ in $L_2(\Omega)$ is defined as
 \be
\langle u,v\rangle_{\R^n}:=\int_{\R^n} [\F^{-1} u](\xi)[\F v](\xi)d\xi=:(\F{\bar{u}},{\F v})_{\R^n}=:(\bar{u}, v)_{\R^n},
\quad
u\in H^s(\R^n),\ v\in H^{-s}(\R^n),\label{dp}
 \ee
 \bea
\langle u,v\rangle_{\Omega}&:=&\langle u,V\rangle_{\R^n}=:(\bar{u}, v)_{\Omega}\ \text{if}\ u\in \s H^s(\R^n),\ v\in H^{-s}(\Omega),\ v=V|_\Omega \text{ with } V\in H^{-s}(\R^n),\qquad\nonumber\\
\langle u,v\rangle_{\Omega}&:=&\langle U,v\rangle_{\R^n}=:(\bar{u},{v})_{\Omega}\ \text{if}\ u\in H^s(\R^n),\ v\in\s H^{-s}(\Omega),\ u=U|_\Omega \text{ with } U\in H^{s}(\R^n)\label{dp3}
 \eea
for $s\in\R$, where $\bar g$ is the complex conjugate of $g$, while $\F$ and $\F^{-1}$ are the distributional  Fourier transform operator and its inverse, respectively, that for integrable functions take form
 $$
 \hat g(\xi)=[\F g](\xi):=\int_{\R^n}e^{-2\pi i x\cdot\xi}g(x)dx,\quad
 g(x)=[\F^{-1} \hat g](x):=\int_{\R^n}e^{2\pi i x\cdot\xi}\hat g(\xi)d\xi .
 $$
For  vector-valued elements $u\in H^s(\R^n)^m$, $v\in H^{-s}(\R^n)^m$, $s\in\R$,  definition  \eqref{dp} should be understood as
 \bes
\langle u,v\rangle_{\R^n}:=\int_{\R^n} {\hat u(\xi)}\cdot\hat v(\xi)d\xi=\int_{\R^n} {\hat u(\xi)}^\top\hat v(\xi)d\xi=:(\ov{\hat u},{\hat v})_{\R^n}=:(\bar{u},{v})_{\R^n},  
 \ees
where $\hat u\cdot\hat v=\hat u^\top\hat v=\sum_{k=1}^m\hat u_k\hat v_k$ is the scalar product of two vectors.

Let
$\mathcal{J}^s$ be the Bessel potential operator defined as
$$[\J^sg](x)=\F^{-1}_{\xi\to x}\{(1+|\xi|^2)^{s/2}\hat g(\xi)\}.$$
The inner product  in $H^s(\Omega)$, $s\in\R$,  is defined as follows,
\bea
 \label{ipRn}
&&\hspace{-2em}(u,v)_{H^s(\R^n)}:=(\J^s u, \J^s v)_{\R^n}=
\int_{\R^n}(1+\xi^2)^s\ov{\hat u(\xi)}{\hat v}(\xi) d\xi
=
\left\langle {\ov u},\J^{2s}{v}\right\rangle_{\R^n},\quad u,v\in H^s(\R^n),
\qquad\qquad
 \\
 \nonumber
&&\hspace{-2em}(u,v)_{H^s(\Omega)}:=\left((I-P)U,(I-P)V\right)_{H^s(\R^n)},\quad
u=U|_\Omega,\ v=V|_\Omega,\quad U,V\in H^s(\R^n).\qquad\qquad
\eea
Here $P: H^s(\R^n)\to \s H^s(\R^n\backslash \bar \Omega)$ is the orthogonal projector, see e.g. \cite[p. 77]{McLean2000}.

For a general Lipschitz domain $\Omega$, let $\{\omega_j\}_{j=1}^J\subset\R^n$  be a finite open cover of $\pO$ and $\{\varphi_j(x)\in\D(\omega_j)\}_{j=1}^J$ be a partition of unity subordinate to it, $\sum_{j=1}^J\varphi_j(x)=1$ for any $x\in\pO$. For any $j$ there exists a half-space domain $\Omega_j$ such that  $\omega_j\bigcap\Omega_j=\omega_j\bigcap\Omega$ and $\Omega_j$ can be linearly transformed  by a rigid translation $\kappa_j$ to a Lipschitz hypograph $\tilde\Omega_j=\{x'\in\R^{n-1} : x_n>\zeta_j(x')\}$, where $\zeta_j$ are some uniformly
Lipschitz functions. Let also $ \varkappa_j: \R^n \to \R^n$ be the  Lipschitz-smooth invertible functions (evidently related to $\zeta_j$ and $\kappa_j$) such that $\R^n_+\ni x\mapsto \varkappa_j(x)\in\Omega_j$, while $D_j(x')$ are the Jacobians of the corresponding boundary mappings $\R^{n-1}\ni x'\mapsto \varkappa_j(x')\in\pO_j$ and $D_j\in L_\infty(\R^{n-1})$.


Similar to \cite[page 85]{Necas1967} we introduce the following definition.
\begin{definition}\label{DOktoO}
Let $\Omega_k$, $\Omega$ be Lipschitz domains. We say that $\Omega_k\to \Omega$ as $k\to\infty$ if $\pO_k$ are represented using the same system of covering charts $\omega_j$ as $\pO$ for all sufficiently large $k$, and
\be\label{LipLim}
\lim_{k\to\infty}|\zeta_{jk}-\zeta_j|_{C^{0,1}(\bar\omega_j)}=0,
 \ee
where $\zeta_{jk}$ and $\zeta_j$ are the corresponding Lipschitz functions for the boundary representation.
\end{definition}

\subsection{Sobolev spaces characterization, traces and extensions}\label{S2.2}
To introduce generalized co-normal derivatives in
Section~\ref{S3sm}, we will need several facts about traces and extensions in
Sobolev spaces on Lipschitz domain. First we give the following usual definition of the trace operator.

\begin{definition}\label{TraceDef} An operator
$\gamma^{+}: H^s(\Omega^+)\to H^{\sigma}(\pO)$ is a trace operator if for each $u\in H^s(\Omega)$ and for any sequence $\phi_k\in\D(\ov\Omega)$ converging to $u$ in $H^s(\Omega)$, the sequence of the boundary values  $\phi_k|_\pO$ converges to $\gamma^+u$ in $H^\sigma(\pO)$. The trace operator $\gamma^{-}: H^s(\Omega^-)\to H^{\sigma}(\pO)$ is defined similarly. If $\gamma^+u=\gamma^-u$ we denote them as $\gamma u$.
\end{definition}
 We have the following well-known trace theorem
\cite[Lemma 3.6]{Costabel1988}.
\begin{theorem}\label{TrTh}
If $\ha<s<\tha$, then the trace operators
 \be \label{gammacont}
 \gamma:
H^s(\R^n)\to H^{s-\ha}(\pO)\quad \mbox{and}\quad \gamma^{\pm}: H^s(\Omega^\pm)\to
H^{s-\ha}(\pO),
\ee
are continuous for any Lipschitz domain $\Omega$.
\end{theorem}
Let $\gamma^*: H^{\ha-s}(\pO)\to H^{-s}(\R^n)$ denote the operator adjoined to the trace operator,
 $$
\langle \gamma^*v, w\rangle=\langle v,\gamma w\rangle\quad\forall\ w\in H^{s}(\R^n),
\quad v\in H^{\frac{1}{2}-s}(\partial\Omega).
 $$
Now we can prove two-side estimates for the trace operator and its adjoined, which particularly imply a statement about the trace operator unboundedness (cf. \cite[Chapter 1, Theorem 9.5]{LiMa1} for  the unboundedness statements in domains with infinitely smooth boundary).

\begin{lemma}\label{tauunb}
Let $\Omega$ be a  Lipschitz domain and $\ha<s\le 1$. Then
 \be
  C'\sqrt{C_s}\|v\|_{H^{\ha-s}(\pO)}\le \|\gamma^*v\|_{H^{-s}(\R^n)}\le C''\sqrt{C_s}\|v\|_{H^{\ha-s}(\pO)}\quad \forall v\in H^{\ha-s}(\pO) \label{gsest}
 \ee
 and thus
 \be
  C'\sqrt{C_s}\le\|\gamma\|_{H^s(\R^n)\to H^{s-\ha}(\pO)}=\|\gamma^*\|_{H^{\ha-s}(\R^{n-1})\to H^{-s}(\R^n)}\le C''\sqrt{C_s},
 \label{ges}
 \ee
where
 $$C_s:=\int_{-\infty}^\infty(1+\eta^2)^{-s}d\eta,$$
$C'$ and $C''$ are positive constants independent of $s$ and $v$.
The norm of the  trace operator
$\gamma : {H^s(\R^n)\to H^{s-\ha}(\pO)}$ tends to infinity as
$s\searrow\ha$ since $C_s\to\infty$,  while the operator $\gamma: H^\ha(\R^n)\to
L_2(\pO)$, if it does exist, is unbounded.
\end{lemma}
\begin{proof}
Let first consider the lemma for the half-space,
$\Omega=\R^n_+=\{x\in\R^n : x_n>0\}$, where $x=\{x',x_n\}$, $x'\in\R^{n-1}$. For $v\in H^{\ha-s}(\R^{n-1})$, taking into account the uniqueness of the trace operator for $s>\ha$,
the distributional Fourier transform gives
 $${\cal F}_{x\to \xi}\{\gamma^*v\}={\cal F}_{x'\to
 \xi'}\{v(x')\}=:\hat{v}(\xi').$$
Then  we have,
 \begin{multline}
\|\gamma^*v\|_{H^{-s}(\R^n)}^2=\int_{\R^n}(1+|\xi|^2)^{-s}
|\hat{v}(\xi')|^2
d\xi\\
=\int_{\R^{n-1}}
\left[\int_{\R}(1+|\xi'|^2+|\xi_n|^2)^{-s}d\xi_n\right]|\hat{v}(\xi')|^2d\xi'=
 C_s\|v\|^2_{H^{\ha-s}(\R^{n-1})},\label{g*}
 \end{multline}
where the substitution $\xi_n=(1+|\xi'|^2)^{\ha}\eta$ was used, cf. \cite[Chap. 2, Proposition 4.6]{ChazPir1982}. Thus
 $$
 \|\gamma\|_{H^s(\R^n)\to H^{s-\ha}(\R^{n-1})}=\|\gamma^*\|_{H^{\ha-s}(\R^{n-1})\to H^{-s}(\R^n)}=\sqrt{C_s}\to\infty\quad \text{ as } s\searrow \ha.
 $$

On the other hand, by \eqref{g*} the norm  $\|\gamma^*v\|_{H^{-\ha}(\R^n)}$ is not finite for any non-zero $v$. This means the operator $\gamma^*:{H^{0}(\R^{n-1})\to H^{-\ha}(\R^n)}$  and thus the operator $ \gamma:{H^\ha(\R^n)\to H^{0}(\R^{n-1})}$ is not bounded, which completes the lemma for $\Omega=\R^n_+$ with $C'=C''=1$.

Let now $\Omega$ be a general Lipschitz domain.
For $v\in L_2(\pO)$, $w\in \D(\R^n)$, using the boundary cover and corresponding partition of unity as in Section~\ref{S2.1} we have,
\begin{multline*}
\langle \gamma^*v, w\rangle_{\R^n}=\langle v,\gamma w\rangle_\pO=\int_\pO v(x)w(x)d\sigma(x)=
\sum_{j=1}^J\int_\pO \varphi_j(x)v(x)w(x)d\sigma(x)=\\
\sum_{j=1}^J\int_{\R^{n-1}} [({\varphi_j}v)\circ\varkappa_j](x')[w\circ\varkappa_j](x')D_j(x')dx'=\\
\sum_{j=1}^J\langle D_j({\varphi_j}v)\circ\varkappa_j, \gamma_0[w\circ\varkappa_j]\rangle_{\R^{n-1}}=
\sum_{j=1}^J\langle \gamma_0^*[D_j({\varphi_j}v)\circ\varkappa_j], w\circ\varkappa_j\rangle_{\R^{n}},
\end{multline*}
where  $\gamma_0$, $\gamma_0^*$ are the trace operator on $\R^n_+$ and its adjoined, respectively. Taking into account density of $\D(\R^n)$ in $H^s(\R^n)$ and of $L_2(\pO)$ in $H^{\ha-s}(\pO)$, we have,
 \begin{multline}\label{gsvG}
\|\gamma^*v\|_{H^{-s}(\R^n)}=\sup_{w\in H^{s}(\R^n)}
\frac{|\langle \gamma^*v, w\rangle_{\R^n}|}{\|w\|_{H^{s}(\R^n)}}
=
\sup_{w\in H^{s}(\R^n)}
\left|\sum_{j=1}^J\left\langle \gamma_0^*[D_j({\varphi_j}v)\circ\varkappa_j], \frac{w\circ\varkappa_j}{\|w\|_{H^{s}(\R^n)}}\right\rangle_{\R^{n}}\right|
 \end{multline}
for any $v\in H^{\ha-s}(\pO)$.

It is well known (see e.g. \cite[Theorem 3.23 and p. 98]{McLean2000}) that
 \bea
\|v\|^2_{H^{\ha-s}(\pO)}=\sum_{j=1}^J\|D_j({\varphi_j}v)\circ\varkappa_j\|^2_{H^{\ha-s}(\R^{n-1})},\quad \ha< s\le \tha, \label{est2G}\\
\tilde C'\|w\|_{H^s(\R^n)}\le\|w\circ\varkappa_j\|_{H^s(\R^n)}\le \tilde C'' \|w\|_{H^s(\R^n)},\quad j=1,...,J,\quad 0\le s\le 1,
 \label{estG}
 \eea
where $\tilde C',\tilde C''$ are some positive constants independent of $s$. By \eqref{g*} and \eqref{est2G},
 $$
\|\gamma_0^*[D_j({\varphi_j}v)\circ\varkappa_j]\|_{H^{-s}(\R^n)}=
\sqrt{C_s}\|D_j({\varphi_j}v)\circ\varkappa_j\|_{H^{\ha-s}(\R^{n-1})}\le
\sqrt{C_s}\|v\|_{H^{\ha-s}(\pO)}.
 $$
Then   \eqref{gsvG} and \eqref{estG} imply
 \bes
 \|\gamma^*v\|_{H^{-s}(\R^n)}\le \tilde C''J\sqrt{C_s}\|v\|_{H^{\ha-s}(\pO)}\quad \forall v\in H^{\ha-s}(\pO),
 \ees
which is the right inequality in  \eqref{gsest}.

On the other hand, we have for $v\in L_2(\pO)$, $w\in \D(\R^n)$,
\begin{multline*}
\langle \varphi_j\gamma^*v, w\rangle_{\R^n}=\langle v,\gamma (\varphi_jw)\rangle_\pO=
 \int_\pO v(x)\varphi_j(x)w(x)d\sigma(x)=\\
 \int_{\pO\cap\omega_j} v(x)\varphi_j(x)w(x)d\sigma(x)=
\int_{\R^{n-1}} [(\varphi_jv_j)\circ\varkappa_j](x')[w\circ\varkappa_j](x')D_j(x')dx'=\\
\langle D_j [(\varphi_jv_j)\circ\varkappa_j], \gamma_0[w\circ\varkappa_j]\rangle_{\R^{n-1}}=
\langle \gamma_0^*\{D_j[(\varphi_jv_j)\circ\varkappa_j]\}, w\circ\varkappa_j\rangle_{\R^{n}}.
\end{multline*}
By \eqref{estG} this implies,
 \begin{multline}\|\varphi_j\gamma^*v\|_{H^{-s}(\R^n)}
=\sup_{w\in H^{s}(\R^n)}
\left|\left\langle \gamma_0^*\{D_j[(\varphi_jv)\circ\varkappa_j]\}, \frac{w\circ\varkappa_j}{\|w\|_{H^{s}(\R^n)}}\right\rangle_{\R^{n}}\right|
=\\
\sup_{w\in H^{s}(\R^n)}
\left|\left\langle \gamma_0^*\{D_j[(\varphi_jv)\circ\varkappa_j]\}, \frac{w\circ\varkappa_j}{\|w\circ\varkappa_j\|_{H^{s}(\R^n)}}\right\rangle_{\R^{n}}
\ \frac{\|w\circ\varkappa_j\|_{H^{s}(\R^n)}}{\|w\|_{H^{s}(\R^n)}}\right|\ge\\
\tilde C'\sup_{w\in H^{s}(\R^n)}
\left|\left\langle \gamma_0^*\{D_j[(\varphi_jv)\circ\varkappa_j]\}, \frac{w\circ\varkappa_j}{\|w\circ\varkappa_j\|_{H^{s}(\R^n)}}\right\rangle_{\R^{n}}
\right|=\tilde C'\|\gamma_0^*\{D_j[(\varphi_jv)\circ\varkappa_j]\}\|_{H^{-s}(\R^n)},
 \label{2.16a}
 \end{multline}
that is by \eqref{g*} and \eqref{est2G},
\begin{multline}\label{2.18}
\sum_{j=1}^J\|\varphi_j\gamma^*v\|_{H^{-s}(\R^n)}^2\ge
 \tilde C'^2\sum_{j=1}^J\|\gamma_0^*\{D_j[(\varphi_jv)\circ\varkappa_j]\}\|_{H^{-s}(\R^n)}^2=\\
 \tilde C'^2{C_s}\sum_{j=1}^J\|D_j[(\varphi_jv)\circ\varkappa_j]\|_{H^{\ha-s}(\R^{n-1})}^2=
\tilde C'^2C_s\|v\|_{H^{\ha-s}(\pO)}^2.
\end{multline}
Since
 \be\label{2.17a}
 \tilde C_j\|\gamma^*v\|_{H^{-s}(\R^n)}\ge\|\varphi_j\gamma^*v\|_{H^{-s}(\R^n)}
 \ee
for $\varphi_j\in\D(\R^n)$, \eqref{2.18} gives the left inequality in \eqref{gsest}.

Obviously, \eqref{gsest} implies \eqref{ges} for $\gamma^*$ and thus for $\gamma$.

As was shown in the first paragraph of the proof,  the functional $\gamma_0^*\{D_j[(\varphi_jv)\circ\varkappa_j]\}$ is not bounded on $H^{\ha}(\R^n)$  for any non-zero $v$, then \eqref{2.16a}, \eqref{2.17a} imply that the operator $\gamma^*:{H^{0}(\pO)\to H^{-\ha}(\R^n)}$  and thus the operator $\gamma:{H^\ha(\R^n)\to H^{0}(\pO)}$ is not bounded.
\end{proof}

For $s>3/2$ the trace operators \eqref{gammacont} are not continuous on Lipschitz domains, however the following weaker statement holds, which was mentioned in \cite{Ding1996} without a proof but can be indeed proved by appropriate estimates of an integral on p. 598 of \cite{Ding1996} for this case.
\begin{lemma}\label{DingMent} If $\Omega$ is a Lipschitz domain and $s>3/2$, then
the trace operators
 \bes
 \gamma:
H^s(\R^n)\to H^{1}(\pO)\quad \mbox{and}\quad \gamma^{\pm}: H^s(\Omega^\pm)\to
H^{1}(\pO)
\ees
are continuous.
\end{lemma}

\begin{lemma}\label{LipExt}
For a Lipschitz domain $\Omega$ there exists a linear
bounded extension operator $\gamma_{-1}: H^{s-\ha}(\pO)\to H^{s}(\R^n)$,
$\ha\le s\le\tha$, which is right inverse to the trace operator
$\gamma$, i.e., $\gamma \gamma_{-1} g=g$ for any $g\in H^{s-\ha}(\pO)$.
(For $s=\ha$ the trace operator $\gamma$ is understood not as in Definition~\ref{TraceDef} but in the non-tangential sense, see, e.g. \cite{JK1981AM}.)
Moreover, $\|\gamma_{-1}\|_{ H^{s-\ha}(\pO)\to H^{s}(\R^n)}\le C$, where $C$
is independent of $s$.
\end{lemma}
\begin{proof}
For Lipschitz domains and $\ha<s\le 1$, the boundedness of the
extension operator is well known, see e.g. \cite[Theorem
3.37]{McLean2000}.

To prove it for the whole range $\ha\le s\le\tha$, let us consider the Green operator $G_\Delta$ that solves the Dirichlet Problem for the Laplace equation in $\Omega$ and continuously maps $H^{s-\ha}(\pO)$ to $H^{s}(\Omega)$ if $\Omega$ is a bounded domain and to $H^{s}_{loc}(\ov\Omega)$ if $\Omega$ is an unbounded domain. Particularly one can take $G_\Delta=V_\Delta\V_\Delta^{-1}$, where the single layer potential $V_\Delta \varphi$ with a
density $\varphi=\V_\Delta^{-1} g\in H^{s-\tha}(\pO)$, solves the
Laplace equation in $\Omega$ with the Dirichlet boundary data
$g$ and $\V_\Delta$ is the direct value of the operator
$V_\Delta$ on the boundary. The operators
$\V_\Delta^{-1}:H^{s-\ha}(\pO)\to H^{s-\tha}(\pO)$ and
$V_\Delta:H^{s-\tha}(\pO) \to H^{s}_{loc}(\R^n)$ are continuous
for  $\ha\le s\le\tha$ as stated in \cite{JK1981BAMS, JK1981AM,
JK1982, Verchota1984, Costabel1988}. Thus it suffice to take
$\gamma_{-1}=\chi G_\Delta $, where $\chi\in\D(\R^n)$ is a
cut-off function such that $\chi=1$ in a sufficiently large open ball such that it includes the boundary $\pO$.
The estimate $\|\gamma_{-1}\|_{ H^{s-\ha}(\pO)\to H^{s}(\R^n)}\le C$, where
$C$ is independent of $s$, then follows.
\end{proof}
Note that continuity of the operator $\gamma$ was not needed in the proof.

Let us denote by $E_0 $  the operator of extension of a function defined in $\Omega$ by zero outside $\Omega$ to a function defined in $\R^n$.
\begin{theorem}\label{McL3.40m-1}
Let $\Omega$ be a Lipschitz domain and $s\ge 0$ while $s\not=\ha+k$ for any integer $k\ge 0$. Then
$$\s{H}^s(\Omega)=H_0^s(\Omega)$$
in the sense that $u|_{\Omega}\in H_0^s(\Omega)$ for any $u\in\s{H}^s(\Omega)$, and $E_0 v\in\s{H}^s(\Omega)$ for any $v\in {H}^s_0(\Omega)$. Moreover
\be \|u|_{\Omega}\|_{H^s(\Omega)}\le \|u\|_{\s{H}^s(\Omega)},\qquad\|E_0 v\|_{\s H^{s}(\Omega)}\le C \|v\|_{H^{s}(\Omega)}\label{tnormest2},\ee
where $C$ depends only on $s$ and on the maximum of the Lipschitz constants of the representation functions $\zeta_j$ for the boundary $\pO$, see Section~\ref{S2.1}.
\end{theorem}
\begin{proof} The first claim is proved in \cite[Theorem 3.33]{McLean2000}. The first estimate in \eqref{tnormest2} is evident, while the second follows from the proofs of the same Theorem 3.33 and Lemma 3.32 in \cite{McLean2000}.
\end{proof}

To characterize the space $H_0^s(\Omega)=\s{H}^s(\Omega)$ for
$\ha<s<\tha$, we will need the following statement.
\begin{lemma}\label{McL3.31}
If $\Omega$ is a Lipschitz domain and $u\in H^s(\Omega)$,
$0<s<\ha$, then
 \begin{equation}
\int_\Omega \mbox{dist}(x,\pO)^{-2s}|u(x)|^2dx\le
C\|u\|^2_{H^s(\Omega)}\label{McL3.31e}
 \end{equation}
and for a given boundary cover the constant $C$ depends only on $s$ and on the maximum of the Lipschitz constants of the boundary representation functions $\zeta_j$, see Section~\ref{S2.1}.
\end{lemma}

\begin{proof} Note first that the lemma claim  for
$u\in\D(\ov{\Omega})$ follows from  the proof of \cite[Lemma 3.32]{McLean2000}. To prove
it for $u\in H^s(\Omega)$, let first the domain $\Omega$ be such
that
\begin{equation}
\mbox{dist}(x,\pO)<C_0<\infty\label{C_0}
\end{equation}
for all $x\in\Omega$, which holds true particularly for bounded
domains. Let $\{\phi_k\}\in\D(\ov{\Omega})$ be a sequence converging to $u$
in $H^s(\Omega)$. If we denote $w(x)=\mbox{dist}(x,\pO)^{-2s}$, then
$w(x)>C_0^{-2s}>0$. Since \eqref{McL3.31e} holds for functions from
$\D(\ov{\Omega})$, the sequence $\{\phi_k\}\in\D(\ov{\Omega})$ is
fundamental in the weighted space $L_2(\Omega,w)$, which is complete,
implying that $\phi_k\in\D(\ov{\Omega})$ converges in this space to a
function $u'\in L_2(\Omega,w)$. Since both $L_2(\Omega,w)$ and
$H^s(\Omega)$ are continuously imbedded in the non-weighted space
$L_2(\Omega)$, the sequence $\{\phi_k\}$ converges in $L_2(\Omega)$
implying the limiting functions $u$ and $u'$ belong to this space and
thus coincide. Then from \eqref{McL3.31e} for $\phi_k$ we immediately obtain it for arbitrary $u\in H^s(\Omega)$.

For the unbounded domains for which condition \eqref{C_0} is not satisfied,  let
$\chi(x)\in\D(\R^n)$ be a cut-off function such that $0\le\chi(x)\le 1$
for all $x$, $\chi(x)=1$ near $\pO$, while $w(x)<1$ for $x\in\supp\
(1-\chi)$.   Then \eqref{C_0} is satisfied in
$\Omega'=\Omega\bigcap\supp\chi(x)$ and
\begin{multline*}
\int_\Omega w(x)|u(x)|^2dx= \int_\Omega(1-\chi(x))w(x)|u(x)|^2dx
 + \int_\Omega \chi(x)w(x)|u(x)|^2dx\le\\
 \|u\|^2_{L_2(\Omega)}+\int_{\Omega'} w(x)|\sqrt{\chi(x)}u(x)|^2dx
\le \|u\|^2_{H^s(\Omega)}+C\|\sqrt{\chi(x)}u\|^2_{H^s(\Omega')}\le
C_1\|u\|^2_{H^s(\Omega)}
 \end{multline*}
due to the previous paragraph.
\end{proof}

Lemma \ref{McL3.31} allows now extending the following statement
known for $\ha<s\le 1$, see  \cite[Theorem 3.40(ii)]{McLean2000},
to a wider range of $s$.

\begin{theorem}\label{McL3.40m-2}
If $\Omega$ is a Lipschitz domain and $\ha<s<\tha$, then
\be H_0^s(\Omega)=\{u\in H^s(\Omega) : \gamma^+u=0\}.\label{2.5E2}\ee
\end{theorem}
\begin{proof}
Equality \eqref{2.5E2}  for $\ha<s\le 1$ is stated in
\cite[Theorem 3.40(ii)]{McLean2000}.

Let $1<s<\tha$. If $u\in H^s_0(\Omega)$
then evidently $\gamma^+u=0$ since $\D$ is dense in
${H}^s_0(\Omega)$ and the trace operator $\gamma^+$ is bounded in
${H}^s(\Omega)$.
To prove that any $u\in{H}^s(\Omega)$ with $\gamma^+u=0$ belongs to ${H}^s_0(\Omega)$, it remains, due to Theorem~\ref{McL3.40m-1}, to prove that $E_0{u}\in{H}^s(\R^n)$. We remark first
of all that $E_0{u}\in {H}^1(\R^n)$ due to the previous paragraph and Theorem~\ref{McL3.40m-1}, and then make estimates similar to those in
the proof of \cite[Theorem 3.33]{McLean2000},
\begin{multline*}
\|E_0{u}\|^2_{{H}^s(\R^n)}\sim
\|E_0{u}\|^2_{W^1_2(\R^n)}+\int_{\R^n}\int_{\R^n}\frac{|\nabla
E_0{u}(x)-\nabla E_0{u}(y)|^2}{|x-y|^{2(s-1)+n}}\ dx\ dy\\
= \|u\|^2_{W^1_2(\Omega)}+\int_{\Omega}\int_{\Omega}\frac{|\nabla
u(x)-\nabla u(y)|^2}{|x-y|^{2(s-1)+n}}\ dx\ dy\\
+\int_{\R^n\backslash\Omega}\int_{\Omega}\frac{|\nabla
u(x)|^2}{|x-y|^{2(s-1)+n}}\ dx\ dy
+\int_{\Omega}\int_{\R^n\backslash\Omega}\frac{|\nabla
u(y)|^2}{|x-y|^{2(s-1)+n}}\ dx\ dy\\
=\|u\|^2_{W^s_2(\Omega)} +2\int_{\Omega}|w_{s-1}(x)\nabla u(x)|^2\
dx,
\end{multline*}
where
 $$
w_{s-1}(x):=\int_{\R^n\backslash\Omega}\frac{dy}{|x-y|^{2(s-1)+n}},\quad
x\in\Omega,
 $$
and $W^s_2(\Omega)$ is the Sobolev-Slobodetski space. Introducing
spherical coordinates with $x$ as an origin, we obtain, $w_{s-1}(x)\le
\frac{\alpha_n}{2(s-1)} \mbox{dist}(x,\pO)^{-2(s-1)}$ for $x\in\Omega$, where $\alpha_n$ is the area of the unit sphere in $\R^n$. Then, taking
into account that $\nabla{u}\in H^{s-1}(\Omega)$ and
$\|\nabla{u}\|_{H^{s-1}(\Omega)}\le \|{u}\|_{H^{s}(\Omega)}$, we
have  by Lemma \ref{McL3.31},
 $$
\|E_0{u}\|^2_{{H}^s(\R^n)}\le
\|u\|^2_{W^s_2(\Omega)} +2C\|u\|^2_{H^s(\Omega)}\le
C_{s}\|u\|^2_{H^s(\Omega)}\ .
 $$
Theorem \ref{McL3.40m-1} completes the proof.
\end{proof}

Let us now give a characterization of the space $H^t_{\pO}$.

\begin{theorem}\label{H_F=0}
Let $\Omega$ be a Lipschitz domain  in $\R^n$.


(i)  If $ t\ge -\ha$, then $H^t_{\pO}=\{0\}$.

(ii) If $-\tha< t< -\ha$, then $g\in H^t_{\pO}$ if and only if $g=\gamma^*v$, i.e.,
\begin{equation}\label{H^t_}
  \langle g, W \rangle_{\R^n}=\langle v, \gamma W\rangle_\pO\quad \forall\
  W\in H^{-t}(\R^n),
\end{equation}
with $v=\gamma_{-1}^*g\in H^{t+\ha}(\pO)$, i.e.,
\begin{equation}\label{H^t_v}
 \langle v, w\rangle_\pO\quad = \langle g, \gamma_{-1}w \rangle_{\R^n}\quad\forall\
  w\in H^{-t-\ha}(\pO),
\end{equation}
where $v$ is independent of the choice of the non-unique operators $\gamma_{-1}$, $\gamma_{-1}^*$, and the estimate
$\|v\|_{H^{t+\ha}(\pO)}\le C \|g\|_{H^t(\R^n)}$ holds with $C$
independent of $t$.
\end{theorem}
\begin{proof}
We will follow an idea in the proof of Lemma 3.39 in
\cite{McLean2000} (see also \cite[Proposition 4.8]{ChazPir1982}), extending it from a half-space to a Lipschitz
domain $\Omega$.

Let $\Omega^+=\Omega$ and $\Omega^-=\R^n\backslash\bar{\Omega}$.
For any $\phi\in {\cal D}(\R^n)$, let us define
$$
\phi^\pm(x)=\begin{cases} \phi(x)& \text{if\ } x\in \Omega^\pm,\\
0& \text{if\ } x\not\in \Omega^\pm.
\end{cases}
$$
Let $t> -\ha$. Then $\phi^\pm\in \s{H}^{-t}(\Omega^\pm)$ (see e.g.
\cite[Theorem 3.40]{McLean2000} and Theorem \ref{McL3.40m-1} for $-\ha<t\le 0$, for
greater $t$ it then follows by embedding),
$\|\phi-\phi^+-\phi^-\|_{H^{-t}(\R^n)}=0$, and there exist
sequences $\{\phi^\pm_k\}\in {\cal D}(\Omega^\pm)$ converging to
$\phi^\pm$ in $\s{H}^{-t}(\Omega^\pm)$ as $k\to\infty$. Hence
$\langle g,\phi\rangle_{\R^n}=\lim_{k\to\infty}\langle g,\phi^+_k +
\phi^-_k\rangle_{\R^n} =0$ for any
$\phi\in {\cal D}(\R^n)$
proving (i) for $t> -\ha$ since ${\cal D}(\R^n)$ is dense in $H^{-t}(\R^n)=[H^t(\R^n)]^*$.

Let us prove (ii). For $g\in H^t_\pO$, $-\tha< t< -\ha$, let
$v\in H^{t+\ha}(\pO)$ be defined by \eqref{H^t_v},
where existence and continuity of $\gamma_{-1}: H^{-t-\ha}(\pO)\to H^{-t}(\Omega)$ is proved in Lemma
\ref{LipExt}. Observe that
 $$
|\langle v,w\rangle_\pO|\le \|g\|_{H^t(\R^n)}
\|w\|_{H^{-t-\ha}(\pO)} \|\gamma_{-1}\|_{ H^{-t-\ha}(\pO)\to
H^{-t}(\R^n)},
 $$
so $\|v\|_{H^{t+\ha}(\pO)}\le \|\gamma_{-1}\|_{ H^{-t-\ha}(\pO)\to
H^{-t}(\R^n)}\ \|g\|_{H^t(\R^n)}\le C \|g\|_{H^t(\R^n)}$, where
$C$ is independent of $t$ due to Lemma \ref{LipExt} if $\gamma_{-1}$ is chosen as in that lemma. We also have
that
 $$
\langle g, W \rangle_{\R^n} - \langle v, \gamma W \rangle_\pO=
 \langle g, \rho \rangle_{\R^n} \quad \forall\
  W\in H^{-t}(\R^n),
 $$
where
 $$
\rho=W-\gamma_{-1}\gamma W\in H^{-t}(\R^n).
 $$
Then we have $\gamma\rho=0$, which due to Theorems \ref{McL3.40m-1}, \ref{McL3.40m-2}
implies $\tilde{\rho}^\pm\in \s{H}^{-t}(\Omega^\pm)$, where
$\tilde{\rho}^\pm$ are extensions of $\rho|_{\Omega^\pm}$ by zero
outside $\Omega^\pm$, and $\rho=\tilde{\rho}^+ + \tilde{\rho}^-$. Thus
there exist sequences $\{\rho^\pm_k\}\in {\cal D}(\Omega^\pm)$ converging
to $\tilde{\rho}^\pm$ in $\s{H}^{-t}(\Omega^\pm)$, implying $\langle g,
\rho \rangle_{\R^n}=0$ since $g\in H^t_\pO$, and thus ansatz
\eqref{H^t_}. To prove that $v$ is uniquely determined by $g$ , i.e., independent of $\gamma_{-1}$, let us consider $v'$ and $v''$ corresponding to different operators $\gamma_{-1}'$ and $\gamma_{-1}''$. Then by \eqref{H^t_},
 \begin{multline*}
\langle v'-v'', w\rangle_\pO\quad =\langle \gamma_{-1}^{*\prime}g-\gamma_{-1}^{*\prime\prime}g, w\rangle_\pO\quad =
\langle g, \gamma_{-1}'w - \gamma_{-1}''w \rangle_{\R^n}\\
 =\langle v', \gamma (\gamma_{-1}'w - \gamma_{-1}''w)\rangle_\pO=0 \quad\forall\
  w\in H^{-t-\ha}(\pO).
 \end{multline*}

It remains to deal with the case $t=-\ha$ in (i). Let $g\in
H^{-\ha}_{\pO}$. Since $H^{-\ha}_{\pO}\subset H^t_{\pO}$ for $-\tha< t< -\ha$,
then $g=\gamma^*v$ for some $v\in H^{t+\ha}(\pO)$ $\forall t\in (-\tha, -\ha)$, and $\|g\|_{H^{t}_{\pO}}=\|\gamma^*v\|_{H^{t}_{\pO}}\ge C'\sqrt{C_{-t}}\,\|v\|_{H^{\ha+t}(\pO)}$ owing to Lemma \ref{tauunb}. Since $C_{-t}\to\infty$ as $t\nearrow -\ha$, this means $\|v\|_{H^{\ha+t}(\pO)}\to 0$ as $t\nearrow -\ha$ implying $v=0$.
\end{proof}
 Combining \eqref{H^t_} and \eqref{H^t_v} we have the following useful statement.
\begin{corollary}\label{C2.9}
Let $\Omega$ be a Lipschitz domain  in $\R^n$.
If $g\in H^t_{\pO}$ with $-\tha< t< -\ha$, then $g=\gamma^*\gamma^*_{-1}g$ for any choice of $\gamma^*_{-1}$.
\end{corollary}

{
\begin{theorem}
\label{H0=H}
Let $\Omega$ be a Lipschitz domain  in $\R^n$ and $s\le\ha$. Then
$\D(\Omega)$ is dense in $H^s(\Omega)$, i.e.,
$H^s(\Omega)=H_0^s(\Omega)$.
\end{theorem}
\begin{proof}
The proof for  $0\le s\le\ha$ is available in \cite[Theorem
3.40(i)]{McLean2000}.
To prove the statement for any $s\le\ha$ we remark that
if $w\in H^s(\Omega)^*=\s H^{-s}(\Omega)$ satisfies $\langle
w,\phi\rangle=0$ for all $\phi\in\D(\Omega)$, then $w\in
H^{-s}_\pO$ and Theorem~\ref{H_F=0}(i) implies $w=0$. Hence,
$\D(\Omega)$ is dense in $H^s(\Omega)$, i.e.,
$H^s(\Omega)=H_0^s(\Omega)$.
\end{proof}

Theorem~\ref{H0=H} implies that for any $u\in\D(\ov\Omega)$ and $s\le\ha$ there exists a sequence $\{\phi_k\}\in\D(\Omega)$ converging to $u$ in $H^s(\Omega)$. Evidently $\phi_k|_\pO$ converges to 0 in $H^{\sigma}(\pO)$ for any $\sigma$  since $\phi_k|_\pO=0$. On the other hand,  $u\in\D(\ov\Omega)$ is the limit in $H^s(\Omega)$ of the sequence $\{\phi_k'\}=u$, meaning that $\phi_k'|_\pO$ converges in $H^{\sigma}(\pO)$ to $u|_\pO$, which is generally non-zero. This leads to the following conclusion of non-existence.
\begin{corollary}\label{Trasenonex}
For $s\le\ha$ the trace operators
$\gamma^{\pm}: H^s(\Omega^\pm)\to H^{\sigma}(\pO)$, understood as in Definition~\ref{TraceDef}, do not exist for any $\sigma$.
\end{corollary}

\begin{remark}
(i) Evidently, Corollary \ref{Trasenonex} holds also if the space $H^{\sigma}(\pO)$ is replaced with any Banach space of distributions on $\pO$.

(ii) The trace operator $\gamma^{\pm}: B(\Omega^\pm)\to H^{\sigma}(\pO)$ can, of course, still exist on some Banach subspaces on $\Omega^\pm$, $B(\Omega^\pm)\subset H^s(\Omega^\pm)$, $s\le\ha$, with the norms stronger than the norm in $H^s(\Omega^\pm)$, particularly on $H^t(\Omega^\pm)$, $t>\ha$.
\end{remark}

The following two statements give conditions when distributions from $H^s(\Omega)$ can be extended to distributions from $\s H^s(\Omega)$ and when the extension can be written in terms of a linear bounded operator. The first of them can be considered as a counterpart of Theorem~\ref{McL3.40m-1} for negative $s$.

\begin{lemma}\label{ExtH}
Let $\Omega$ be a Lipschitz domain,   $s<\ha$, $s\not=\ha-k$ for any integer $k>0$. Then for any $g\in H^s(\Omega)$ there exists $\tilde g\in \s H^s(\Omega)$ such that $g=\tilde g|_\Omega$ and $\|\tilde g\|_{\s H^s(\Omega)}\le C\|g\|_{H^s(\Omega)}$, where $C>0$ does not depend on $g$.
\end{lemma}

\begin{proof}
Any distribution $g\in H^s(\Omega)$ is a bounded linear functional on  $\s H^{-s}(\Omega)$. On the other hand, for any $v\in{H}_0^{-s}(\Omega)\subset {H}^{-s}(\Omega)$ its zero extension  $\tilde v=E_0 v$ belongs to $\s{H}^{-s}(\Omega)$ with
 \be\|\tilde v\|_{\s H^{-s}(\Omega)}\le C \|v\|_{H^{-s}(\Omega)}\label{tnormest}\ee
for $s\le 0$, $s\not=\ha-k$, by Theorem~\ref{McL3.40m-1}. This holds true also for $0<s<\ha$ since then $\s{H}^{-s}(\Omega)=[{H}^{s}(\Omega)]^*=[{H}^{s}_0(\Omega)]^*
=[\s{H}^{s}(\Omega)]^*={H}^{-s}(\Omega)$ by Theorems \ref{H0=H} and \ref{McL3.40m-1},
while the extension $\tilde v\in \s H^{-s}(\Omega)$ is defined as
 \bes
\langle\tilde v,w\rangle:=\langle v,E_0 w\rangle\quad \forall\ w\in H^{s}(\Omega), \quad 0<s<\ha,
 \ees
and by Theorems \ref{H0=H} and \ref{McL3.40m-1},
 \begin{multline*}
\|\tilde v\|_{\s H^{-s}(\Omega)}=
\sup_{w\in H^{s}(\Omega)\backslash\{0\}} \frac{|\langle\tilde v,w\rangle|}{\|w\|_{ H^{s}(\Omega)}}=
\sup_{w\in H^{s}(\Omega)\backslash\{0\}} \frac{|\langle v,E_0 w\rangle|}{\|w\|_{ H^{s}(\Omega)}}\\
\le
C\sup_{w\in H^{s}(\Omega)\backslash\{0\}} \frac{|\langle v,E_0 w\rangle|}{\|E_0w\|_{ \s H^{s}(\Omega)}}\le
 C \|v\|_{H^{-s}(\Omega)}.
 \end{multline*}
giving estimate \eqref{tnormest}.

Thus  the functional $g\in H^s(\Omega)$ continuous on $\s H^{-s}(\Omega)$ and thus on $H^{-s}_0(\Omega)$ can be extended by the Hahn-Banach theorem to a functional $\tilde g\in \s H^{s}(\Omega)$ continuous on $H^{-s}(\Omega)$ such that $\|\tilde g\|_{\s H^s(\Omega)}=\|\tilde g\|_{[H^{-s}(\Omega)]^*}= \|g\|_{[H^{-s}_0(\Omega)]^*}$. Then by estimate \eqref{tnormest} for $s<\ha$, $s\not=\ha-k$, we have,
$$
\|g\|_{[H^{-s}_0(\Omega)]^*}=
\sup_{v\in H_0^{-s}(\Omega)\backslash\{0\}} \frac{|\langle g,v\rangle|}{\|v\|_{ H_0^{-s}(\Omega)}} \le C\sup_{\tilde v\in \s H^{-s}(\Omega)\backslash\{0\}}
  \frac{|\langle g,\tilde v\rangle|}{\|\tilde v\|_{\s H^{-s}(\Omega)}}\le
C\|g\|_{[\s H^{-s}(\Omega)]^*}= C\|g\|_{H^s(\Omega)},
$$
which completes the proof.
\end{proof}

\begin{theorem}\label{ExtOper}
Let $\Omega$ be a Lipschitz domain and  $-\tha< s<\ha$, $s\not=-\ha$. There exists a bounded linear extension operator
$\s{E}^s : H^s(\Omega) \to \s{H}^s(\Omega)$, such that
$\s{E}^sg|_\Omega=g$, $\forall\ g\in H^s(\Omega)$. For $-\ha< s<\ha$ the extension operator is unique,  $(\s E^s)^*=\s E^{-s}$ and
 \be\|\s E^s g\|_{\s H^{s}(\Omega)}\le C \|g\|_{H^{s}(\Omega)},\label{tnormest3}
 \ee
where $C$ depends only on $s$ and on the maximum of the Lipschitz constants of the representation functions $\zeta_j$ for the boundary $\pO$, see Section~\ref{S2.1}.
\end{theorem}
\begin{proof}
If $0\le s<\ha$, then $\s{H}^s(\Omega)=\{E_0 u,\ u\in{H}^s(\Omega)\}$,  which implies that one can take $\s{E}^s=E_0$, where the operator $E_0: H^s(\Omega)\to\s H^s(\Omega)$ of extension by zero is continuous by the Theorems \ref{McL3.40m-1} and \ref{H0=H} with the estimate \eqref{tnormest3} following from estimate \eqref{tnormest2}.

If $-\ha< s<0$, we define $\s{E}^s$ as
 $$
\langle\s{E}^s g,v\rangle_{\Omega}:=\langle g,{E}_0 v\rangle_{\Omega},\quad
\forall g\in{H}^s(\Omega),\ \forall v\in{H}^{-s}(\Omega),
 $$
 i.e., $\s{E}^s={E}_0 ^*=(\s{E}^{-s})^*$, which is continuous with the estimate \eqref{tnormest3} following from  the previous paragraph.

Theorem \ref{H_F=0} implies that the extension
operator $\s{E}^s : H^s(\Omega) \to \s{H}^s(\Omega)$ is unique for $-\ha<
s<\ha$.

Let now $-\tha< s<-\ha$. For $s$ in this range, the trace operator
$\gamma^{+}: H^{-s}(\Omega)\to H^{-s-\ha}(\pO)$ is bounded due to
\cite[Lemma 3.6]{Costabel1988} (see also \cite[Theorem
3.38]{McLean2000}), and there exists a bounded right inverse to the trace operator
$\gamma_{-1}: H^{-s-\ha}(\pO)\to H^{-s}(\Omega)$, see Lemma \ref{LipExt}.
Then  $(I-\gamma_{-1}\gamma^{+})$ is a bounded
projector from $H^{-s}(\Omega)$ to
${H}^{-s}_0(\Omega)=\s{H}^{-s}(\Omega)$ due to Theorem \ref{McL3.40m-1}. Thus any functional $v\in
H^{s}(\Omega)$ can be continuously mapped into the functional
$\tilde{v}\in \s{H}^{s}(\Omega)$ such that
$\langle\tilde{v},u\rangle=\langle v,E_0(I-\gamma_{-1}\gamma^{+})u\rangle$ for any $u\in H^{-s}(\Omega)$. Since
$\tilde{v}u=vu$ for any $u\in \s{H}^{-s}(\Omega)$, we have,
$$\s{E}^s:=[E_0(I-\gamma_{-1}\gamma^{+})]^* : H^s(\Omega) \to \s{H}^s(\Omega)$$
is a
bounded extension operator.
\end{proof}

Since the extension
operator $\s{E}^s : H^s(\Omega) \to \s{H}^s(\Omega)$ is unique for $-\ha<
s<\ha$, we will call it {\em canonical extension operator} (as opposite to other possible extensions from $H^s(\Omega)$ to $\s{H}^\sigma(\Omega)$, $\sigma<-\ha$). For $-\tha<
s<-\ha$, on the other hand, the operator $\gamma_{-1}:
H^{-s-\ha}(\pO)\to H^{-s}(\Omega)$ in the proof of Theorem
\ref{ExtOper}  is not unique, implying non-uniqueness of $\s{E}^s
: H^s(\Omega) \to \s{H}^s(\Omega)$.

We will later need the following two results.
\begin{lemma}\label{Hsto1}
Let $\Omega$ and $\Omega'\subset\Omega$ be open sets, and $s\le 0$. If $u\in H^s(\Omega)$, then $\|u\|_{H^s(\Omega')}\to 0$ as the Lebesgue measure of $\Omega'$ tends to zero.
\end{lemma}
\begin{proof}
Let $\phi\in \D(\ov\Omega)$. Then
$$
\|u\|_{H^s(\Omega')}\le \|u-\phi\|_{H^s(\Omega')}+\|\phi\|_{H^s(\Omega')}\le
\|u-\phi\|_{H^s(\Omega)}+\|\phi\|_{L_2(\Omega')}.
$$
For any $\epsilon>0$ we can chose $\phi$ such that $\|u-\phi\|_{H^s(\Omega)}<\epsilon/2$ due to the density of $\D(\ov\Omega)$ in $H^s(\Omega)$ and then chose $\Omega'$ with sufficiently small measure so that $\|\phi\|_{L_2(\Omega')}< \epsilon/2$.
\end{proof}


\begin{lemma}\label{Hsto3}
Let  $\Omega_k\subset\Omega$ be a sequence of Lipschitz domains converging to a Lipschitz domain $\Omega$ and $-\ha< s<1/2$.
If $u\in H^s(\Omega)$ and $\tilde u_k=\s E^s u|_{\Omega_k}$, then  there exists a constant $C$ independent of $u$ and $k$ such that $\|\tilde u_k\|_{\s H^s(\Omega_k)}\le C\|u\|_{H^s(\Omega)}$ for all sufficiently large $k$.
\end{lemma}
\begin{proof}
By Theorem~\ref{ExtOper},
$$\|\tilde u_k\|_{\s H^s(\Omega_k)}\le C_k\|u|_{\Omega_k}\|_{H^s(\Omega_k)}\le C_k\|u\|_{H^s(\Omega)},$$
where $C_k$ depend only on $s$ and on the maximum of the Lipschitz constants of the representation functions $\zeta_{jk}$ for the boundaries $\pO_k$. By \eqref{LipLim}, the Lipschitz constants are bounded and henceforth so are $C_k$.
\end{proof}

\section{Partial differential operator extensions and co-normal
derivatives for infinitely smooth coefficients}\label{S3sm}
Let us consider in ${\Omega}$ a system of $m$ complex linear differential equations of the second order with respect to $m$ unknown functions $\{u_{i}\}_{i=1}^m=u: \Omega\to \C^m$, which for sufficiently smooth $u$ has the following strong form,
\begin{equation}
\label{2.1}  Au(x):=
 -\sum\limits_{i,j=1}^n\pa_i[  a_{ij}(x)\,\pa_j u(x)]
 +\sum\limits_{j=1}^n b_{j}(x)\,\pa_j u(x)+c(x)u(x)=
 f(x),  \;\;\;\; x \in \Omega,
\end{equation}
where $f: \Omega\to\C^m$,  $\pa_j:=\pa/\pa{x_j}$
 $(j=1,2,...,n)$, $a(x)=\{a_{ij}(x)\}_{i,j=1}^n=\{\{a_{ij}^{kl}(x)\}_{k,l=1}^m\}_{i,j=1}^n$,  $b(x)=\{\{b_{i}^{kl}(x)\}_{k,l=1}^m\}_{i=1}^n$ and $c(x)=\{c^{kl}(x)\}_{k,l=1}^m$, i.e., $a_{ij}, b_{i},c: \Omega\to\C^{m\times m}$ for fixed indices $i,j$. If $m=1$, then \eqref{2.1} is a scalar equation. In this paper we assume that $a, b, c \in C^\infty(\ov{\Omega})$; the case of non-smooth coefficients is addressed in \cite{JCE-H}, see also \cite{MikhailovArXiv2009}.

The operator $A$ is (uniformly) strongly elliptic in an open domain $\Omega$ if there exists a bounded $m\times m$ matrix-valued function $\theta(x)$ such that
\bes\label{SEC}
\RE\{\bar\zeta^\top\theta(x)\sum\limits_{i,j=1}^n a_{ij}(x)\xi_i\xi_j\zeta\}\ge C|\xi|^2|\zeta|^2
\ees
for all $x\in\Omega$, $\xi\in\R^n$ and $\zeta\in \C^m$, where $C$ is a positive constant, see e.g. \cite[Definition 3.6.1]{Hsiao-Wendland2008} and references therein. We say that the operator $A$ is uniformly strongly elliptic in a closed domain $\bar\Omega$ if its is uniformly strongly elliptic in an open domain $\Omega'\supset \bar\Omega$. We will need the strong ellipticity in relation with the solution regularity, starting from Theorem~\ref{RegThSm}.

\subsection{Partial differential operator extensions and generalized co-normal
derivative}

For  $u\in H^s(\Omega)$, $f\in H^{s-2}(\Omega)$, $s\in\R$, equation system \eqref{2.1} is understood in the distribution sense as
\bes\label{L=fdist}
    \langle Au,v \rangle_\Omega=\langle f,v\rangle_\Omega\quad \forall v\in
    {\cal D}(\Omega),
\ees
where $v: \Omega\to \C^m$ and
\begin{equation}\label{Ldist}
    \langle Au,v \rangle_\Omega:=\E(u,v)\quad \forall v\in
    {\cal D}(\Omega),
\end{equation}
 \be\label{Edef}
\E(u,v)=\E_\Omega(u,v):=
  \sum\limits_{i,j=1}^{n} \left\langle a_{ij}\pa_j u , \pa_i v\right\rangle_{\Omega}
 +\sum\limits_{j=1}^{n} \left\langle b_{j}\pa_j u  ,  v\right\rangle_{\Omega}
 +\left\langle c u  ,  v\right\rangle_{\Omega}.
 \ee
Bilinear form \eqref{Edef} is well defined for any $v\in {\cal D}(\Omega)$ and moreover, the bilinear functional $\E:\{H^s(\Omega),\widetilde H^{2-s}(\Omega)\}\to\C$ is bounded for any $s\in\R$.
Since the set ${\cal D}(\Omega)$ is dense in $\s{H}^{2-s}
(\Omega)$, expression \eqref{Ldist} defines then a bounded linear operator $A:
H^s(\Omega)\to H^{s-2}(\Omega)=[\s{H}^{2-s} (\Omega)]^*$,
$s\in\R$,
\begin{equation}\label{LH1}
    \langle Au,v \rangle_\Omega:=\E(u,v)\quad \forall v\in
    \s{H}^{2-s}(\Omega).
\end{equation}

Let now $\ha<s<\tha$. In addition to the operator $A$ defined by \eqref{LH1}, let us consider also the {\em aggregate} partial differential operator
 $\check{A}$,
defined as,
\begin{equation}\label{Ltil}
    \langle \check{A}u,v \rangle_\Omega:=\check\E(u,v)\quad \forall v\in
    {H}^{2-s}(\Omega),
\end{equation}
where
\be\label{Echeckdef}
\check\E(u,v)=\check\E_\Omega(u,v):=
  \sum\limits_{i,j=1}^{n} \left\langle \s{E}^{s-1}(a_{ij}\pa_j u) , \pa_i v\right\rangle_{\Omega}
 +\sum\limits_{j=1}^{n} \left\langle \s{E}^{s-1}(b_{j}\pa_j u ) ,  v\right\rangle_{\Omega}
 +\left\langle \s{E}^{s-1}(c u)  ,  v\right\rangle_{\Omega}
 \ee
and $\s{E}^{s-1}: H^{s-1}(\Omega)\to \s H^{s-1}(\Omega)$ is a bounded extension operator, which is unique by Theorem~\ref{ExtOper}.
Note that by \eqref{dp3} one can rewrite \eqref{Ltil} also as
\begin{equation*}\label{LH1i}
    (\check Au,v )_{\Omega}:=\Phi(u,v)\quad \forall v\in
    {H}^{2-s}(\Omega),
\end{equation*}
where $\Phi(u,v)=\ov{\check\E(u,\bar v)}$ is the sesquilinear form.

If $s=1$, i.e. $u,v\in H^1(\Omega)$, evidently
\bes
\check\E(u,v)=\E(u,v)=
  \int_\Omega\left[\sum\limits_{i,j=1}^{n}  (a_{ij}\pa_j u) \cdot \pa_i v
 +\sum\limits_{j=1}^{n} (b_{j}\pa_j u ) \cdot  v
 + c u \cdot   v\right]\,dx .
 \ees

The aggregate operator $\check{A}: H^s(\Omega)\to\s{H}^{s-2}(\Omega)=[{H}^{2-s} (\Omega)]^*$ is bounded since $\pa_i v\in H^{1-s}(\Omega)$, $v\in H^{2-s}(\Omega)\subset H^{1-s}(\Omega)$. For any $u\in H^s(\Omega)$, the
functional $\check{A}u$ belongs to $\s{H}^{s-2}(\Omega)$ and is an
extension of the functional ${A}u\in {H}^{s-2}(\Omega)$ from the
domain of definition  $\s{H}^{2-s}(\Omega)\subset {H}^{2-s}(\Omega)$ to the domain of
definition ${H}^{2-s}(\Omega)$.

The distribution $\check{A}u$ is not the only possible extension of the functional ${A}u$, and any functional of the form
\begin{equation}\label{Ext}
     \check{A}u +g,\quad  g\in
    {H}^{s-2}_{\partial\Omega}
\end{equation}
gives another extension. On the other hand, any extension of the domain
of definition of the functional ${A}u$ from $\s{H}^{2-s}(\Omega)$ to
${H}^{2-s}(\Omega)$ has evidently form \eqref{Ext}. The existence of
such extensions is provided by Lemma \ref{ExtH}.

For $u\in H^s(\Omega)$, $s>\tha$,
the strong (classical) co--normal derivative operator
 \begin{equation}\label{clCD}
T^+_cu(x) := \sum_{i,j=1}^n
a_{ij}(x)\,\gamma^+[\pa_j u(x)]\nu_i(x)
\end{equation}
is well defined on $\pO$ in the sense of traces. Here $\gamma^+[\pa_j u]\in H^{s-\tha}(\pO)\subset L_2(\pO)$ if $\tha<s<\frac{5}{2}$, while the outward (to $\Omega$) unit normal vector $\nu(x)$ at
the point $x\in \pO$ belongs to $L_\infty(\pO)$ for the Lipschitz boundary $\pO$, implying  $T^+_cu\in L_2(\pO)$.
Note that for Lipschitz domains one can not generally expect that $T^+_c u$ belongs to $H^s(\pO)$, $s>0$, even for infinitely smooth $u$.

We can  extend the definition  of
the generalized co--normal derivative, given in \cite[Lemma 4.3]{McLean2000} for $s=1$ (cf. also \cite[Lemma 2.2]{Kohr-Pintea-Wendland2010} for the generalized co--normal derivative on a manifold boundary), to a range of Sobolev
spaces as follows.

\begin{definition}\label{GCDd}
Let $\Omega$ be a Lipschitz domain, $\ha< s<\tha$,  $u\in H^s(\Omega)$, and $Au=\tilde{f}|_{_\Omega}$ in
$\Omega$ for some $\tilde{f}\in \s{H}^{s-2}(\Omega)$. Let us
define {\em the generalized co--normal derivative\ }
$T^+(\tilde{f},u) \in H^{s-\tha}(\pO)$  as
\begin{equation}
\label{Tgendef} \left\langle
 T^+(\tilde{f},u)\,,\, w
\right\rangle _{\pO}:=
\check\E(u,\gamma_{-1} w)-\langle \tilde{f},\gamma_{-1} w \rangle_\Omega
=\langle \check{A}u - \tilde{f},\gamma_{-1} w \rangle_\Omega
  \quad  \forall\ w\in H^{\tha-s} (\pO),
\end{equation}
where $\gamma_{-1} : H^{\tha-s}(\pO)\to H^{2-s}(\Omega)$ is a bounded
right inverse to the trace operator.
\end{definition}
The notation $T^+(\tilde{f},u)$ corresponds to the notation $\s T^+(\tilde{f},u)$ in \cite{MikhailovIMSE06}.

\begin{theorem}\label{GCDl}
Under the hypotheses of Definition \ref{GCDd}, the generalized
co--normal derivative
$T^+(\tilde{f},u)$ is independent of the operator $\gamma_{-1}$, the estimate
\begin{equation}\label{estimate}
\|T^+(\tilde{f},u)\|_{H^{s-\tha}(\pO)}\le
C_1\|u\|_{H^s(\Omega)} + C_2\|\tilde{f}\|_{\s{H}^{s-2}(\Omega)}
\end{equation}
takes place, and the first Green identity holds in the following
form,
\begin{equation}
\label{Tgen} \left\langle
 T^+(\tilde{f},u)\,,\, \gamma^+v
\right\rangle _{\pO}
=\check\E(u,v)-\langle \tilde{f},v \rangle_\Omega
=\langle\check{A}u - \tilde{f},v \rangle_\Omega
  \quad  
\forall\ v\in H^{2-s}(\Omega).
\end{equation}
\end{theorem}
\begin{proof}
For $s=1$ the theorem proof is available in \cite[Lemma 4.3]{McLean2000}, which idea is extended here to the whole range $\ha< s<\tha$.

By Lemma \ref{LipExt},  a bounded  operator $\gamma_{-1} :
H^{\tha-s}(\pO)\to H^{2-s}(\Omega)$ does exist. Then estimate
\eqref{estimate} follows from \eqref{Tgendef}.

To prove independence of the co-normal derivative
$T^+(\tilde{f},u)$ of $\gamma_{-1} $, let us consider two co-normal
derivatives generated by two different operators $\gamma_{-1}^\prime$ and $\gamma_{-1}^{\prime\prime}$. Then their difference is
 $$
\langle T^{\prime+}(\tilde{f},u)-T^{\prime\prime+}(\tilde{f},u),w
\rangle_\pO=\langle \check{A}u - \tilde{f},\gamma_{-1}^\prime w-\gamma_{-1}^{\prime\prime}w \rangle_\Omega
\quad  \forall\ w\in H^{\tha-s} (\pO).
 $$
By definition, $\check{A}u - \tilde{f}\in {H}_\pO^{s-2}$, which by Corollary \ref{C2.9} implies that
 \begin{multline*}
\langle \check{A}u - \tilde{f},\gamma_{-1}^\prime w-\gamma_{-1}^{\prime\prime}w \rangle_\Omega =
\langle \check{A}u - \tilde{f},\gamma_{-1}^\prime w-\gamma_{-1}^{\prime\prime}w \rangle_{\R^n} =
\langle \gamma^*\gamma^*_{-1}(\check{A}u - \tilde{f}),\gamma_{-1}^\prime w-\gamma_{-1}^{\prime\prime}w \rangle_{\R^n} =\\
\langle \gamma^*_{-1}(\check{A}u - \tilde{f}),
\gamma \gamma_{-1}^\prime w-\gamma\gamma_{-1}^{\prime\prime}w\rangle_\pO=
\langle \gamma^*_{-1}(\check{A}u - \tilde{f}),
w-w\rangle_\pO=0\quad \forall\ w\in H^{\tha-s} (\pO).
\end{multline*}

To prove \eqref{Tgen}, let $V\in H^{2-s}(\R^n)$ be such that $v=V|_\Omega$ implying $\gamma^+v=\gamma V$. Taking again into account that $\check{A}u - \tilde{f}\in H^{s-2}_\pO$,  we have by Corollary \ref{C2.9},
\begin{multline*}
\left\langle
 T^+(\tilde{f},u)\,,\, \gamma^+v \right\rangle _{\pO}
=\langle\check{A}u - \tilde{f},\gamma_{-1} \gamma^+ v \rangle_\Omega
=\langle\check{A}u - \tilde{f},\gamma_{-1} \gamma V \rangle_{\R^n}\\
=\langle\gamma^*\gamma^*_{-1}(\check{A}u - \tilde{f}),V \rangle_{\R^n}
=\langle\check{A}u - \tilde{f},V \rangle_{\R^n}=
\langle\check{A}u - \tilde{f},v \rangle_\Omega
\end{multline*}
as required.
\end{proof}

Because of the involvement of $\tilde{f}$, the
generalized co-normal derivative $T^+(\tilde{f},u)$  is
generally {\em non-linear} in $u$. It becomes linear if a linear
relation is imposed between $u$ and $\tilde{f}$ (including
behavior of the latter on the boundary $\partial\Omega$), thus
fixing an extension of
$\tilde{f}|_{_\Omega}$ into
$\s{H}^{s-2}(\Omega)$. For example, $\tilde{f}|_{_\Omega}$ can be
extended as
 $\check{f}:=\check{A}u,$
which generally does not coincide with $\tilde{f}$. Then obviously,
$T^{+}(\check{f},u)=T^{+}(\check{A}u,u)=0$, meaning that the co-normal derivatives associated with any other possible extension $\tilde{f}$ appears to be aggregated in $\check{f}$ as
 \be
\label{ftot}
\langle \check{f},v \rangle_\Omega=\langle \tilde{f},v \rangle_\Omega+ \left\langle T^+(\tilde{f},u)\,,\, \gamma^+v \right\rangle _{\pO}
 \ee
due to \eqref{Tgen}. This justifies the term {\em aggregate} for the extension $\check{f}$, and thus for the operator $\check{A}u$.

As follows from Definition~\ref{GCDd}, the generalized co-normal derivative is still linear with respect to the couple $(\tilde f, u)$, i.e.,
$$
T^+(\alpha_1\tilde{f}_1,\alpha_1u_1) + T^+(\alpha_2\tilde{f}_2,\alpha_2u_2)=
T^+(\alpha_1\tilde{f}_1+\alpha_2\tilde{f}_2,\alpha_1u_1+\alpha_2u_2)
$$
for any complex numbers $\alpha_1,\alpha_2$.

In fact, for a given function $u\in H^s(\Omega)$, $\ha<s<\tha$, any
distribution $\tau\in H^{s-\tha}(\pO)$ may be nominated as a co-normal
derivative of $u$, by an appropriate extension $\tilde{f}$ of the
distribution $Au\in {H}^{s-2}(\Omega)$ into $\s{H}^{s-2}(\Omega)$. This extension is
again given by the second Green formula \eqref{Tgen} re-written as
follows (cf. \cite[Section 2.2, item 4]{Agranovich2003RMS} for $s=1$),
\begin{equation}
\label{Luext}
 \langle \tilde{f},v \rangle_\Omega:=
 \check\E(u,v)-\left\langle \tau ,  \gamma^{+}v \right\rangle _{\pO}=
\langle \check{A}u-\gamma^{+*}\tau,v \rangle_\Omega
  \quad  
\forall\ v\in H^{2-s} (\Omega).
\end{equation}
Here the operator $\gamma^{+*} : H^{s-\tha}(\pO) \to
\s{H}^{s-2}(\Omega)$ is adjoined to the trace operator, $\langle
\gamma^{+*}\tau, v\rangle_\Omega :=\left\langle \tau , \gamma^{+}v
\right\rangle _{\pO}$ for all $\tau\in H^{s-\tha}(\pO)$ and $v\in
H^{2-s} (\Omega)$. Evidently, the distribution $\tilde{f}$ defined by
\eqref{Luext} belongs to $\s{H}^{s-2}(\Omega)$ and is an extension of
the distribution $Au$ into $\s{H}^{s-2}(\Omega)$ since $\gamma^{+}v=0$
for $v\in \s{H}^{2-s}(\Omega)$.

For $u\in C^1(\ov{\Omega})\subset H^1(\Omega)$,  one can take $\tau$ equal to the strong co-normal derivative, $T^+_c u \in L_\infty(\pO)$, and relation \eqref{Luext} can be considered as the {\em classical extension} of $f=Au\in H^{-1}(\Omega)$ to $\tilde f_c\in \s H^{-1}(\Omega)$, which is evidently linear.

\subsection{Boundary value problems}\label{NewPVPs}
Consider the BVP weak settings  for PDE system \eqref{2.1} on Lipschitz domain for $\ha<s<\tha$.

{\em The Dirichlet problem:} for $f\in H^{s-2}(\Omega)$ and $\varphi_0\in H^{s-\ha}(\pO)$, find $u\in H^s(\Omega)$ such that
\begin{eqnarray}\label{wD}
    \langle Au,v \rangle_\Omega&=&\langle f,v \rangle_\Omega\quad \forall v\in
    \s{H}^{2-s}(\Omega),\\
    \gamma^+u&=&\varphi_0\quad\mbox{on}\ \pO.\label{bD}
\end{eqnarray}

{\em The Neumann problem:} for $\check f\in \s{H}^{s-2}(\Omega)$, find $u\in H^s(\Omega)$ such that
\begin{eqnarray}\label{wN}
    \langle \check Au,v \rangle_\Omega&=&\langle \check f,v \rangle_\Omega\quad \forall v\in
    {H}^{2-s}(\Omega).
\end{eqnarray}
Here $Au$ and $\check Au$ are defined by \eqref{LH1} and \eqref{Ltil}, respectively.

To set the mixed problem, let   $\partial_D\Omega $ and $\partial_N\Omega =\pO\backslash \overline{\partial_D\Omega}$ be  nonempty, open
sub--manifolds of $\pO$, and ${H}^{s}_{0} (\Omega,\partial_D\Omega)=\{w\in {H}^{s}(\Omega): \gamma^+ w=0$ on $\partial_D\Omega\}$. We introduce the {\em mixed aggregate} operator $\check{A}_{\partial_D\Omega }: H^s(\Omega) \to
[{H}^{2-s}_{0} (\Omega,{\partial_D\Omega })]^*$, defined
as
 $$
    \langle \check{A}_{\partial_D\Omega }u,v \rangle_\Omega:=
    \langle \check{A}u,v \rangle_\Omega=
\check\E(u,v)\quad \forall\ v\in {H}^{2-s}_{0} (\Omega,\partial_D\Omega ).
 $$
The mixed operator $\check{A}_{\partial_D\Omega }$ is bounded by the same argument as the aggregate operator $\check{A}$. For any $u\in H^s(\Omega)$, the
distribution $\check{A}_{\partial_D\Omega }u$ belongs to $[{H}^{2-s}_{0} (\Omega,{\partial_D\Omega })]^*$ and is an
extension of the functional ${A}u\in {H}^{s-2}(\Omega)$ from the
domain of definition  $\s{H}^{2-s}(\Omega)={H}^{2-s}_0(\Omega)\subset {H}^{2-s}_{0} (\Omega,\partial_D\Omega)$ to the domain of
definition ${H}^{2-s}_{0} (\Omega,\partial_D\Omega )$, and a restriction of the functional $\check{A}u\in \s{H}^{s-2}(\Omega)$ from the
domain of definition  ${H}^{2-s}(\Omega)\supset {H}^{2-s}_{0} (\Omega,\partial_D\Omega)$ to the domain of
definition ${H}^{2-s}_{0} (\Omega,\partial_D\Omega)$.

For $v\in {H}^{2-s}_{0} (\Omega,\partial_D\Omega )$, the trace $\gamma^+v$ belongs to $\s H^{\tha-s}(\partial_N\Omega )$. If $Au=\tilde{f}|_{_\Omega}$ in
$\Omega$ for some $\tilde{f}\in \s{H}^{s-2}(\Omega)$, then the first Green identity \eqref{Tgen} gives,
\begin{eqnarray}
   && \langle \check{A}_{\partial_D\Omega }u,v \rangle_\Omega=
   \langle \check f_m,v \rangle_\Omega,\nonumber\\
 \label{ftotM}
&&\langle \check{f}_m,v \rangle_\Omega=\langle \tilde{f},v \rangle_\Omega+ \left\langle T^+(\tilde{f},u)\,,\, \gamma^+v \right\rangle _{\partial_N\Omega }
\quad \forall\ v\in {H}^{2-s}_{0} (\Omega,\partial_D\Omega ),
\end{eqnarray}
where, evidently, $\check{f}_m\in [{H}^{2-s}_{0} (\Omega,{\partial_D\Omega })]^*$. This leads to the following weak setting.

{\em The mixed (Dirichlet-Neumann) problem:} for $\check f_m\in [{H}^{2-s}_{0} (\Omega,{\partial_D\Omega })]^*$ and $\varphi_0\in H^{s-\ha}(\partial_D\Omega)$, find $u\in H^s(\Omega)$ such that
\begin{eqnarray}\label{wM}
    \langle \check A_{\partial_D\Omega }u,v \rangle_\Omega&=&\langle \check f_m,v \rangle_\Omega\quad \forall v\in
    {H}^{2-s}_{0} (\Omega,\partial_D\Omega ),\\
    \gamma^+u&=&\varphi_0\quad\mbox{on}\ \partial_D\Omega.\label{wM2}
\end{eqnarray}

The Neumann and the mixed problems are formulated in terms of the aggregate right hand sides  $\check f$ and $\check f_m$, respectively, prescribed on their own, i.e., without necessary splitting them into the right hand side inside the domain $\Omega$ and the part related with the prescribed co-normal derivative. If a right hand side extension $\tilde f$ and an associated non-zero generalized co-normal derivative $T^+(\tilde{f},u)$ are prescribed instead, then $\check f$ and $\check f_m$ can be expressed through them by relations \eqref{ftot}, \eqref{ftotM}. Thus the co-normal derivative does not enter, in fact,  the weak settings of the Dirichlet, Neumann or mixed problem, implying that the non-uniqueness of  $T^+(\tilde{f},u)$ for a given function $u\in H^s(\Omega)$, $\ha<s<\tha$, does not influence the BVP weak settings, (cf. \cite[Section 2.2, item 4]{Agranovich2003RMS} for $s=1$).
On the other hand, for a given $u\in H^s(\Omega)$ the aggregate right hand sides $\check f$ and $\check f_m $ are uniquely determined by  \eqref{wN}, \eqref{wM}, as are, of course, $f$ and $\varphi_0$ by \eqref{wD},  \eqref{bD}/\eqref{wM2}.

Note that one can take $v=\bar w$ to make the settings \eqref{wD}-\eqref{bD}, \eqref{wN} and \eqref{wM}-\eqref{wM2} look closer to the usual variational formulations, cf. e.g. \cite{LiMa1}.

\subsection{Canonical co-normal derivative}

As we have seen above, for an arbitrary $u\in H^s(\Omega)$, $\ha<s<\tha$, the co-normal derivative $T^+(\tilde{f},u)$ is generally  non-uniquely determined by $u$. An exception is $T^{+}(\check{A}u,u)\equiv 0$ but such co-normal derivative evidently differs from the strong co-normal derivative ${T}^{+}_c u$, given by \eqref{clCD} for sufficiently smooth $u$.
Another one way of making generalized co-normal derivative unique in $u\in H^1(\Omega)$ was presented in \cite[Lemma 5.1.1]{Hsiao-Wendland2008} and is in fact associated with an extension of
$Au\in H^{-1}(\Omega)$ to $\tilde f\in\widetilde H^{-1}(\Omega)$, such that $\tilde f$ is orthogonal in $H^{-1}(\R^n)$ to $H^{-1}_\pO\subset H^{-1}(\R^n)$. However it appears (see Lemma \ref{contrHW}), that even for infinitely smooth functions $f$ such extension $\tilde f$ does not generally belong to $L_2(\R^n)$, which implies that  the so-defined co-normal derivative operator $\tau$ from \cite[Lemma 5.1.1]{Hsiao-Wendland2008} is not  a bounded extension of the strong co-normal derivative operator.

Nevertheless, it is still possible to point out some subspaces of $H^s(\Omega)$, $\ha<s<\tha$, where a unique definition of the co-normal derivative by $u$ is possible and leads to the strong co-normal derivative for sufficiently smooth $u$. We define below one such sufficiently wide subspace.


\begin{definition}\label{Hst}
Let $s\in\R$ and $A_*:H^s(\Omega)\to {\cal D}^*(\Omega)$ be a
linear operator. For $t\ge -\ha$, we introduce a space
 $
 H^{s,t} (\Omega;A_*):=\{g:\;g\in H^s (\Omega),\ A_*g|_\Omega=\tilde{f}_g|_\Omega,\
\tilde{f}_g\in
\s{H}^{t}(\Omega)\}
 $
equipped with the graphic norm,
 $ \|g\|_{ H^{s,t} (\Omega;A_*)}^2:= \|g\|_{ H^s (\Omega)}^2+ \|\tilde{f}_g\|_{\s{H}^{t}(\Omega)}^2
 $.
 \end{definition}

The distribution  $\tilde{f}_g\in \s{H}^{t}(\Omega)$, $t\ge -\ha$,
in the above definition is an extension of the distribution
$A_*g|_\Omega\in {H}^{t}(\Omega)$, and the extension  is unique
(if it does exist), since otherwise the difference between any two
extensions  belongs to ${H}^{t}_\pO$ but ${H}^{t}_\pO=\{0\}$ for
$t\ge -\ha$ due to the Theorem \ref{H_F=0}.  The uniqueness
implies that the norm $ \|g\|_{ H^{s,t} (\Omega;A_*)}$ is well defined.
Note that another subspace of such kind, where $A_*g|_\Omega$
belongs to $L_p(\Omega)$ instead of $H^{t}(\Omega)$, was presented
in \cite[p. 59]{Grisvard1985}. A particular case, $H^{s,0} (\Omega;A_*)$, was extensively employed in \cite{Costabel1988}.

If $s_1\le s_2$ and $t_1\le t_2$, then we have the
embedding, $ H^{s_2,t_2} (\Omega;A_*)\subset H^{s_1,t_1}
(\Omega;A_*)$.

\begin{remark}\label{R3.4} If $s\in\R$, $-\ha<t<\ha$, and $A_*:H^s(\Omega)\to {H}^{t}(\Omega)$ is a linear continuous operator, then
$H^{s,t} (\Omega;A_*)=H^s(\Omega)$ by Theorem~\ref{ExtOper}.
\end{remark}

\begin{lemma} Let $s\in\R$.
If a linear operator $A_*: H^s(\Omega)\to \D^*(\Omega)$ is
continuous,
then the space $H^{s,t} (\Omega;A_*)$ is complete for any $t\ge
-\ha$.
\end{lemma}
\begin{proof}
Let $\{g_k\}$ be a Cauchy sequence in $H^{s,t} (\Omega;A_*)$. Then
there exists a Cauchy sequence $\{\tilde{f}_{k}\}$ in
$\s{H}^{t}(\Omega)$ such that
$\tilde{f}_{k}|_\Omega=A_*g_k|_\Omega$. Since $H^s(\Omega)$ and
$\s{H}^{t}(\Omega)$ are complete,  there exist elements
$g_0\in H^s(\Omega)$ and $\tilde{f}_0\in \s{H}^{t}(\Omega)$ such
that $\|g_k-g_0\|_{H^s(\Omega)}\to 0$,
$\|\tilde{f}_{k}-\tilde{f}_0\|_{\s{H}^{t}(\Omega)}\to 0$ as
$k\to\infty$. On the other hand, continuity of $A_*$ implies that
$|\langle A_*(g_k-g_0),\phi\rangle|\to 0$ for any $\phi\in \D(\Omega)$.
Taking into account that $A_*g_k|_\Omega=\tilde{f}_{k}|_\Omega$,
we obtain
\begin{eqnarray*}
|\langle \tilde{f}_0-A_*g_0,\phi\rangle|&\le&
|\langle\tilde{f}_0-\tilde{f}_{k},\phi\rangle |+|\langle \tilde{f}_k-A_*g_0,\phi\rangle|\\
&\le&\|\tilde{f}_0-\tilde{f}_{k}\|_{\s H^t(\Omega)}\|\phi\|_{H^{-t}(\Omega)}+
|\langle A_*(g_k-g_0),\phi\rangle|\to
0,\quad  k\to\infty \quad \forall \phi\in\D(\Omega),
\end{eqnarray*}
 i.e.,
$A_*g_0|_\Omega=\tilde{f}_0|_\Omega\in H^t(\Omega)$, which implies $A_*g_0$ is extendable to $\tilde{f}_0\in \s H^t(\Omega)$ and thus $g_0\in H^{s,t} (\Omega;A_*)$.
\end{proof}

We will further use the space $H^{s,t} (\Omega;A_*)$ for the case when the operator $A_*$ is the operator $A$ from \eqref{Ldist} or the operator $A^*$ formally adjoined to it (see Section~\ref{2ndGI}).

\begin{definition}\label{Dce} Let $s\in\R$, $t\ge -\ha$. The operator $\tilde A$ mapping functions $u\in H^{s,t}(\Omega;A)$ to the extension of the distribution $Au\in {H}^t(\Omega)$ to $\s{H}^t(\Omega)$ will be called {\em the canonical extension} of the operator $A$.
\end{definition}

\begin{remark}
If $s\in\R$, $t\ge -\ha$, then  $\|\tilde Au\|_{\s{H}^t(\Omega)}\le \|u\|_{H^{s,t}(\Omega;A)}$ by definition of the space $H^{s,t}(\Omega;A)$, i.e., the linear operator
 $
\tilde A: H^{s,t}(\Omega;A)\to \s{H}^t(\Omega)
 $
is continuous.
Moreover, if $\ -\ha<t<\ha$, then by Theorem~\ref{ExtOper} and uniqueness of the extension of ${H}^t(\Omega)$ to $\s{H}^t(\Omega)$, we have the representation $\tilde A:=\s E^tA$.
\end{remark}

As in \cite[Definition 3]{MikhailovIMSE06} for scalar PDE, let us define the {\em canonical} co-normal
derivative operator. This extends  \cite[Theorem 1.5.3.10]{Grisvard1985} and \cite[Lemma 3.2]{Costabel1988} where
co-normal derivative operators acting on functions from $H^{1,0}_p(\Omega;\Delta)$ and
$H^{1,0}(\Omega;A)$, respectively, were defined.

\begin{definition}\label{Dccd}
For $u\in H^{s,-\ha}(\Omega;A)$, $\ha<s<\tha$,   we define the {\em canonical
co-normal derivative} as $T^+u:= T^+(\tilde A u,u)\in H^{s-\tha}(\pO)$, i.e.,
\bes
\label{Tcandef}
 \left\langle  T^+u\,,\, w\right\rangle _{\pO}
:= \check\E(u,\gamma_{-1} w)-\langle \tilde A u,\gamma_{-1} w \rangle_\Omega
=\langle\check{A}u- \tilde A u,\gamma_{-1} w \rangle_\Omega
  \quad  
\forall\ w\in H^{\tha-s} (\pO),
\ees
where $\gamma_{-1} : H^{s-\ha}(\pO)\to H^{s}(\Omega)$ is a bounded right inverse to the trace
operator.
\end{definition}

Theorem \ref{GCDl} for the generalized co-normal derivative and
Definition \ref{Hst} imply the following statement.
\begin{theorem}\label{GCDlc}
Under hypotheses of Definition \ref{Dccd},  the canonical
co-normal derivative $T^+u$ is independent of the
operator $\gamma_{-1}$, the operator $T^+ : H^{s,-\ha}(\Omega;A)\to
H^{s-\tha}(\pO)$ is continuous, and the first Green identity holds
in the following form,
\begin{eqnarray*} \label{Tcan}
 \left\langle T^{+}u\,,\, \gamma^+v\right\rangle _{_\pO}
 =\left\langle T^+(\tilde A u,u)\,,\, \gamma^+v\right\rangle _{_\pO}
 =\check\E(u,v)-\langle \tilde A u,v \rangle_\Omega\nonumber\\
 =\langle \check{A}u-\tilde A u,v \rangle_\Omega
  \quad  
\forall\ v\in H^{2-s} (\Omega) .
\end{eqnarray*}
\end{theorem}
Thus unlike the generalized co-normal derivative, the canonical
co-normal derivative is uniquely defined by the function $u$ and
the operator $A$ only, uniquely fixing an extension of the latter on
the boundary.

Definitions \ref{GCDd} and \ref{Dccd} imply that the generalized
co-normal derivative of $u\in H^{s,-\ha}(\Omega;A)$, $\ha<s<\tha$,
for any other extension $\tilde{f}\in \s{H}^{s-2}(\Omega)$ of the
distribution $Au|_\Omega\in {H}^{-\ha}(\Omega)$ can be expressed as
\begin{equation*}
\label{Tgentil} \left\langle
 T^+(\tilde{f},u)\,,\, w
\right\rangle _{_\pO}=\left\langle T^{+}u\,,\, w \right\rangle
_{_\pO}+\langle \tilde A u - \tilde{f},\gamma_{-1} w \rangle_\Omega
  \quad  
\forall\ w\in H^{\tha-s} (\pO) .
\end{equation*}

Note that the distributions $\check{A}u-\tilde{f}$, $\check{A}u-\tilde A u$
and $\tilde A-\tilde{f}$ belong to ${H}^{2-s}_{\pO}$ since $\tilde Au$,
$\check{A}u$, $\tilde{f}$ belong to $\s{H}^{2-s}(\Omega)$, while
$\tilde Au|_\Omega=\check{A}u|_\Omega=\tilde{f}|_\Omega=Au|_\Omega\in
{H}^{s-2}(\Omega)$.

Since by Theorem~\ref{GCDlc} the canonical co-normal derivative does not depend on the  extension operator $\gamma_{-1}$, the latter can be always chosen such that $\gamma_{-1}w$ has a support only near the boundary, which means that the co-normal derivative $T^+u$ is determined by the behavior of $u$ near the boundary. We can formalize this in the following statement.
\begin{theorem}\label{GCDlc1}
Let $\Omega$ and $\Omega'\subset\Omega$ be bounded or unbounded open Lipschitz domains, $\pO\subset\pO'$, $u\in H^{s,-\ha}(\Omega;A)$, $u\in H^{s,-\ha}(\Omega';A)$, $\ha<s<\tha$, while $T^+u$ and $T^{\prime+}u$ be the canonical
co-normal derivatives on $\pO$ and $\pO'$ respectively. Then $T^+u=r_{_\pO}T^{\prime+}u$.
\end{theorem}
\begin{proof}
By  the definition of the restriction operator $r_{_\pO}$ and Definition \ref{Dccd} we have,
\bes
 \left\langle  T^{\prime+}u\,,\, w\right\rangle _{\pO'}
:= \check\E_{\Omega'}(u,\gamma'_{-1} w)-\langle \tilde A_{\Omega'} u,\gamma'_{-1} w \rangle_{\Omega'}
  \quad  
\forall\ w\in H^{\tha-s} (\pO')\ : r_{_{\pO'\backslash\pO}}w=0 ,
\ees
where $\gamma'_{-1} : H^{s-\ha}(\pO')\to H^{s}(\Omega')$ is a bounded right inverse to the trace operator. Since $\gamma\gamma'_{-1} w =0$ on $\pO'\backslash\pO$, we can extend $\gamma'_{-1} w$ by zero on $\Omega\backslash\Omega'$ to $\gamma_{-1} w$. The operator $\gamma_{-1} : H^{s-\ha}(\pO)\to H^{s}(\Omega)$ is continuous, and we arrive at
\begin{multline*}
\label{Tcandef}
 \left\langle  T^{\prime+}u\,,\, w\right\rangle _{\pO}
=\check\E_\Omega(u,\gamma_{-1} w)-\langle \tilde A_{\Omega'} u,\gamma_{-1} w \rangle_\Omega
=\check\E_\Omega(u,\gamma_{-1} w)-\langle \tilde A_{\Omega} u,\gamma_{-1} w \rangle_\Omega
= \left\langle  T^+u\,,\, w\right\rangle _{\pO}
  \\
\forall\ w\in H^{\tha-s} (\pO),
\end{multline*}
\end{proof}

Theorem~\ref{GCDlc1} can be considered as an alternative definition of the canonical co-normal derivative, where the domain $\Omega'$ can be chosen arbitrarily small, and particularly can be take bounded when $\Omega$ is unbounded (with compact boundary). Note that similar reasoning holds also for the generalized co-normal derivative.

To give conditions when the canonical co-normal derivative $T^+ u$ coincides with the strong
co-normal derivative $T^+_c u$,
if the latter does exist in the trace sense, we prove in Lemma~\ref{densL} below that $\D(\ov{\Omega})$ is dense in $H^{s,t}(\Omega;A)$. The proof is based on the following local regularity theorem well known for the case of infinitely smooth coefficients, see e.g. \cite{Schwartz1958, Agmon1965, LiMa1}.
\begin{theorem}\label{RegThSm} Let $\Omega$ be an open set in $\R^n$, $s_1\in\R$, function $u\in H^{s_1}_{loc}(\Omega)^m$, $m\ge 1$,
satisfy strongly elliptic system \eqref{2.1} in $\Omega$ with $f\in H_{loc}^{s_2}(\Omega)^m$, $s_2>s_1-2$, and infinitely smooth coefficients. Then $u\in H^{s_2+2}_{loc}(\Omega)^m$.
\end{theorem}

Now we are in the position to prove the density theorem
\begin{theorem}\label{densL}
If $\Omega$ is a bounded Lipschitz domain, $s\in\R$,   $-\ha\le t<\ha$ and the operator $A$ is strongly elliptic on $\ov{\Omega}$,  then
$\D(\ov{\Omega})$ is dense in $H^{s,t}(\Omega;A)$.
\end{theorem}
\begin{proof}
We modify appropriately the proof from \cite[Lemma
1.5.3.9]{Grisvard1985} given for another space of such kind associated with the Laplace operator.

For every continuous linear functional $l$ on $H^{s,t}(\Omega;A)$
there exist distributions $\tilde{h}\in \s{H}^{-s}(\Omega)$ and
$g\in H^{-t}(\Omega)$ such that
 $$
l(u)=\langle \tilde{h},u\rangle_\Omega +\langle g,
\tilde Au\rangle_\Omega.
 $$

To prove the lemma claim, it suffice to show that any $l$, which
vanishes on $\D(\ov{\Omega})$, will vanish on any $u\in
H^{s,t}(\Omega;A)$. Indeed, if $l(\phi)=0$ for any
$\phi\in\D(\ov{\Omega})$, then
 \begin{equation}\label{j=0}
\langle \tilde{h},\phi\rangle_\Omega +\langle g,\tilde A\phi\rangle_\Omega=0.
 \end{equation}
Let us consider the case $-\ha<
t<\ha$ first and extend $g$ outside ${\Omega}$ to
$\tilde{g}=\s E^{-t}g\in \s{H}^{-t}(\Omega)$. Equation \eqref{j=0} gives by Theorem~\ref{ExtOper},
 \begin{multline*}
\langle \tilde{h},\phi\rangle_{\Omega'} +\langle \tilde{g},A\phi\rangle_{\Omega'}=
\langle \tilde{h},\phi\rangle_{\Omega} +\langle \tilde{g},A\phi\rangle_{\Omega}=
\langle \tilde{h},\phi\rangle_{\Omega} +\langle\s E^{-t}g,A\phi\rangle_{\Omega}=\\
\langle \tilde{h},\phi\rangle_{\Omega} +\langle {g},\s E^tA\phi\rangle_{\Omega}=
\langle \tilde{h},\phi\rangle_{\Omega} +\langle {g},\tilde A\phi\rangle_{\Omega}=0
 \end{multline*}
for any $\phi\in\D(\Omega')$ on some domain $\Omega'\supset \ov{\Omega}$, where the operator $A$ is still strongly elliptic. This means
 \be\label{A*g} A^*\tilde{g}=-\tilde{h} \quad \mbox{in } \Omega'\ee
in the sense of distributions, where $A^*$ is the operator formally adjoint to $A$.
If $t\le s-2$, then evidently $\tilde{g}\in \s{H}^{2-s}(\Omega)$. If $t> s-2$, then \eqref{A*g} and Theorem~\ref{RegThSm} imply
$\tilde{g}\in H^{2-s}_{loc}(\Omega')$ and consequently
$\tilde{g}\in \s{H}^{2-s}(\Omega)$.

In the case $t=-\ha$,
one can extend $g\in H^\ha(\Omega)$ outside $\ov{\Omega}$ by zero to $\tilde{g}\in
\s{H}^{\ha-\epsilon}(\Omega)$, $0<\epsilon$, and prove as
in the previous paragraph that $\tilde{g}\in \s{H}^{2-s}(\Omega)$.

If $-\ha<t<\ha$ or [$t=-\ha$, $s\le \tha$]
then for any $u\in H^{s,t}(\Omega;A)$, we have,
 \bes
l(u)=\langle -A^*\tilde g,u\rangle_\Omega +\langle g ,\tilde Au\rangle_\Omega =-\langle\tilde g,Au\rangle_\Omega
+\langle\tilde g,Au\rangle_\Omega =0.
 \ees
Thus $l$ is identically zero.

On the other hand, if $t=-\ha$, $s> \tha$, let $\{\tilde g_k\}\in\D(\Omega)$ be a sequence converging, as $k\to\infty$, to $g$
in $H^{\ha}_0(\Omega)=H^{\ha}(\Omega)$, cf. Theorem~\ref{H0=H}, and thus to $\tilde g$ in $\s{H}^{2-s}(\Omega)$. Then for any $u\in H^{s,\ha}(\Omega;A)$, we have,
 \begin{multline*}
l(u)=\langle -A^*\tilde g,u\rangle_\Omega +\langle g ,\tilde Au\rangle_\Omega =\lim_{k\to\infty}\left\{\langle -A^*\tilde g_k,u\rangle_\Omega
+\langle \tilde g_k,\tilde Au\rangle_\Omega\right\} \\
=\lim_{k\to\infty}\left\{-\langle
\tilde g_k,Au\rangle_\Omega +\langle \tilde g_k, Au\rangle_\Omega\right\}=0,
 \end{multline*}
which completes the proof.

\end{proof}

\begin{lemma}\label{CCCDLD}
Let $u\in
H^{s,-\ha}(\Omega;A)$, $\ha<s<\tha$, and $\{u_k\}\in \D(\ov{\Omega})$
be a sequence such that
 \be\label{u-uk}
 \|u_k-u\|_{H^{s,-\ha}(\Omega;A)}\to 0\quad \text{as}\
k\to \infty.
 \ee
Then
 $\|T^+_c u_k-T^+ u\|_{H^{s-\tha}(\pO)}\to 0$ as $k\to \infty$.
\end{lemma}
\begin{proof}
Using the definition of $T^{+}u$ and the
classical first Green identity for $u_k$, we have for any $w\in
H^{\tha-s}(\pO)$,
\begin{multline*}
\left|\left\langle T^{+} u -T^+_c u_k, w\right\rangle _{_\pO}\right|
 =\left|\check\E(u-u_k,\gamma_{-1} w)- \langle \tilde A (u-u_k), \gamma_{-1} w\rangle_\Omega\right|\le
 \\
 C\|u-u_k\|_{H^{s,-\ha}(\Omega;A)}\|w\|_{H^{\tha-s}(\pO)}.
  \end{multline*}
This implies
$$\|T^+_c u_k-T^+ u\|_{H^{s-\tha}(\pO)}\le \|u-u_k\|_{H^{s,-\ha}(\Omega;A)}\to 0\quad \mbox{as}\quad\mbox k\to \infty.$$
\end{proof}
Note that a sequence satisfying \eqref{u-uk} does always exist for bounded Lipschitz domains by Theorem~\ref{densL}.

The following statement gives the equivalence of the classical co-normal derivative (in the trace
sense) and the canonical co-normal derivative, for functions from
$H^s(\Omega)$, $s>\tha$.

\begin{corollary}\label{CCCD}
If  $u\in
H^s(\Omega)$, $s>\tha$, then $T^+ u=T^+_c u\in L_2(\pO)$.
\end{corollary}
\begin{proof}
If $u\in H^s(\Omega)$, $\tha<s<\frac{5}{2}$, then $\gamma^+[\partial_j u]\in H^{s-\tha}(\pO)$, $T^+_c u\in L_2(\pO)$  and $u\in H^{s,
s-2}(\Omega;A)\subset H^{s,
-\ha}(\Omega;A)\subset H^{1,-\ha}(\Omega;A)$  by Remark~\ref{R3.4}.  Let $\{u_k\}\in
\D(\ov{\Omega})$ be a sequence such that $\|u_k-u\|_{H^s(\Omega)}\to 0$  and thus
 $$\|u_k-u\|_{H^{1,-\ha}(\Omega;A)}\le\|u_k-u\|_{H^{s,-\ha}(\Omega;A)}\le C\|u_k-u\|_{H^{s}(\Omega)} \to 0,  \quad
k\to\infty.$$
Then
 \bes\label{EC3.12}
\|T^+u-T^+_c u\|_{H^{-\ha}(\pO)}\le \|T^+u-T^+_c u_k\|_{H^{-\ha}(\pO)}+\|T^+_c(u_k-u)\|_{H^{-\ha}(\pO)},
 \ees
where the first norm in the right hand side vanishes as $k\to\infty$ by Lemma \ref{CCCDLD}, while for the second norm we have,
\begin{multline*}
\|T^+_c(u_k-u)\|_{H^{-\ha}(\pO)}\le
\|\sum_{i,j=1}^n a_{ij}\gamma^+[\partial_j(u_k-u)]n_j\|_{L_2(\pO)}\le\\
 C_1\|a\|_{L_\infty(\pO)}\
\|\gamma^+\nabla(u_k-u)\|_{L_2(\pO)}\le
 C_2\|a\|_{L_\infty(\pO)}\
\|u_k-u\|_{H^{s}(\Omega)}\to 0, \quad
k\to \infty.
\end{multline*}
For $s\ge\frac{5}{2}$ the corollary follows by imbedding.
\end{proof}

For a Lipschitz domain $\Omega$, the membership $u\in H^{s,t}_{loc}(\Omega;A)$ with $\ha<s<\tha$, $-\ha<t<\ha$ implies  by Theorem~\ref{RegThSm} that $u\in H^{t+2}_{loc}(\Omega)$. Thus $u\in H^{t+2}_{loc}(\ov\Omega_1)$ for any Lipschitz subdomain $\Omega_1 $ of $\Omega$ such that $\ov \Omega_1\subset\Omega$.  On $\pO_1$ then
$T^+ u=T^+_c u\in L_2(\pO_1)$  by Corollary~\ref{CCCD}.

\begin{lemma}\label{T+=limT+c}
Let $\Omega$ and $\{\Omega_k\}$ be Lipschitz domains such that $\ov{\Omega}_k\subset\Omega$ and $\Omega_k\to\Omega$ as $k\to \infty$ (cf. Definition~\ref{DOktoO}). If  $u\in H^{s,t}_{loc}(\ov\Omega;A)$ for some $s\in (\ha,\tha)$ and $t\in(-\ha,\ha)$, then
$\langle T^+ u,v^+ \rangle_\pO =\lim_{k\to\infty}\langle T^+_c u,v^+ \rangle_{\pO_k}$ for any $v\in H^{2-s}(\Omega^+)$.
\end{lemma}
\begin{proof}
By Theorem~\ref{GCDlc1} it suffice to consider only a bounded domain $\Omega$.
Let $\Omega'_k:=\Omega\setminus\ov{\Omega_k}$ be the layer between $\pO$ and $\partial\Omega_k$.  By Theorem~\ref{RegThSm}, $u\in H^{t+2}_{loc}(\Omega)$, which by Corollary~\ref{CCCD} implies $T^+ u=T^+_c u\in L_2(\pO_k)$ on $\pO_k$. Then
\begin{multline}\label{3.28a}
 \langle T^+ u,v^+ \rangle_\pO - \langle T^+_c u,v^+ \rangle_{\pO_k}=
 \langle T^+ u,v^+ \rangle_{\pO'_k}=\\
  \check\E_{\Omega'_k}(u,v)- \langle \tilde A_{\Omega'_k} u,v \rangle_{\Omega'_k}
  =\check\E_{\Omega'_k}(u,v)- \langle A u,\tilde v_{\Omega'_k} \rangle_{\Omega'_k},
 \end{multline}
where
$\tilde A_{\Omega'_k} u=\s E_{\Omega'_k}^t r_{\Omega'_k}A u\in \s H^t(\Omega'_k)$
and
$\tilde v_{\Omega'_k}=\s E_{\Omega'_k}^{-t}r_{\Omega'_k}v\in\s H^{-t}(\Omega'_k)$
are the unique extensions of $r_{\Omega'_k}A u\in H^t(\Omega'_k)$ and $r_{\Omega'_k}v\in H^{2-s}(\Omega'_k)\subset H^{-t}(\Omega'_k)$, respectively.

By \eqref{Echeckdef} and Theorem~\ref{ExtOper} we have for the first term in the right hand side of \eqref{3.28a},
 \begin{multline*}
 |\check\E_{\Omega'_k}(u,v)|\le
  C\sum\limits_{i,j=1}^{n}\|a_{ij}\|_{L_\infty(\Omega'_k)} \|\pa_j u\|_{H^{s-1}(\Omega'_k)} \|\pa_i v\|_{H^{1-s}(\Omega'_k)}
 +\\
 {C\sum\limits_{j=1}^{n} \|b_{j}\|_{L_\infty(\Omega'_k)}\|\pa_j u \|_{H^{s-1}(\Omega'_k)}  \|v\|_{H^{1-s}(\Omega'_k)}
 + C\|c\|_{L_\infty(\Omega'_k)} \| u \|_{H^{s-1}(\Omega'_k)}\|v\|_{H^{1-s}(\Omega'_k)}},
\end{multline*}
where $C$ does not depend on $k$ for sufficiently large $k$.
Then for $\ha<s\le 1$,
 \begin{multline*}
 |\check\E_{\Omega'_k}(u,v)|\le{C\sum\limits_{i,j=1}^{n}\|a_{ij}\|_{L_\infty(\Omega)} \|\pa_j u\|_{H^{s-1}(\Omega'_k)} \|\pa_i v\|_{H^{1-s}(\Omega)}
 +}\\
 {C\sum\limits_{j=1}^{n} \|b_{j}\|_{L_\infty(\Omega)}\|\pa_j u \|_{H^{s-1}(\Omega'_k)}  \|v\|_{H^{1-s}(\Omega)}
 + C\|c\|_{L_\infty(\Omega)} \| u \|_{H^{s-1}(\Omega'_k)}\|v\|_{H^{1-s}(\Omega)}\le}\\
 \{C_1\|\nabla u \|_{H^{s-1}(\Omega'_k)} + C_2\| u \|_{H^{s-1}(\Omega'_k)}\} \|v\|_{H^{2-s}(\Omega)}\to 0,\quad k\to\infty
\end{multline*}
 by Lemma~\ref{Hsto1} since the Lebesgue measure of $\Omega'_k$ tends to zero.
For $1<s< \tha$ similarly,
 \begin{multline*}
 |\check\E_{\Omega'_k}(u,v)|\le{C\sum\limits_{i,j=1}^{n}\|a_{ij}\|_{L_\infty(\Omega)} \|\pa_j u\|_{H^{s-1}(\Omega)} \|\pa_i v\|_{H^{1-s}(\Omega'_k)}
 +}\\
 {C\sum\limits_{j=1}^{n} \|b_{j}\|_{L_\infty(\Omega)}\|\pa_j u \|_{H^{s-1}(\Omega)}  \|v\|_{H^{1-s}(\Omega'_k)}
 + C\|c\|_{L_\infty(\Omega)} \| u \|_{H^{s-1}(\Omega)}\|v\|_{H^{1-s}(\Omega'_k)}\le}\\
 \{C_3\|\nabla v \|_{H^{1-s}(\Omega'_k)} + C_4\|v \|_{H^{1-s}(\Omega'_k)}\} \|u\|_{H^{s}(\Omega)}\to 0,\quad k\to\infty.
\end{multline*}

For the last term in \eqref{3.28a} we have by Lemmas \ref{Hsto3} and \ref{Hsto1},
 \begin{multline*}
|\langle A u,\tilde v_{\Omega'_k} \rangle_{\Omega'_k}|\le
\|A u\|_{H^t(\Omega'_k)}\|\tilde v_{\Omega'_k}\|_{\s H^{-t}(\Omega'_k)}\le
C\|A u\|_{H^t(\Omega'_k)}\|v\|_{ H^{-t}(\Omega)}\le\\
C\|A u\|_{H^t(\Omega'_k)}\|v\|_{ H^{2-s}(\Omega)}\to 0,\quad k\to\infty,
 \end{multline*}
if $-\ha<t\le 0$. On the other hand, if $0<t<\ha$, then again by Lemmas \ref{Hsto3} and \ref{Hsto1},
 \begin{multline*}
|\langle A u,\tilde v_{\Omega'_k} \rangle_{\Omega'_k}|=
|\langle \tilde A_{\Omega'_k} u, v\rangle_{\Omega'_k}|
\le
\|\tilde A_{\Omega'_k} u\|_{\s H^t(\Omega'_k)}\|v\|_{H^{-t}(\Omega'_k)}\le\\
C\|A u\|_{H^t(\Omega)}\|v\|_{ H^{-t}(\Omega'_k)}\to 0,\quad k\to\infty.
 \end{multline*}
\end{proof}

Lemma~\ref{T+=limT+c} allows to show that the classical and canonical co-normal derivatives coincide also in another case (apart from the one from Corollary~\ref{CCCD}). First note, that $C^1(\ov{\Omega})\subset H^{1}(\Omega)$ for bounded domain $\Omega$ and $C^1(\ov{\Omega'})\subset H^{1}(\Omega')$ for any bounded subdomain $\Omega'$ of unbounded domain $\Omega$, but  $C^1(\ov{\Omega})$ is not a subset of $H^{1,t}_{loc}(\ov\Omega;A)$. For $u\in C^1(\ov{\Omega})$, evidently, $\lim_{k\to\infty}\langle T^+_c u,v^+ \rangle_{\pO_k}=\langle T^+_c u,v^+ \rangle_\pO$ for any $v\in H^{2-s}(\Omega^+)$ if $\Omega_k\to\Omega$ as $k\to \infty$, $\ov{\Omega}_k\subset\Omega$. This immediately implies the following statement.
\begin{theorem}\label{CCCDHsg}
If $\Omega$ is a Lipschitz domain and $u\in C^1(\ov{\Omega})\bigcap H^{1,t}_{loc}(\ov\Omega;A)$ for some $t\in(-\ha,\ha)$, then $T^+u=T^+_cu\in L_\infty(\pO)$.
\end{theorem}

\section{Formally adjoined PDE system and the second Green identity}\label{2ndGI}
The PDE system formally adjoined to \eqref{2.1} is given in the strong form as
\begin{equation*}
\label{2.1*}  A^*v(x):=
 -\sum_{i,j=1}^n\pa_i [\bar  a_{ji}^\top(x)\,\pa_j v(x)]
 -\sum_{j=1}^n \,\pa_j [\,\bar b_{j}^\top(x)v(x)]+\bar c^\top(x)v(x)=
 f(x),  \;\;\;\; x \in \Omega.
\end{equation*}
Similar to the operator $A$, for any $v\in H^{2-s}(\Omega)$, $s\in\R$, the weak form of the operator $A^*$ is
\bes\label{LH1*}
    \langle A^*v,u \rangle_\Omega:=\E^*(v,u)\quad \forall u\in
    \s{H}^{s}(\Omega),
\ees
    where
\bes
    \E^*(v,u)=\ov{\E(\bar u,\bar v)}
\ees
is the bilinear form and so defined operator $A^*:
H^{2-s}(\Omega)\to H^{-s}(\Omega)=[\s{H}^{s} (\Omega)]^*$ is bounded for any $s\in\R$.

For $\ha<s<\tha$ let us consider also the {\em aggregate} operator
 $\check{A}^*: H^{2-s}(\Omega)\to\s{H}^{-s}(\Omega)=[{H}^{s} (\Omega)]^*$,
defined as,
\begin{equation}\label{Ltil*}
    \langle \check{A}^*v,u \rangle_\Omega:=\check\E^*(v,u)\quad \forall u\in
    {H}^{s}(\Omega),
\end{equation}
    where by \eqref{Echeckdef},
\begin{multline}
    \check\E^*(v,u)=\ov{\check\E(\bar u,\bar v)}=\Phi(\bar u,v)\label{E*E}
=\\
  \sum\limits_{i,j=1}^{n} \left\langle \bar a_{ij}\pa_j u , \s{E}^{1-s}\pa_i v\right\rangle_{\Omega}
 +\sum\limits_{j=1}^{n} \left\langle \bar b_{j}\pa_j u ,\s{E}^{1-s} v\right\rangle_{\Omega}
 +\left\langle \bar c u , \s{E}^{1-s}v\right\rangle_{\Omega}\quad
\end{multline}
which implies that $\check{A}^*: H^{2-s}(\Omega)\to\s{H}^{-s}(\Omega)$ is  bounded. For any $v\in H^{2-s}(\Omega)$, the
distribution $\check{A}^*v$ belongs to $\s{H}^{-s}(\Omega)$ and is  an extension of the functional ${A}^*v\in {H}^{-s}(\Omega)$ from the
domain of definition  $\s{H}^{s}(\Omega)$ to the domain of
definition ${H}^{s}(\Omega)$.

Relations \eqref{Ltil*}, \eqref{E*E} and  \eqref{Ltil} lead to the {\em aggregate second Green identity},
\begin{equation}\label{A2GI}
\langle \check{A}u,\bar v \rangle_\Omega=\langle u, \ov{\check{A}^*v}\rangle_\Omega,\quad u\in
    {H}^{s}(\Omega),\quad v\in H^{2-s}(\Omega), \quad\ha<s<\tha.
\end{equation}

For a sufficiently smooth function $v$, let
 \begin{equation*}\label{clCDm}
 T^+_{*c}v(x) := \sum\limits_{i,j=1}^n
\bar a_{ji}^\top(x)\,\gamma^+[\pa_j v(x)]\nu_i(x)+\sum_{i=1}^n \, \bar b_{i}^\top(x)\gamma^+v(x)\nu_i
 \end{equation*}
be the strong (classical) modified co-normal derivative (it corresponds to  $\s{\mathfrak{B}}_\nu v$ in \cite{McLean2000}), associated with the  operator $A^*$.

If $v\in H^{2-s}(\Omega)$, $\ha< s<\tha$,  and $A^*v=\tilde{f}_*|_{_\Omega}$ in
$\Omega$ for some $\tilde{f}_*\in \s{H}^{-s}(\Omega)$, we
define {\em the generalized modified co--normal derivative\ } $ T^+_*(\tilde{f}_*,v) \in H^{\ha-s}(\pO)$,  associated with the  operator $A^*$, similar to  Definition \ref{GCDd}, as
 $$
\left\langle
  T^+_*(\tilde{f}_*,v)\,,\, w
\right\rangle _{\pO}:=
\check\E^*(v,\gamma_{-1} w)-\langle \tilde{f}_*,\gamma_{-1} w \rangle_\Omega
  \quad  \forall\ w\in H^{s-\ha} (\pO).
 $$
As in Theorem \ref{GCDl}, this leads to the following first Green identity for the function $v$,
\begin{equation}
\label{Tgen*} \left\langle
 T^+_*(\tilde{f}_*,v)\,,\, u ^{+}
\right\rangle _{\pO}
=\check\E^*(v,u)-\langle \tilde{f}_*,u \rangle_\Omega
  \quad  
\forall\ u\in H^{s}(\Omega),
\end{equation}
which by \eqref{E*E}  implies
 \begin{equation}
\label{Tgen*bar} \left\langle
u ^{+},\ov{ T^+_*(\tilde{f}_*,v)}
\right\rangle _{\pO}
=\check\E(u,\bar v)-\langle u, \ov{\tilde{f}}_* \rangle_\Omega
  \quad  
\forall\ u\in H^{s}(\Omega).
\end{equation}
If, in addition, $Au=\tilde{f}|_{_\Omega}$ in $\Omega$ with some
$\tilde{f}\in \s{H}^{s-2}(\Omega)$, then combining \eqref{Tgen*bar} and
the first Green identity \eqref{Tgen} for $u$, we arrive at the following generalized second Green identity,
\begin{equation} \label{2.5sg*}
\langle \tilde{f},\bar v \rangle_\Omega\
 -\langle u, \ov{\tilde{f}}_*\rangle _{_\Omega}=
 \left\langle u ^{+}, \ov{ T^+_*(\tilde{f}_*, v)}\right\rangle _{_\pO}
 - \left\langle T^+(\tilde{f},u)\,,\, \ov{v^+}\right\rangle _{_\pO}
 .
\end{equation}
Taking in mind \eqref{Tgen*}, \eqref{Ltil*} and \eqref{Tgen}, \eqref{Ltil}, this, of course, leads to the aggregate second Green identity \eqref{A2GI}.

If  $\ha<s<\tha$ and $v \in H^{2-s,-\ha}(\Omega;A^*)$, then similar to Definitions \ref{Dce} and \ref{Dccd} we can introduce the {\em canonical extension} $\s A^*$ of the operator $A^{*}$, and the {\em canonical modified
co-normal derivative} $ T^+_*v:=  T^+_*(\s A^* v,v)\in H^{\ha-s}(\pO)$, i.e.,
$$
\left\langle
  T^+_*v\,,\, w
\right\rangle _{\pO}:=
\check\E^*(v,\gamma_{-1} w)-\langle \s A^*v,\gamma_{-1} w \rangle_\Omega
  \quad  \forall\ w\in H^{s-\ha} (\pO).
 $$
Then the first Green identity \eqref{Tgen*bar} becomes,
\begin{equation*}
\label{Tcon*bar} \left\langle
u ^{+},\ov{ T^+_* v}
\right\rangle _{\pO}
=\check\E(u,\bar v)-\langle u, \ov{\s A^*v} \rangle_\Omega
  \quad  
\forall\ u\in H^{s}(\Omega).
\end{equation*}
For $u\in H^s(\Omega)$, $Au=\tilde{f}|_{_\Omega}$ in $\Omega$, where
$\tilde{f}\in \s{H}^{s-2}(\Omega)$, the second Green identity \eqref{2.5sg*} takes form,
\begin{equation} \label{2.5s*}
\langle \tilde{f},\bar v \rangle_\Omega\
 -\left\langle u,\ov{\s A^*v}\right\rangle _{_\Omega}=
 \left\langle u ^{+},\ov{ T^+_*v}\right\rangle _{_\pO}
 - \left\langle T^+(\tilde{f},u),\ov{v^+}\right\rangle _{_\pO}.
\end{equation}
 This form was a starting point in formulation and analysis of the extended boundary-domain integral equations in \cite{Liverpool2005UKBIM}.

If, moreover, $u \in H^{s,-\ha}(\Omega;A)$,  we obtain from
\eqref{2.5s*} the second Green identity for the canonical extensions and
canonical co-normal derivatives,
\begin{equation} \label{GreenCan*}
\left\langle \tilde Au,\bar v\right\rangle _{_\Omega}
-\left\langle u,\ov{\s A^*v}\right\rangle _{_\Omega}=
\left\langle u ^{+},\ov{ T^+_*v}\,\right\rangle _{_\pO}
-\left\langle T^{+}u\,,\, \ov{v^+}\right\rangle _{_\pO}
 .
\end{equation}
Particularly, if $u,v \in H^{1,0}(\Omega;A)$, then \eqref{GreenCan*} takes the familiar form, cf. \cite[Lemma 3.4]{Costabel1988},
 $$
\int_{\Omega}[\,\ov{v(x)}Au(x)- u(x)\ov{A^*v(x)}\,]dx=
\left\langle u ^{+},\ov{ T^+_*v}\,\right\rangle _{_\pO}-
 \left\langle T^{+}u\,,\, \ov{v^+}\right\rangle _{_\pO}
.
 $$




\section{APPENDIX}
\begin{lemma}\label{contrHW}
There exist a distribution $w\in H^{-1}_{\pO}$ and a function $f\in L_2(\R^n)$, $f=0$ on $\Omega^-$, such that $(w,f)_{H^{-1}(\R^n)}\not=0$.
\end{lemma}
\begin{proof}
Under the definition \eqref{ipRn} of the inner product in $H^s(\R^n)$,
 \be\label{wf}
{(w,f)}_{H^{-1}(\R^n)}
 =\langle \ov w,\J^{-2}f\rangle_{\R^n}.
 \ee
By Theorem \ref{H_F=0}, for any distribution $\ov w\in H^{-1}_{\pO}$ there exists a distribution $v\in H^{-1/2}(\pO)$ such that
 \be\label{vJ-2}
\langle\ov w,\J^{-2}f\rangle_{\R^n}=\langle v,\gamma\J^{-2}f\rangle_{\pO},
 \ee
where $\gamma$ is the trace operator.

Denoting $\Phi=\J^{-2}f\in H^2(\R^n)$, we have, $\J^2\Phi=f$ in $\R^n$, and taking in mind the explicit representation for the operator $\J^2$, the latter equation can be rewritten as
 \be\label{Phieq}
\J^2\Phi\equiv-\frac{1}{4\pi^2}\Delta\Phi+\Phi=f\quad\mbox{in } \R^n
 \ee
and its solution as
\bes
\J^{-2}f(y)=\Phi(y)={\P}f:=\int_\Omega F(x,y)f(x)dx,\quad y\in\R^n.
\ees
Here ${\P}$ is the Newton volume potential and $F(x,y)$ is the well known fundamental solution of equation \eqref{Phieq}. For example, for $n=3$,
 \be\label{AF}
 F(x,y)=C\frac{e^{-2\pi|x-y|}}{|x-y|}\ .
 \ee
Then \eqref{wf}, \eqref{vJ-2} give,
 \be\label{Ap1}
{(w,f)}_{H^{-1}(\R^n)}=\langle v,\gamma\J^{-2}f\rangle_{\pO}=
\langle v,\gamma{\P}f\rangle_{\pO}.
 \ee
If we assume $(w,f)_{H^{-1}(\R^n)}=0$ for any $w\in H^{-1}_{\pO}$, then \eqref{Ap1} implies $\gamma{\P}f=0$, which is not the case for arbitrary $f\in L_2(\Omega)$ and particularly for $f=1$ in $\Omega$ due to \eqref{AF}.
\end{proof}


\newpage


\title
\chapter{Solution regularity and co-normal derivatives for elliptic systems with non-smooth coefficients on Lipschitz domains}

\begin{center}{Sergey E. Mikhailov\footnote{
e-mail: {\sf sergey.mikhailov@brunel.ac.uk}, 
     Fax: +44\,189\,5269732}\\
     Brunel University London,
     Department of Mathematics,\\
     Uxbridge, UB8 3PH, UK}
\end{center}

 \noindent{Published in: {\em J. Math. Anal. and Appl.,} 2012, DOI: 10.1016/j.jmaa.2012.10.045
\hfill\phantom{*} 
 }\\

\makeatletter
\renewcommand{\@oddhead}{\vbox{\hbox to\textwidth{\scriptsize %
 {\em JMAA}, 2012, DOI: 10.1016/j.jmaa.2012.10.045
 \hfill S.E.Mikhailov 
 }
 \hrule
 }}
 \makeatother
 
{\bf Abstract}\\
Elliptic PDE systems of the second order  with  coefficients from $L_\infty$ or H\"older-Lipschitz spaces are considered in the paper. Continuity of the operators in corresponding Sobolev spaces is stated and the internal (local) solution regularity theorems are generalized to the non-smooth coefficient case.
For functions from the Sobolev space $H^s(\Omega)$, $\ha<s<\tha$,
definitions of non-unique generalized  and unique canonical
co-normal derivative are considered, which are related to possible
extensions of a partial differential operator and the PDE right hand
side from the domain $\Omega$ to its boundary. It is proved that the canonical co-normal derivatives coincide with the classical ones when both exist.
A generalization of  the boundary value problem settings, which makes them insensitive to the co-normal derivative inherent non-uniqueness is given. \\

\noindent{\bf Keywords}. Partial differential equation systems; Non-smooth coefficients; Sobolev spaces; Solution regularity; Classical, generalized and canonical co-normal derivatives; Weak BVP settings.

\section{Introduction}

It is well known that for a function from the Sobolev space $H^{s}(\Omega)$,
$\ha<s<\tha$, the {\em strong}  co-normal derivative defined on  the boundary  in the trace sense, does not generally exist. Instead, if the function satisfies a second order partial differential equation (or a system of such equations) with a  right-hand side from $H^{s-2}(\Omega)$, a {\em generalized}  co-normal derivative operator can be defined by the first Green's identity, cf. e.g. \cite[Lemma 4.3]{McLean2000-H} for $s=1$. However this
definition is related to an extension of the PDE operator and its right
hand side from the domain $\Omega$, where they are prescribed, to
the domain boundary, where they are not. Since the extensions are
non-unique, the generalized co-normal derivative operator appears to be
non-unique and non-linear unless a
linear relation between the PDE solution and the extension of its right hand side
is enforced. This leads to a revision of  the boundary value problem settings, to make them insensitive to the co-normal derivative inherent non-uniqueness.
For functions $u$ from a subspace of $H^{s}(\Omega)$, $\ha<s<\tha$,
which can be mapped by the (extended) PDE operator into the space
$\widetilde{H}^{t}(\Omega)$, $t\ge -\ha$,  one can define a {\em
canonical} co-normal derivative (cf. \cite[Theorem 1.5.3.10]{Grisvard1985-H} and \cite[Lemma 3.2]{Costabel1988-H} for $s=1$, $t=0$),
which is unique, linear
in $u$, and coincides with the co-normal derivative in the trace
sense if the latter does exist.
These notions were developed, to some extent, in \cite{MikMMAS2006-H} for a PDE with an infinitely smooth coefficient on a
domain with an infinitely smooth boundary. In \cite{MikJMAA2011}  the
analysis was generalized to the co-normal derivative operators
for some elliptic PDE systems with infinitely smooth coefficients and the right hand side
from ${H}^{s-2}(\Omega)$, $\ha<s<\tha$, on a Lipschitz
domain.

In this paper, we extend the previous results to solutions of elliptic second order PDE systems on interior or exterior Lipschitz domains with compact boundaries and  $L_\infty$ or H\"older-Lipschitz coefficients. To show that the canonical co-normal derivatives coincide with the classical ones, some new facts about solution regularity of PDEs with non-smooth coefficients are also proved in the paper.

The paper is arranged as follows.  Section~\ref{S2-H} provides a number of
auxiliary facts on Sobolev (Bessel potential) spaces.
In Section~\ref{S3sm1}, we describe some $L_\infty-$based  Sobolev-Slobodetski spaces that essentially coincide with the H\"older-Lipschitz spaces, to use them for PDE coefficients, and prove boundedness of PDE operators with such coefficients in appropriate Sobolev spaces.
In Section~\ref{SolReg} we generalize the well know result about the local solution regularity of elliptic PDE systems to the case of relaxed smoothness of the PDE coefficients. In addition to the  differentiation  argument employed usually in the solution regularity analysis, we use for our proof also the Bessel potential operator that appeared to be more suitable for H\"older-smooth coefficients.   The solution regularity theorems are implemented  then  in Section~\ref{S-CCD}.
In Section~\ref{S3sm3} all results of \cite{MikJMAA2011} about the generalized co-normal derivatives for PDE systems with smooth coefficients are extended to non-smooth coefficients.
Particularly, we
introduce and analyse the generalized  co-normal derivatives on interior and exterior Lipschitz
domains (with compact boundaries), associated with  elliptic systems of second order PDEs with the right hand side
from ${H}^{s-2}(\Omega)$, $\ha<s<\tha$. The weak settings of Dirichlet, Neumann and mixed problems (revised versions for the latter two) are considered and it is shown that they are well posed in spite of the inherent non-uniqueness of the generalized co-normal derivatives.
In Section~\ref{S-CCD} we introduce and analyse the canonical  co-normal derivative operator uniquely defined on some subspaces ${H}^{s,t}(\Omega;A)$ of the usual Sobolev spaces ${H}^{s}(\Omega)$, $\ha<s<\tha$, $-\ha\le t$, generalizing the corresponding results of \cite{MikJMAA2011} to the case of non-smooth coefficients of the PDE operator.
It is proved that for elliptic systems the canonical co-normal derivative coincides with the classical (strong) one for the cases when both exist. Some auxiliary estimates and necessary assertions from \cite{MikJMAA2011} are provided in two Appendices.

The present paper updates and complements the preliminary results from \cite{MikhailovIMSE06-H}.

\section{Some function spaces}\label{S2-H}

\subsection{Sobolev spaces}\label{S2.1-H}
Unless stated otherwise, we suppose that $\Omega=\Omega^+$ is an interior or exterior open domain of
$\R^n$, which boundary $\pa \Omega$ is a  compact,
Lipschitz, $(n-1)-$dimensional set. Let $\ov{\Omega}$ denote the closure of $\Omega$ and $\Omega^-=\R^n\backslash \ov{\Omega}$ its complement. In what follows
${\cal D}(\Omega)=C^\infty_{\mathrm{comp}}(\Omega)$ denotes the space of Schwartz
test functions, ${\cal D}(\overline\Omega):=\{\varphi=\phi|_\Omega,\ \phi\in{\cal D}(\R^n)\}$, while ${\cal D}^*(\Omega)$ denotes the space of
Schwartz distributions; $ H^s(\R^n)= H^s_2(\R^n)$, $
H^s(\pO)=H^s_2(\pO)$ are the Sobolev  (Bessel potential) spaces, where $s\in
\R$ is a number (see, e.g., \cite{LiMa1-H}).

 We denote by $\s{H}^s (\Omega)$ the closure of $\D(\Omega)$ in  $ H^s(\R^n)$, which can be characterized as
 $
\s{H}^s (\Omega)=\{g:\;g\in H^s  (\R^n),\; \supp \,g
\subset\ov{\Omega}\}
 $, see e.g. \cite[Theorem 3.29]{McLean2000-H}.
The space ${H}^s (\Omega)$ consists of restrictions on
$\Omega$ of distributions  from ${H}^s  (\R^n)$,
 ${H}^s (\Omega):=\{g|_{_{\Omega}}:\;g\in{H}^s (\R^n)\}$,
and $H_0^s(\Omega)$ is the closure of ${\cal D}(\Omega)$ in
$H^s(\Omega)$. We recall that $H^s(\Omega)$ coincide with the
Sobolev--Slobodetski spaces $W^s_2(\Omega)$ for any non-negative $s$.
We denote ${H}^s_{\mathrm{loc}}(\Omega):=\{g : \varphi g\in{H}^s(\Omega)\ \forall\varphi\in\D(\Omega)\}$. We will use also the notation $H^{s}_{\mathrm{loc}}(\ov\Omega):=\{g : \varphi g\in{H}^s(\Omega)\ \forall\varphi\in\D(\ov\Omega)\}$ and note that $H^{s}_{\mathrm{loc}}(\ov\Omega)=H^{s}(\Omega)$ for interior domains but not for the exterior ones.

Note that distributions from ${H}^s (\Omega)$ and $H_0^s(\Omega)$ are defined only in $\Omega$, while distributions from $\s{H}^s (\Omega)$ are defined in $\R^n$ and include the distributions supported only on the boundary $\pO$. For  $s\ge 0$, we can identify $\s{H}^s (\Omega)$ with the subset of functions from ${H}^s (\Omega)$, whose extensions by zero outside $\Omega$ belong to ${H}^s (\R^n)$, cf. \cite[Theorem 3.33]{McLean2000-H}, i.e., identify functions $u\in \s{H}^s (\Omega)$ with their restrictions, $u|_\Omega\in {H}^s (\Omega)$. However generally we will distinguish distributions $u\in\s{H}^s (\Omega)$ and $u|_\Omega\in {H}^s (\Omega)$, especially for $s\le-\ha$.

We denote by ${H}^s_{_\pO}$ the subspace of ${H}^s (\R^n)$ (and of
$\s{H}^s (\Omega)$), whose elements are supported on $\pO$, i.e.,
 $
{H}^s_{_\pO}:=\{g:\;g\in H^s  (\R^n),\; \supp \,g \subset{\pO}\}.
 $
A characterization of this space is provided in Theorem~\ref{H_F=0-H} in Appendix B.
To simplify notations for vector-valued functions, $u:\Omega\to\C^m$, we will often write $u\in H^s(\Omega)$ instead of $u\in H^s(\Omega)^m=H^s(\Omega;\C^m)$, etc.

As usual (see e.g. \cite{LiMa1, McLean2000}), for two  elements from dual complex Sobolev spaces the bilinear dual product $\langle \cdot,\cdot\rangle_\Omega$ associated with the sesquilinear inner product  $(\cdot,\cdot)_\Omega:=(\cdot,\cdot)_{L_2(\Omega)}$ in $L_2(\Omega)$ is defined as
 \bea
\langle u,v\rangle_{\R^n}&:=&\int_{\R^n} [\F^{-1} u](\xi)[\F v](\xi)d\xi=:(\F{\bar{u}},{\F v})_{\R^n}=:(\bar{u}, v)_{\R^n},
\quad
u\in H^s(\R^n),\ v\in H^{-s}(\R^n),\label{dp-H}\\
\langle u,v\rangle_{\Omega}&:=&\langle u,V\rangle_{\R^n}=:(\bar{u}, v)_{\Omega}\ \text{if}\ u\in \s H^s(\R^n),\ v\in H^{-s}(\Omega),\ v=V|_\Omega \text{ with } V\in H^{-s}(\R^n),\qquad\label{dp2}\\
\langle u,v\rangle_{\Omega}&:=&\langle U,v\rangle_{\R^n}=:(\bar{u},{v})_{\Omega}\ \text{if}\ u\in H^s(\R^n),\ v\in\s H^{-s}(\Omega),\ u=U|_\Omega \text{ with } U\in H^{s}(\R^n)\label{dp3-H}
 \eea
for $s\in\R$, where $\bar g$ is the complex conjugate of $g$, while $\F$ and $\F^{-1}$ are the distributional  Fourier transform operator and its inverse, respectively, that for integrable functions take form
 $$
 \hat g(\xi)=[\F g](\xi):=\int_{\R^n}e^{-2\pi i x\cdot\xi}g(x)dx,\quad
 g(x)=[\F^{-1} \hat g](x):=\int_{\R^n}e^{2\pi i x\cdot\xi}\hat g(\xi)d\xi .
 $$
For  vector-valued elements $u\in H^s(\R^n)^m$, $v\in H^{-s}(\R^n)^m$, $s\in\R$,  definition  \eqref{dp-H} should be understood as
 \bes
\langle u,v\rangle_{\R^n}:=\int_{\R^n} {\hat u(\xi)}\cdot\hat v(\xi)d\xi=\int_{\R^n} {\hat u(\xi)}^\top\hat v(\xi)d\xi=:(\ov{\hat u},{\hat v})_{\R^n}=:(\bar{u},{v})_{\R^n},  
 \ees
where $\hat u\cdot\hat v=\hat u^\top\hat v=\sum_{k=1}^m\hat u_k\hat v_k$ is the product of two $m-$dimensional vectors.

Let
$\mathcal{J}^s$ be the Bessel potential operator defined as
\be \label{BPD}
[\J^sg](x)=\F^{-1}_{\xi\to x}\{(1+|\xi|^2)^{s/2}\hat g(\xi)\}.
\ee
The inner product  in $H^s(\Omega)$, $s\in\R$,  is defined as follows,
\beas
&&\hspace{-2em}(u,v)_{H^s(\R^n)}:=(\J^s u, \J^s v)_{\R^n}=
\int_{\R^n}(1+\xi^2)^s\ov{\hat u(\xi)}{\hat v}(\xi) d\xi
=
\left\langle {\ov u},\J^{2s}{v}\right\rangle_{\R^n},\quad u,v\in H^s(\R^n),
\qquad\qquad
 \\
 \nonumber
&&\hspace{-2em}(u,v)_{H^s(\Omega)}:=\left((I-P)U,(I-P)V\right)_{H^s(\R^n)},\quad
u=U|_\Omega,\ v=V|_\Omega,\quad U,V\in H^s(\R^n).\qquad\qquad
\eeas
Here $P: H^s(\R^n)\to \s H^s(\R^n\backslash \bar \Omega)$ is the orthogonal projector, see e.g. \cite[p. 77]{McLean2000-H}.

\section{Elliptic PDE systems with non-smooth coefficients}\label{S3sm1}

\subsection{Some  Sobolev-Slobodetski and H\"older-Lipschitz spaces}
For an open set $\Omega$ let $W^\mu_\infty(\Omega)$, $\mu\ge 0$, be the Sobolev-Slobodetski space equipped with the norm
 $$\|g\|_{W^\mu_\infty(\Omega)}:=\sum_{0\le|\alpha|\le\mu}\|\partial^\alpha g\|_{L_\infty(\Omega)}<\infty$$
for integer $\mu$, and with the norm
 $$\|g\|_{W^\mu_\infty(\Omega)}:=\|g\|_{W^{\lfloor\mu\rfloor}_\infty(\Omega)}+ |g|_{W^{\mu}_\infty(\Omega)}<\infty,\quad   |g|_{W^{\mu}_\infty(\Omega)}:=\sum_{|\alpha|=\lfloor\mu\rfloor}\left\|\frac{\partial^\alpha g(x)-\partial^\alpha g(y)}{|x-y|^{\mu-\lfloor\mu\rfloor}}\right\|_{L_\infty(\Omega\times\Omega)}$$
for non-integer $\mu$, where $\lfloor\mu\rfloor$ is the integer part of $\mu$.
Evidently $W_\infty^0(\Omega)=L_\infty(\Omega)$, while (possibly after adjusting functions on zero measure sets, cf. \cite[Ch. V, \S 4, Proposition 6]{Stein1970}) $W_\infty^\mu(\Omega)$ is  the usual
H\"older space $C^{\mu}(\Omega)=C^{0,\mu}(\Omega)$ for $ 0<\mu< 1$, $W_\infty^{\mu}(\Omega)=C^{\lfloor\mu\rfloor,\mu-\lfloor\mu\rfloor}(\Omega)$ for non-integer
$\mu>1$, and  $W_\infty^{\mu}(\Omega)$  is the Lipschitz space $C^{\mu-1,1}(\Omega)$ for integer $\mu\ge 1$.

Let $\R_+(s)$ be the set of all non-negative numbers if $s$ is integer and of all positive numbers otherwise.

\begin{definition}\label{Chat}For an open set $\Omega$ and $\mu\ge 0$ let $\bar{C}^\mu(\overline\Omega)$ be the set of restrictions on $\Omega$ of all functions from $W^\mu_\infty(\R^n)$, equipped with the norm $\|g\|_{\bar{C}^\mu(\overline\Omega)}=\inf_{G|_\Omega=g}\|G\|_{W_\infty^\mu(\R^n)}$.

The set $\bar{C}^\mu_+(\overline\Omega)$ is defined as $\bar{C}^\mu(\overline\Omega)$ for integer non-negative $\mu$ and as $\bigcup_{\nu>\mu}\bar{C}^\nu(\overline\Omega)$ for non-integer nonnegative $\mu$. Evidently $g\in \bar{C}^\mu_+(\overline\Omega)$  if and only if $g\in \bar{C}^{\mu +\epsilon}(\overline\Omega)$ for some $\epsilon\in\R_+(\mu)$.
\end{definition}
Obviously $\|g\|_{W_\infty^0(\Omega)}= \|g\|_{\bar{C}^0(\overline\Omega)}=\|g\|_{L_\infty(\Omega)}$ i.e.  $\bar{C}^0(\overline\Omega)=W_\infty^0(\Omega)=L_\infty(\Omega)$, and  $\|g\|_{W_\infty^\mu(\Omega)}\le \|g\|_{\bar{C}^\mu(\overline\Omega)}$, $\bar{C}^\mu(\overline\Omega)\subset W_\infty^\mu(\Omega)$ for $ \mu>0$.
The space $\bar{C}^\mu(\overline\Omega)$ for $\mu>0$ is similar to the space $C_c^{\lfloor\mu\rfloor, \mu-\lfloor\mu\rfloor}(\bar\Omega)$ for non-integer $\mu$ and to the space $C_c^{\mu-1,1}(\bar\Omega)$ for integer $\mu$, used in \cite[p.21]{Grisvard1985-H}, except that functions from $\bar{C}^\mu(\overline\Omega)$ may not have a compact support in $\R^n$ assumed for functions from $C_c^{k,\alpha}(\bar\Omega)$.

\begin{theorem}\label{GrL}
Let $\Omega$ be an open set, $s\in\R$, $g\in \bar{C}^{\mu}(\overline\Omega)$,
$\mu-|s|\in\R_+(s)$.
 Then $gv\in H^s(\Omega)$
for every $v\in H^s(\Omega)$, and
 $
\|gv\|_{H^s(\Omega)}\le C \|g\|_{\bar{C}^\mu(\overline\Omega)} \|v\|_{H^s(\Omega)},
 $
where $C$ is independent of $g$, $v$ or $\Omega$.
\end{theorem}
\begin{proof}
Note that the theorem is close to the statement given in \cite[Theorem
1.4.1.1]{Grisvard1985-H} without proof.

Let first $\Omega=\R^n$. The case $s=0$ is evident.
For $s>0$  the estimate can be obtained from \cite[Theorem 2(b)]{Zolesio1977}  with parameters $s_1=\mu$, $s_2=s$, $p_1=\infty$, $q_1=p_2=q_2=p=q=2$ there (see also \cite[Theorem 1.4.4.2]{Grisvard1985-H}). A simpler proof for all $s\in\R$ is available in \cite[\S 9, Theorems 11-13]{Agranovich2006SS}.

When $\Omega\not=\R^n$, let $V\in H^s(\R^n)$ and $G\in W_\infty^s(\R^n)$ be such that $v=V|_\Omega$, $\|V\|_{H^s(\R^n)}=\|v\|_{H^s(\Omega)}$, $g=G|_\Omega$, $\|G\|_{W^{\mu}_\infty(\R^n)}<2\|g\|_{\bar{C}^{\mu}(\overline\Omega)}$.
Then $GV\in H^s(\R^n)$ by the previous paragraph and
\begin{equation*}
\|gv\|_{H^s(\Omega)}\le \|GV\|_{H^s(\R^n)}\le C \|G\|_{W^{\mu}_\infty(\R^n)} \|V\|_{H^s(\R^n)}<2C\|g\|_{\bar{C}^{\mu}(\overline\Omega)} \|v\|_{H^s(\Omega)}.
\end{equation*}
\end{proof}
Note that the condition on $g$ in Theorem~\ref{GrL} is equivalent to the membership $g\in\bar{C}^{|s|}_+(\overline\Omega)$.

\subsection{PDE systems}\label{S3sm2}
Let us consider in an open set ${\Omega}$ a system of $m$ complex linear differential equations of the second order with respect to $m$ unknown functions $\{u_{i}\}_{i=1}^m=u: \Omega\to \C^m$, which for sufficiently smooth $u$ and $f$ has the following strong form,
\begin{equation}
\label{2.1-H}  Au(x):=
 -\sum\limits_{i,j=1}^n\pa_i[  a_{ij}(x)\,\pa_j u(x)]
 +\sum\limits_{j=1}^n b_{j}(x)\,\pa_j u(x)+c(x)u(x)=
 f(x),  \;\;\;\; x \in \Omega,
\end{equation}
where $f: \Omega\to\C^m$,  $\pa_j:=\pa/\pa{x_j}$
 $(j=1,2,...,n)$, $a(x)=\{a_{ij}(x)\}_{i,j=1}^n=\{\{a_{ij}^{kl}(x)\}_{k,l=1}^m\}_{i,j=1}^n$,  $b(x)=\{\{b_{i}^{kl}(x)\}_{k,l=1}^m\}_{i=1}^n$ and $c(x)=\{c^{kl}(x)\}_{k,l=1}^m$, i.e., $a_{ij}, b_{i},c: \Omega\to\C^{m\times m}$ for fixed indices $i,j$. If $m=1$, then \eqref{2.1-H} is a scalar equation.
The PDE system formally adjoint to \eqref{2.1-H} is given in the strong form as
\begin{equation}
\label{2.1*-H}  A^*v(x):=
 -\sum_{i,j=1}^n\pa_i [\bar  a_{ji}^\top(x)\,\pa_j v(x)]
 -\sum_{j=1}^n \,\pa_j [\,\bar b_{j}^\top(x)v(x)]+\bar c^\top(x)v(x)=
 f(x),  \;\;\;\; x \in \Omega.
\end{equation}

\begin{definition}\label{DC^sigma}
For $\sigma\in\R$, we will say that the coefficients of equation \eqref{2.1-H} belong to the class ${\cal C}_+^\sigma(\overline\Omega)$, i.e. $\{a,b,c\}\in {\cal C}_+^\sigma(\overline\Omega)$,  if
$a\in \bar{C}^{|\sigma|}_+(\overline\Omega)$, $b\in \bar{C}^{\mu_b(\sigma)}_+(\overline\Omega)$,  $\mu_b(\sigma):=\max(0,|\sigma-\ha|-\ha)$, $c\in \bar{C}^{\mu_c(\sigma)}_+(\overline\Omega)$, $\mu_c(\sigma):=\max(0,|\sigma|-1).$

For an open set $\Omega$, as usual, $\{a,b,c\}\in{\cal C}^\sigma_{+\mathrm{loc}}({\Omega})$ means that $\{a,b,c\}\in{\cal C}_+^\sigma(\ov{\Omega'})$ for any $\ov{\Omega'}\subset\Omega$.
\end{definition}

Note that if $\sigma_1\le\sigma\le\sigma_2$, then
${\cal C}_+^{\sigma_1}(\overline\Omega)\bigcap{\cal C}_+^{\sigma_2}(\overline\Omega)\subset
{\cal C}_+^\sigma(\overline\Omega)\subset
{\cal C}_+^{\sigma_1}(\overline\Omega)\bigcup{\cal C}_+^{\sigma_2}(\overline\Omega)$.

Let  $u\in H^s(\Omega)$, $\{a,b,c\}\in {\cal C}_+^{s-1}(\overline\Omega)$, $f\in H^{s-2}(\Omega)$, $s\in\R$. Equation system \eqref{2.1-H} is understood in the distributional sense as
\bes
    \langle Au,v \rangle_\Omega=\langle f,v\rangle_\Omega\quad \forall v\in
    {\cal D}(\Omega),
\ees
where $v: \Omega\to \C^m$ and
\begin{equation}\label{Ldist-H}
    \langle Au,v \rangle_\Omega:=\E(u,v)\quad \forall v\in
    {\cal D}(\Omega),
\end{equation}
 \be\label{Edef-H}
\E(u,v)=\E_\Omega(u,v):=
  \sum\limits_{i,j=1}^{n} \left\langle a_{ij}\pa_j u , \pa_i v\right\rangle_{\Omega}
 +\sum\limits_{j=1}^{n} \left\langle b_{j}\pa_j u  ,  v\right\rangle_{\Omega}
 +\left\langle c u  ,  v\right\rangle_{\Omega}.
 \ee
 Let us denote
\be\label{sbsc}
s_b(s)=\begin{cases}
s-1 &\mbox{if } s<1\\
0&\mbox{if } 1\le s\le 2\\
s-2&\mbox{if } 2<s
\end{cases}
,\qquad \qquad
s_c(s)=\begin{cases}
s &\mbox{if } s<0\\
0&\mbox{if } 0\le s\le 2\\
s-2&\mbox{if } 2<s
\end{cases}
\ee
Taking into account that $\mu_b(s-1)=|s_b(s)|$ and $\mu_c(s-1)=|s_c(s)|$, Theorem~\ref{GrL} implies that $a_{ij}\pa_j u\in H^{s-1}(\Omega)$, $b_{j}\pa_j u\in H^{s_b(s)}(\Omega)$, $c u\in H^{s_c(s)}(\Omega)$. Thus bilinear form \eqref{Edef-H} is well defined for any $v\in {\cal D}(\Omega)$ and moreover, the bilinear functional $\E:\{H^s(\Omega),\widetilde H^{2-s}(\Omega)\}\to\C$ is bounded for any $s\in\R$.
Since the set ${\cal D}(\Omega)$ is dense in $\s{H}^{2-s}
(\Omega)$, expression \eqref{Ldist-H} defines then a bounded linear operator $A:
H^s(\Omega)\to H^{s-2}(\Omega)=[\s{H}^{2-s} (\Omega)]^*$,
$s\in\R$,
\begin{equation}\label{LH1-H}
    \langle Au,v \rangle_\Omega:=\E(u,v)\quad \forall v\in
    \s{H}^{2-s}(\Omega).
\end{equation}
Similar to the operator $A$,  the weak form of the operator $A^*$ for any $v\in H^{2-s}(\Omega)$, $s\in\R$, is
\be\label{LH1*-H}
    \langle A^*v,u \rangle_\Omega:=\E^*(v,u)\quad \forall u\in
    \s{H}^{s}(\Omega),
\ee
    where
$
    \E^*(v,u)=\ov{\E(\bar u,\bar v)}
$
is the bilinear form and so defined operator $A^*:
H^{2-s}(\Omega)\to H^{-s}(\Omega)=[\s{H}^{s} (\Omega)]^*$ is bounded for any $s\in\R$.

The above paragraph can be summarized as the following assertion.
\begin{theorem}\label{T4.4} If $s\in\R$ and $\{a,b,c\}\in {\cal C}_+^{s-1}(\overline\Omega)$, then bilinear form \eqref{Edef-H}, $\E: \{H^s(\Omega),\widetilde H^{2-s}(\Omega)\}\to\C$ is bounded, while expressions \eqref{LH1-H} and  \eqref{LH1*-H} define  bounded linear operators
 $
 A:H^s(\Omega)\to H^{s-2}(\Omega)
 $
and
 $
 A^*:H^{2-s}(\Omega)\to H^{-s}(\Omega)
 $, respectively
\end{theorem}
Note that for the particular important case $s=1$, the conditions on the coefficients in Theorem~\ref{T4.4} mean $a,b,c\in L_\infty(\Omega)$.

\section{Local  regularity of solutions to elliptic systems with H\"older-Lipschitz coefficients}\label{SolReg}

In this section we extend the well known result about the local regularity of elliptic PDE solutions, to the case of relaxed smoothness of the PDE coefficients. This will be used  then to prove counterparts of \cite[Theorems 3.12 and 3.16]{MikJMAA2011} in Section~\ref{S3}.

The local  regularity of solutions to elliptic PDEs \eqref{2.1-H} and \eqref{2.1*-H} for the case of infinitely smooth coefficients is well known (see e.g. \cite{Schwartz1958, Agmon1965, LiMa1}). For non-infinitely smooth coefficients, the case $a,b,c\in C^{k,1}(\bar\Omega)$, $s_1=1$, $s_2= k$ with integer $k\ge 0$ can be found in \cite[Theorem 4.16]{McLean2000-H}, and the case $a\in C^{0,1}(\bar\Omega)$, $b=0$, $c=const$, $s_2\in (-3/2,-1/2)$ in \cite[Theorem 4]{Savare1998}, extended  in \cite{Agranovich2007faa} to general elliptic systems with all coefficients from $C^{0,1}(\bar\Omega)$.  In Theorems \ref{RegTh} and \ref{RegThInf} below we prove the local regularity results for arbitrary H\"older coefficients and wider ranges of the Sobolev space indices $s_1$ and $s_2$.

Let us define the matrix function $\mathcal A(x,\xi):=\sum_{i,j=1}a_{ij}(x)\xi_i\xi_j$ for $\xi\in\R^n$.
The partial differential operator $A$ is elliptic in the sense of Petrovsky at a point $x$, where the coefficients $a_{ij}^{kl}(x)$ are defined, if
$
\det\mathcal A(x,\xi)\not=0
$
for any non-zero $\xi\in\R^n$ (see e.g. \cite[Section 55]{Miranda1970}),  evidently implying
$
|\det\mathcal A(x,\xi)|\ge C(x)|\xi|^{2m}$ for all $\xi\in\R^n
$
 with some positive $C(x)$, which in turn gives the following estimate for the matrix norm $|\cdot|$ of the inverse matrix $\mathcal A^{-1}(x,\xi)$,
\be\label{PE3}
|\mathcal A^{-1}(x,\xi)|\le C_0(x)|\xi|^{-2}\quad\forall\ \xi\in\R^n
\ee
with some $C_0(x)>0$.
We say that the operator $A$ is elliptic in a domain if it is  elliptic at each point of the domain.

Note that we will need the ellipticity in this paper only in  proving solution regularity in Theorems \ref{RegTh} and \ref{RegThInf}, which will be then used only to prove equivalence of the strong and canonical co-normal derivatives in Section~\ref{S3}.

Differentiation or Nirenberg difference quotient arguments are employed usually in the solution regularity analysis in \cite{Nirenberg1955, Schwartz1958, Agmon1965, LiMa1}, but  we will also need for our proof some powers of the Bessel potential operator $\J$ to deal with the H\"older-smooth coefficients along with the solution and the right hand side in some range of Sobolev spaces and  have to prove  first Lemma~\ref{Jcomm} and Corollary~\ref{JLcomm} about commutators.

\begin{lemma}\label{Jcomm}
Let $s$ be real, $k$ be integer,   $w\in H^s(\R^n)^m$, $g\in W^{\sigma+\ve}_\infty(\R^n)^m$, $\sigma=\left|s-k+\frac{1}{2}\right|+|k|+\frac{1}{2}$ and $\ve\in\R_+(\sigma)$. Then $\J^{2k}(gw)-g\J^{2k} w\in { H^{s
-2k+1}(\R^n)}^m$ and
\begin{eqnarray}
 \|\J^{2k}(gw)-g\J^{2k} w\|_{ H^{s -2k+1}(\R^n)^m}&\le&
 C|k|\ \|g\|_{W^{\sigma+\ve}_\infty(\R^n)^m}\|w\|_{H^s(\R^n)^m}.\label{DJt}
\end{eqnarray}
\end{lemma}
\begin{proof}
The proof below is given for $m=1$,  generalization to the vector case, $m>1$, is evident. For $k=0$ the lemma is trivial. Let now $k>0$.
Denoting the Fourier convolution by $*$ we have due to\eqref{BPD},
\begin{multline*}
K(\xi):=\F[\J^{2k}(gw)-g\J^{2k} w](\xi)=(1+|\xi|^2)^{k}(\widehat g  * \widehat w)(\xi)-(\widehat g *\F[\J^{2k} w])(\xi)=\\
\int_{\R^n}[(1+|\xi|^2)^{k}-(1+|\xi-\eta|^2)^{k}]\widehat g (\eta)
\widehat w(\xi-\eta)d\eta=
\int_{\R^n}[(\eta\cdot\xi+\eta\cdot(\xi-\eta)]f_k(\xi,\xi-\eta)\widehat g (\eta)  \widehat w(\xi-\eta)d\eta\\
=\frac{1}{2\pi i}\int_{\R^n}\widehat{\nabla g}(\eta)\cdot(\xi + \xi-\eta)
f_k(\xi,\xi-\eta)\widehat w(\xi-\eta)d\eta,
\end{multline*}
where
 $$
 f_k(\xi,\xi-\eta):=\frac{(1+|\xi|^2)^{k}-(1+|\xi-\eta|^2)^{k}}{|\xi|^2-|\xi-\eta|^2}
 =\frac{p^{2k}(\xi)-p^{2k}(\xi-\eta)}{p^2(\xi)-p^2(\xi-\eta)}
 =\sum_{j=1}^kp^{2(k-j)}(\xi)p^{2(j-1)}(\xi-\eta)
 $$
and $p(\zeta):=(1+|\zeta|^2)^{1/2}$.  This implies
\begin{align*}
K(\xi)&= \frac{1}{2\pi i}\sum_{j=1}^kp^{2(k-j)}(\xi)\left[ \xi\cdot\int_{\R^n}\widehat{\nabla g}(\eta)p^{2(j-1)}(\xi-\eta)\widehat w(\xi-\eta)d\eta
+\int_{\R^n}\widehat{\nabla g}(\eta)\cdot(\xi-\eta)p^{2(j-1)}(\xi-\eta)\widehat w(\xi-\eta)d\eta\right]
\\
&= \frac{-1}{4\pi^2}\sum_{j=1}^kp^{2(k-j)}(\xi)\mathcal F\left[
\nabla\cdot\left\{({\nabla g})\J^{2(j-1)}w\right\}
+{\nabla g}\cdot\J^{2(j-1)}\nabla w\right](\xi).
\end{align*}
Taking into account Theorem~\ref{GrL}, we obtain,
\begin{align*}
\|\J^{2k}(gw)-g\J^{2k}w\|_{ H^{s -2k+1}(\R^n)}&=\|p^{s -2k+1} K\|_{ L_2(\R^n)}
\\
&
= \frac{1}{4\pi^2}\left\|\sum_{j=1}^kp^{s+1-2j}\mathcal F\left[
\nabla\cdot\left\{({\nabla g})\J^{2(j-1)}w\right\}
+{\nabla g}\cdot\J^{2(j-1)}\nabla w\right]\right\|_{L_2(\R^n)}\\
 &\le \frac{1}{4\pi^2}\sum_{j=1}^k\left\|
\nabla\cdot\left\{({\nabla g})\J^{2(j-1)}w\right\}
+{\nabla g}\cdot\J^{2(j-1)}\nabla w\right\|_{H^{s+1-2j}(\R^n)}\\
&\le C_1\sum_{j=1}^k\left[
\|g\|_{W^{|s+2-2j|+1+\ve_1}_\infty(\R^n)}
+\|g\|_{W^{|s+1-2j|+1+\ve_1}_\infty(\R^n)}\right] \|w\|_{H^{s}(\R^n)}.
\end{align*}
for any $\ve_1\in\R_+(s)$. That is,
 \be\label{DJt1}
 \|\J^{2k}(gw)-g\J^{2k}w\|_{ H^{s -2k+1}(\R^n)} \le C_1k(\|g\|_{W^{|s |+1+\ve_1}_\infty(\R^n)} + \|g\|_{W^{|s -2k+1|+1+\ve_1}_\infty(\R^n)})\|w\|_{H^s(\R^n)}.
 \ee

Let now $k<0$.
If we denote $v=\J^{2k}w\in
H^{s-2k}(\R^n)$, then by inequality \eqref{DJt1}, where $2k$ is
replaced with $-2k$ and  $s-2k$ with $s$, we obtain,
\begin{multline*}
\|\J^{2k}(gw)-g\J^{2k}w\|_{ H^{s-2k+1}(\R^n)}=\|\J^{2k}[g\J^{-2k}v-\J^{-2k}(gv)]\|_{ H^{s-2k+1}(\R^n)}=\|g\J^{-2k}v-\J^{-2k}(gv)\|_{ H^{s+1}(\R^n)}\\
\shoveleft
{\le  C_1|k|(\|g\|_{W^{|s-2k|+1+\ve_1}_\infty(\R^n)} + \|g\|_{W^{|s+1|+1+\ve_1}_\infty(\R^n)})\|v\|_{H^{s-2k}(\R^n)}}\\
= C_1|k|(\|g\|_{W^{|s-2k|+1+\ve_1}_\infty(\R^n)} +
\|g\|_{W^{|s+1|+1+\ve_1}_\infty(\R^n)})\|w\|_{H^s(\R^n)}.
\end{multline*}

Inequality \eqref{DJt} follows for both positive and negative $k$ if we remark that
$$ \sigma:=\left|s-k+\frac{1}{2}\right|+|k|+\frac{1}{2}
=\begin{cases}\max(|s|+1,\ |s-2k+1|+1),&k> 0\\
        \max( |s-2k|+1,\ |s+1|+1)&k< 0
\end{cases}.$$
\end{proof}

Let us denote by $A_0$ the principal divergence part of the operator $A$ from \eqref{2.1-H},  i.e.,
\be
\label{2.1_0}
A_0u(x):=
 -\sum\limits_{i,j=1}^n\pa_i[  a_{ij}(x)\,\pa_j u(x)].
\ee

Bearing in mind that the Bessel potential operators $\J^{2k}$ commutate with
differentiation, Lemma~\ref{Jcomm} implies the following assertion.
\begin{corollary}\label{JLcomm}
Let $s$ be real, $k$ be integer, $u\in H^s(\R^n)^m$, $a_{ij}\in
W^{\sigma+\ve}_\infty(\R^n)^{m\times m}$,
$\sigma=\left|s-k-\frac{1}{2}\right|+|k|+\frac{1}{2}$, $\ve\in\R_+(\sigma)$.  Then $\J^{2k}(A_0u)-A_0\J^{2k}
u\in { H^{s-2k-1}(\R^n)^m}$ and
\begin{eqnarray*}
 \|\J^{2k}(A_0u)-A_0\J^{2k} u\|_{ H^{s-2k-1}(\R^n)^m}&\le&
 C|2k|\|a\|_{W^{\sigma+\ve}_\infty(\R^n)^m}\|u\|_{H^s(\R^n)^m}.
\end{eqnarray*}
\end{corollary}

If $\Omega$ is an open set while a set $\Omega'$ is such that $\ov{\Omega'}\subset\Omega$, we will denote this as  $\Omega'\Subset\Omega$.

Now we are in a position to prove the following local regularity theorem.

\begin{theorem} \label{RegTh}
Let $\Omega$ be an open set in $\R^n$, $s_1\in\R$, $m\ge 1$, $u\in H^{s_1}_{\mathrm{loc}}(\Omega)^m$, $f\in H^{s_2}_{\mathrm{loc}}(\Omega)^m$, $s_2>s_1-2$. If $u$ satisfies either
\\
(a) elliptic (in the sense of Petrovsky) system \eqref{2.1-H}, $Au=f$, in
$\Omega$ with $\{a,b,c\}\in {\cal C}_{+\mathrm{loc}}^{s_1-1}(\Omega)\bigcap {\cal C}_{+\mathrm{loc}}^{s_2+1}(\Omega)$ or
\\
(b)  elliptic (in the sense of Petrovsky) system \eqref{2.1*-H}, $A^*u=f$, in
$\Omega$ with $\{a,b,c\}\in {\cal C}_{+\mathrm{loc}}^{1-s_1}(\Omega)\bigcap {\cal C}_{+\mathrm{loc}}^{-s_2-1}(\Omega)$,

\noindent then $u\in H^{s_2+2}_{\mathrm{loc}}(\Omega)^m$.
\end{theorem}
\begin{proof}
Note that the theorem hypothesis $s_2>s_1-2$ implies that either $s_1\not=1$ or $s_2\not=-1$ and thus $a\in \bar C^\mu_{\mathrm{loc}}(\Omega)$ for some $\mu>0$ and particularly, $a\in C(\Omega)$ (maybe after adjusting $a$ on a zero measure set, that we will assume to be done).
We give a proof only for part (a) of the theorem, organized in several steps, for part (b) it is similar.
\paragraph{Step (0)}
As usual, cf. e.g. \cite[Chapter 2, Theorem 3.1 ]{LiMa1-H}, let us first consider the case $a=const$, $b=0$, $c=0$ and $\Omega=\R^n$. Suppose a function $U$ satisfies equation \eqref{2.1-H}. Application of the Fourier transform reduces this equation to
$(2\pi)^2\mathcal A(\xi)\widehat U(\xi)=\widehat f(\xi)$. Resolving it for $\widehat U$ and applying ellipticity estimate \eqref{PE3}, we obtain
 $
 (1+|\xi|^2) |\widehat U(\xi)|\le C_1|\widehat f(\xi)|+|\widehat U(\xi)|
 $
with $C_1=(2\pi)^{-2}C_0$,
implying
 \begin{equation}\label{estu}
 \|U\|_{H^{s+2}(\R^n)}\le
 C_1\|f\|_{H^{s}(\R^n)}+\|U\|_{H^{s}(\R^n)}\quad \forall s\in\R.
 \end{equation}
\paragraph{Step (i)} Let now the coefficients $\{a,b,c\}\in {\cal C}_{+\mathrm{loc}}^{s_1-1}(\Omega)\bigcap {\cal C}_{+\mathrm{loc}}^{s_2+1}(\Omega)$ be not generally constant, $\Omega$ be not generally $\R^n$, and $u\in H^{s_1}_{\mathrm{loc}}(\Omega)$.
Let $B_{\rho}=B_{y,\rho}\subset\Omega'\Subset\Omega$ be an open ball of radius $\rho$ centered at a point $y\in\Omega$. Let $a$, $b$, $c$ and
$u$ 
be extended outside $\Omega'$ to $\{a^e,b^e,c^e\}\in {\cal C}_{+}^{s_1-1}(\R^n)\bigcap {\cal C}_{+}^{s_2+1}(\R^n)$ and $u^e\in H^{s_1}(\R^n)$, and we will further drop the superscript $e$ for brevity.

Let $\eta \in \D(B_{\rho})$ be a cut-off function such that  $\eta (x)=1$ in $B_{\rho/2}$. Then $U_\eta(x):=\eta (x)u(x)$ belongs to $H^{s_1}(\R^n)$, is compactly supported in $B_{\rho}$ and satisfies equation
 \begin{equation}\label{Lx0}
A_{0y}U_\eta=\eta f+A_{\eta}u-A_0^-U_\eta\quad\text{in}\ \R^n.
\end{equation}
Here $A_{0y}$ is the principal part of the operator with the coefficient matrix $a(y)$, thus constant in $x$, i.e.,
\bea
A_{0y}U_\eta&:=&-\sum_{i,j=1}^n a_{ij}(y)\pa_i\pa_j U_\eta,\\
A_{\eta}u&:=& -\sum_{i,j=1}^n (\pa_i\eta )a_{ij}\pa_j u
 -\sum_{i,j=1}^n\pa_i[(\pa_j\eta )a_{ij} u]
 -\sum_{j=1}^n\eta b_j \pa_ju
 - \eta cu,\label{Leta}\\
 A_0^-U_\eta&:=&-\sum_{i,j=1}^n \pa_i(a_{ij}^-\pa_j U_\eta),
\label{L0}\eea
and $a^-(x):=a(x)-a(y)$.
Let
\be\label{range2+3}
s_2+1\le s_1<s_2+2,
\ee
see Fig.~\ref{zones}.
\begin{figure}[htb]
\centerline{%
\psfig{file=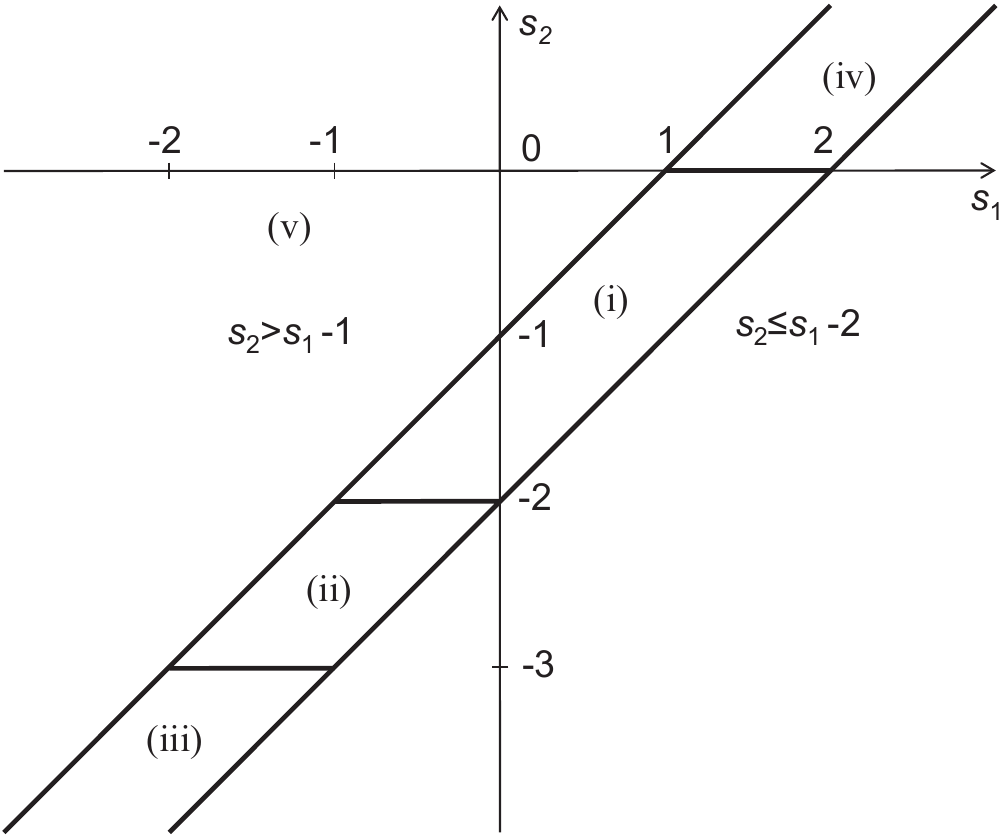,height=7cm}}%
\caption{\rm Regions of parameters $s_1$, $s_2$.}
 \label{zones}
 \end{figure}
 Then by Theorem~\ref{GrL},
\begin{multline}
\label{Lrho}
\|A_{\eta}u\|_{H^{s_2}(\R^n)}
\le C_2\left[\|(\nabla\eta )a\nabla u \|_{H^{s_2}(\R^n)}
+\| (\nabla\eta )a u\|_{H^{s_2+1}(\R^n)}
+\|\eta b\nabla u \|_{H^{s_2}(\R^n)}
+\|\eta c u \|_{H^{s_2}(\R^n)}\right]
\\
\le
CC_2\left[\|\nabla\eta \|_{W_\infty^{|s_1-1|+\ve_1}(\R^n)}\|a\nabla u\|_{H^{s_1-1}(\R^n)}
+\|\nabla\eta \|_{W_\infty^{|s_2+1|+\ve_2}(\R^n)}\|a u\|_{H^{s_2+1}(\R^n)}\right.
\\
+\left.\|\eta \|_{W_\infty^{|s_2|+\ve_2}(\R^n)}\|b\nabla u\|_{H^{s_2}(\R^n)}
+\|\eta \|_{W_\infty^{|s_2|+\ve_2}(\R^n)}\|c u\|_{H^{s_2}(\R^n)}
\right]
\le C_3(\eta)\|u \|_{H^{s_1}(\R^n)},
\end{multline}
\begin{eqnarray}
C_3(\eta)&:=&
 CC_2C'\left[\|\nabla\eta \|_{W_\infty^{|s_1-1|+\ve_1}(\R^n)}
  \|a\|_{W_\infty^{|s_1-1|+\ve_1}(\R^n)}
 +\|\nabla\eta \|_{W_\infty^{|s_2+1|+\ve_2}(\R^n)}
  \|a\|_{W_\infty^{|s_2+1|+\ve_2}(\R^n)}\right.\nonumber\\
&&\left.+\|\eta \|_{W_\infty^{|s_2|+\ve_2}(\R^n)}\|b\|_{W_\infty^{\mu^0_b+\ve^0_b}(\R^n)}
+\|\eta \|_{W_\infty^{|s_2|+\ve_2}(\R^n)}\|c\|_{W_\infty^{\mu^0_c+\ve^0_c}(\R^n)}
\right],\label{C3}
\end{eqnarray}
\begin{eqnarray*}
\mu^0_b&:=&\min\{|s| : s_2\le s\le s_1-1\}=\max\{s_2, 1-s_1, 0\}
=\max\{\mu_b(s_1-1),\,\mu_b(s_2+1)\},
\\
\mu^0_c&:=&\min\{|s| : s_2\le s\le s_1\}=\max\{s_2, -s_1, 0\}
=\max\{\mu_c(s_1-1),\,\mu_c(s_2+1)\},
\end{eqnarray*}
by Definition~\ref{DC^sigma} and condition \eqref{range2+3},
while by the theorem hypothesis there exist $\ve_1\in\R_+(s_1)$, $\ve_2\in\R_+(s_2)$, $\ve^0_b\in \R_+(\mu^0_b)$, $\ve^0_c\in \R_+(\mu^0_c)$ such that the norms of the coefficients $a,b,c$ are bounded in \eqref{C3}.

Let us assume the condition
 \be \label{case-a}
 |s_2+1|<1
 \ee
in addition to condition \eqref{range2+3}, which correspond to region (i) in Fig. \ref{zones}.

Let us define $a_0^-(x):=\begin{cases}
a^-(x),& x\in \bar B_{\rho}\\
a^-(x\rho/|x|),& x\not\in \bar B_{\rho}
\end{cases}$\ .\\
Then it is easy to see that\\
$\|a_0^-\|_{ W_\infty^{|s_2+1|+\ve_2}(\R^n)}=
\|a_0^-\|_{ C^{|s_2+1|+\ve_2}(\R^n)}=
\|a^-\|_{ C^{|s_2+1|+\ve_2}(\bar B_{\rho})}$
for some $\ve_2$ such that
\be\label{ve2}
\ve_2\in\R_+(s_2),\quad |s_2+1|+\ve_2< 1.
\ee
Thus, since $\supp U_\eta\subset B$, we have from \eqref{L0}  by Theorem \ref{GrL},
\begin{multline}\label{Lyuest}
\|A_0^-U_\eta\|_{H^{s_2}(\R^n)}
 \le C_4 \| a^-\nabla U_\eta\|_{H^{s_2+1}(\R^n)}=C_4 \| a_0^-\nabla U_\eta\|_{H^{s_2+1}(\R^n)}\\
\le
 CC_4 \| a_0^-\|_{W_\infty^{|s_2+1|+\ve_2/2}(\R^n)}\| \nabla U_\eta\|_{H^{s_2+1}(\R^n)}
 \le CC_4 \| a^-\|_{C^{|s_2+1|+\ve_2/2}(\ov B_{\rho})}C_5\| U_\eta\|_{H^{s_2+2}(\R^n)}
 \end{multline}

Applying estimate \eqref{estu} to equation \eqref{Lx0} and taking into account estimates \eqref{Lrho} and  \eqref{Lyuest}, we then have under conditions \eqref{range2+3} and  \eqref{ve2},
\begin{eqnarray}\label{estu21}
 C_6(\rho)\|U_\eta\|_{H^{s_2+2}(\R^n)}
 &\le& C_7(\eta)\|f\|_{H^{s_2}(B_{\rho})}
 +C_8(\eta)\|u \|_{H^{s_1}(B_{\rho})},
\end{eqnarray}
 $$
 C_6(\rho):=1-C_1CC_4C_5\| a^-\|_{C^{|s_2+1|+\ve_2/2}(\bar B_{\rho})},\
 C_7(\eta):= C_1C\|\eta \|_{\bar C^{|s_2|+\ve_2}(\bar B_{\rho})},\
 C_8(\eta):=C_1C_3(\eta)+  C_7(\eta).
 $$
The parameter $C_7(\eta)$ and, due to the theorem hypotheses, also  $C_3(\eta)$ and thus $C_8(\eta)$ are  finite for any $\rho\in (0,\infty)$.
We will prove that $C_6(\rho)$ is positive for sufficiently small $\rho$ under conditions \eqref{range2+3}, \eqref{case-a}.

Let first  $s_2=-1$, and consider estimate \eqref{estu21} with
$\ve_2=0$.  Since  $a^-(y)=0$ and $a^-$ is continuous in $\bar B_{\rho}$, for any sufficiently small  $\rho>0$, the norm $\| a^-\|_{C^{|s_2+1|+\ve_2/2}(\bar B_{\rho})}=\|a^-\|_{C(\bar B_{\rho})}$ becomes small enough for $C_6(\rho)$ in \eqref{estu21} to be positive.

Let now  $0<|s_2+1|<1$.  Due to the theorem hypothesis, there exists $\ve_2\in (0,1-|s_2+1|)$ such that $a^-\in C^{|s_2+1|+\ve_2}(\bar B_{\rho})$, which implies the following estimate,
\beas
\|a^-\|_{C^{|s_2+1|+\ve_2/2}(\bar B_{\rho})}&=&\|a^-\|_{C(\bar B_{\rho})}+
 |a|_{C^{|s_2+1|+\ve_2/2}(\bar B_{\rho})}\le
 \|a^-\|_{C(\bar B_{\rho})}+
 (2\rho)^{\ve_2/2} |a|_{C^{|s_2+1|+\ve_2}(\bar B_{\rho})}, \\
 |a|_{C^{|s_2+1|+\ve_2}(\bar B_{\rho})}&:=&\sup_{x,x'\in \bar B_{\rho}}\frac{|a(x)-a(x')|}{|x-x'|^{|s_2+1|+\ve_2}}\le |a|_{C^{|s_2+1|+\ve_2}(\bar \Omega')}<\infty.
 \eeas
Thus again  for any sufficiently small  $\rho>0$, the norm $\| a^-\|_{C^{|s_2+1|+\ve_2/2}(\bar B_{\rho})}$ becomes small enough for $C_6(\rho)$ in \eqref{estu21} to be positive.

This means $U_\eta\in H^{s_2+2}(\R^n)$ implying $u\in H^{s_2+2}(B_{y,\rho(y)/2})$ for arbitrary point $y\in\Omega$
under conditions \eqref{range2+3}, \eqref{case-a}.
Thus any compact subdomain  $\bar\Omega'$ of the open domain $\Omega$ has an open cover by the balls $B_{y,\rho(y)/2}$ such that $u\in H^{s_2+2}(B_{y,\rho(y)/2})$. Due to the compactness of $\bar\Omega'$, there exists a finite subset of the balls, $B^j:=B_{y^{j},\rho(y^{j})/2}$, $j=1,2,...,J$, still covering $\bar\Omega'$. Let $\{\varphi_j(x)\in\D(B^j)\}_{j=1}^J$ be a partition of unity, $\sum_{j=1}^J\varphi_j(x)=1$ for any $x\in\Omega'$ and $U_j\in H^{s_2+2}(\R^n)$ be such that $U_j=u$ on $B^j$ and $\|U^j\|_{H^{s_2+2}(\R^n)}=\|u\|_{H^{s_2+2}(B^j)}$. Then by Theorem~\ref{GrL},
\begin{multline*}
\|u\|_{H^{s_2+2}(\Omega')}=\|\sum_{j=1}^J\varphi_ju\|_{H^{s_2+2}(\Omega')}
=\|\sum_{j=1}^J\varphi_jU^j\|_{H^{s_2+2}(\Omega')}
\le\sum_{j=1}^J\|\varphi_jU^j\|_{H^{s_2+2}(\R^n)}\\
\le C\sum_{j=1}^J\|\varphi_j\|_{W_\infty^\mu(\R^n)}\| U^j\|_{H^{s_2+2}(\R^n)}
= C\sum_{j=1}^J\|\varphi_j\|_{W_\infty^\mu(\R^n)}\|u\|_{H^{s_2+2}(B^j)},
\end{multline*}
for any $\mu>|s_2+2|$. Thus $u\in H^{s_2+2}(\Omega')$ for any compact $\bar\Omega'\subset\Omega$, implying $u\in H^{s_2+2}_{\mathrm{loc}}(\Omega)$ under conditions \eqref{range2+3}, \eqref{case-a}.

\paragraph{Step (ii)} Let us prove the theorem under conditions
$
s_2+1\le s_1<s_2+2,\quad -3<s_2\le-2,
$
 that are satisfied in region (ii) in Fig.~\ref{zones}. We proceed as in Step (i) but instead of estimate \eqref{Lyuest} for the term  $A_0^-U_\eta$ we split it into two parts
$$
A_0^-U_\eta=A_{01}^-U_\eta+A_{02}^-U_\eta, \quad
A_{01}^-U_\eta:=\sum_{i,j=1}^n \pa_i((\pa_ja_{ij}^-) U_\eta), \quad
A_{02}^-U_\eta:=-\sum_{i,j=1}^n \pa_i\pa_j(a_{ij}^- U_\eta)
$$
and estimate each of them as follows,
 \begin{multline*}
\|A_{01}^-U_\eta\|_{H^{s_2}(\R^n)}
\le C_4 \| (\nabla a^-)U_\eta\|_{H^{s_2+1}(\R^n)}
\le C_4 \| (\nabla a^-)U_\eta\|_{H^{s_1}(\R^n)}\\
\le
CC_4 \|a^-\|_{W_\infty^{|s_1|+1+\ve_1/2}(\R^n)}\|U_\eta\|_{H^{s_1}(\R^n)}
=
CC_4 \|a^-\|_{W_\infty^{|s_1-1|+\ve_1/2}(\R^n)}\|U_\eta\|_{H^{s_1}(\R^n)}
 \end{multline*}
where we have taken into account that $s_1<0$ in region (ii), and
\begin{multline*}
\|A_{02}^-U_\eta\|_{H^{s_2}(\R^n)}
\le C_4 \| a^- U_\eta\|_{H^{s_2+2}(\R^n)}
\le CC_4 \| a_0^- U_\eta\|_{H^{s_2+2}(\R^n)}
\le CC_4 \| a^-\|_{C^{|s_2+2|+\ve_2/2}(\ov B_{\rho})}\| U_\eta\|_{H^{s_2+2}(\R^n)}
 \end{multline*}
Taking into account that  $\|A_{02}^-U_\eta\|_{H^{s_2}(\R^n)}$  can be made arbitrarily small by choosing sufficiently small ball radius $\rho$, as in Step (i), since $0\le |s_2+2|<1$, this proves the theorem for region (ii).

\paragraph{Step (iii)} Let us prove the theorem under conditions
\be\label{iiic}
s_2+1\le s_1<s_2+2,\quad s_2\le-3,
\ee
 that are satisfied in region (iii) in Fig.~\ref{zones}.
For arbitrary  $\Omega'\Subset\Omega$ let $\eta\in C^\infty(\Omega)$ with $\supp\eta\in\Omega$ 
and $\eta=1$ in $\Omega'$. Then the function $U_\eta=\eta u\in H^{s_1}(\R^n)$ satisfies equation
\be\label{f+Leta}
A_0U_\eta=f_\eta,\quad f_\eta=\eta f+A_{\eta}u \quad\mbox{in } \R^n,
 \ee
where
$A_0$ is given by \eqref{2.1_0}, $A_{\eta}$ by \eqref{Leta},
while $A_{\eta}u\in H^{s_2}( \R^n)$ by estimate \eqref{Lrho}.
This implies $f_\eta\in H^{s_2}( \R^n)$.

Let $\displaystyle k=-\left\lfloor\frac{-1-s_2}{2}\right\rfloor=\left\lceil\frac{1+s_2}{2}\right\rceil$ and let us denote $v:=\J^{2k} U_\eta$. Then $k\le -1$ by the second condition in \eqref{iiic}, while $v\in H^{s_1-2k}(\R^n)$. Acting by $\J^{2k}$ on \eqref{f+Leta}, we arrive at the following equation for $v$
 \bea\label{JLv}
 A_0(v)=f_v, \quad\mbox{in } \R^n,
 \eea
where
$
 f_v=\J^{2k}f_\eta- [\J^{2k} A_0 u-A_0\J^{2k} u].
$
To employ Corollary~\ref{JLcomm} with $s=s_1$, we have for its parameter,
$$
\sigma=\left|s_1-k-\frac{1}{2}\right|+|k|+\frac{1}{2}
=-s_1+k+\frac{1}{2}+|k|+\frac{1}{2}=1-s_1
$$
since
$\displaystyle 0<-k\le\frac{-1-s_2}{2}< \frac{1-s_1}{2}$
due to the first condition in \eqref{iiic}. Then by the theorem hypothesis on the coefficients, the conditions of Corollary~\ref{JLcomm}  are satisfied, which implies $[\J^{2k} A_0 u-A_0\J^{2k} u]\in H^{s_1-2k-1}(\R^n)$. Then taking into account the first condition in \eqref{iiic} again, we obtain $f_v\in H^{s_2-2k}(\R^n)$. Denoting $s'_1=s_1-2k$, $s'_2=s_2-2k$, we arrive at equation \eqref{JLv} for $v\in H^{s_1'}(\R^n)$ with $f_v\in H^{s_2'}(\R^n)$, where
$
s'_2+1\le s'_1<s'_2+2,\quad -3<s'_2\le 1,
$
and coefficients
$a\in \bar{C}^{|s_1-1|}_+(\overline\Omega)\bigcap\bar{C}^{|s_2+1|}_+(\overline\Omega)\subset
\bar{C}^{|s'_1-1|}_+(\overline\Omega)\bigcap\bar{C}^{|s'_2+1|}_+(\overline\Omega) $,
which is covered by Steps (i) and (ii) implying $v\in H^{s_2'+2}(\R^n)=H^{s_2+2-2k}(\R^n)$. Thus, $U_\eta:=\J^{-2k}v\in H^{s_2+2}(\R^n)$.
This gives $ u\in H^{s_2+2}(\Omega')$, which implies the theorem claim in region (iii).

\paragraph{Step (iv)} Let us prove the theorem under conditions
$
s_2+1\le s_1<s_2+2,\quad s_2\ge0,
$
 that are satisfied in region (iv) in Fig. \ref{zones}.
Let $\alpha$ be a multiindex such that $|\alpha|=\lfloor s_2\rfloor+1$. Then
\eqref{f+Leta} implies
\be\label{alphaf}
A_0\partial^\alpha U_\eta=\pa^\alpha f_\eta+ f^\alpha_{\eta},\quad
f^\alpha_{\eta}=A_0\partial^\alpha U_\eta-\partial^\alpha A_0 U_\eta.
 \ee
Since $f^\alpha_{\eta}$ is a commutator, we obtain that
$f^\alpha_{\eta}\in H^{s_1-|\alpha|-1}(R^n)\subset H^{s_2-|\alpha|}(R^n)$,
where the theorem hypothesis on smoothness of the coefficient matrix $a$ and Theorem~\ref{GrL} were taken into account. Then $\pa^\alpha f_\eta+ f^\alpha_{\eta}\in H^{s_2-\alpha}(R^n)$ giving
$\partial^\alpha U_\eta\in H_{\mathrm{loc}}^{s_2-|\alpha|+2}(R^n)$ by Step (i), which implies $ u\in H^{s_2+2}_{\rm loc}(\Omega)$, i.e. the theorem claim for region (iv).

\paragraph{Step (v)} Now we finally prove the theorem for
  $s_2> s_1-1$,
i.e. for region (v).
Since $f\in H^{s_2}(\Omega')$ on any open set $\Omega'\Subset\Omega$, we have also $f\in H^{s_1-1}(\Omega')$,  i.e., we arrive at the situation covered by Steps (i)-(iv) with $s_2=s_1-1$, which implies $u\in H^{s_1+1}(\Omega')$. If $s_1\le s_2$, we iterate this procedure, obtaining at the end $u\in H^{s_2+2}(\Omega')$, i.e. the theorem claim, if  $s_2-s_1$ is integer, or $u\in H^{s_1+k}(\Omega')$, where $k=\lfloor s_2-s_1+2\rfloor$, otherwise. Recalling in the latter case that $f\in H^{s_2}(\Omega')$  we can apply the corresponding steps from (i)-(iv) again, which finishes the proof for region (v).
\end{proof}

Theorem~\ref{RegTh} gives solution regularity on any sub-domain $\Omega'$ with compact closure $\bar\Omega'\subset\Omega$. The following theorem generalizes it to sub-domain $\Omega'$ with non-compact closure $\bar\Omega'\subset\Omega$ and particularly proves regularity at infinity for exterior (unbounded) domains.

\begin{theorem}\label{RegThInf}
 Let $\Omega$ be $\R^n$ or an open exterior or interior domain with a compact boundary in  $\R^n$, $s_1\in\R$, $s_2>s_1-2$, $u\in H^{s_1}(\Omega')^m$ and $f\in H^{s_2}(\Omega')^m$ on any open set $\Omega'\Subset\Omega$, $m\ge 1$. Let $u$ satisfy either
\\
(a) elliptic (in the Petrovsky sense) system \eqref{2.1-H}, $Au=f$, in $\Omega$ with 
$\{a,b,c\}\in {\cal C}_+^{s_1-1}(\R^n)\bigcap {\cal C}_+^{s_2+1}(\R^n)$
or
\\
(b) elliptic (in the Petrovsky sense) system \eqref{2.1*-H}, $A^*u=f$, in $\Omega$ with 
$\{a,b,c\}\in {\cal C}_+^{1-s_1}(\R^n)\bigcap {\cal C}_+^{-s_2-1}(\R^n),$\\
and in the case of the infinite domain $\Omega$ there exist finite matrices $a_{ij}(\infty):=\lim_{x\to\infty}a_{ij}(x)$ satisfying the ellipticity condition  $\det\sum_{i,j=1}a_{ij}(\infty)\xi_i\xi_j\not=0$. Then $u\in H^{s_2+2}(\Omega')^m$ on any $\Omega'\Subset\Omega$.
\end{theorem}

\begin{proof} The theorem claim for  subdomains $\Omega'\Subset\Omega$ with compact closure is implied by Theorem~\ref{RegTh}. To complete the proof, we have to consider an infinite subdomain $\Omega'\Subset\Omega$ of an infinite domain $\Omega$. Note that the theorem hypothesis $s_2>s_1-2$ implies that either $s_1\not=1$ or $s_2\not=-1$ and thus $a\in \bar C^\mu(\Omega)$ for any $\Omega'\Subset\Omega$ for some $\mu>0$ and particularly, $a\in C(\Omega)$ (maybe after adjusting $a$ on a zero measure set, that we will assume to be done).

The proof follows the pattern of the proof of Theorem~\ref{RegTh} and  we will mostly refer to that proof instead of repeating it whenever possible. We give only a proof for part (a) of the theorem; the proof for part (b) is similar.
\paragraph{Step (i)} Let the coefficients $a,b,c$ be not generally constant, $\Omega$ be either $\R^n$ or an open exterior domain with a compact boundary in $\R^n$.
In the latter case let $u$ be extended outside $\Omega$ to $u^e\in H^{s_1}(\R^n)$, and we will further drop the superscript $e$ for brevity.
Let  $B_{\rho}=B_{0,\rho}$ be an open ball of radius $\rho$ centred at zero.
Let $\rho$ be sufficiently large, so that $B_\rho$ includes the boundary of $\Omega$ (if $\Omega\not=\R^n$). Let us chose a cut-off function $\eta \in C^\infty(\R^n)$ such that  $\eta (x)=1$ in $\R^n\backslash B_{2\rho}$ and $\eta (x)=0$ in $B_\rho$. Denoting $U_\eta(x):=\eta (x)u(x)$ we obtain that  $\supp U_\eta\subset\R^n\backslash B_{\rho}\subset\Omega$.

Then the function $U_\eta$ satisfies equation
 \begin{equation}\label{Lx0inf}
A_{0\infty}U_\eta=\eta f+A_{\eta}u-A_{\infty}^-U_\eta\quad\text{in}\ \R^n.
\end{equation}
Here $A_{0\infty}$ is the principal part of the operator with the constant coefficient matrix $a(\infty)$, i.e.,
\bea
A_{0\infty} U_\eta&:=&-\sum_{i,j=1}^n a_{ij}(\infty)\pa_i\pa_j U_\eta,\\
A_{\eta}u&=& -\sum_{i,j=1}^n (\pa_i\eta )a_{ij}\pa_j u
 -\sum_{i,j=1}^n\pa_i[(\pa_j\eta )a_{ij} u]
 -\sum_{j=1}^n\eta b_j \pa_ju
 - \eta cu,\label{Letainf}\\
 A_{\infty}^-U_\eta&=&-\sum_{i,j=1}^n \pa_i(a_{ij}^-\pa_j U_\eta),
\label{Linfty}\eea
where $a^-(x)=a(x)-a(\infty)$.

Let
\be\label{range2+3inf}
s_2+1\le s_1<s_2+2,
\ee
see Fig.~\ref{zones}. Then by Theorem~\ref{GrL}, we again, as in the proof of Theorem~\ref{RegTh}  arrive at estimate \eqref{Lrho}, where $C_3(\rho)$ is defined by \eqref{C3}.

Let us define $a_\infty^-(x)=\begin{cases}
a^-(x),& x\in \R^n\backslash B_{\rho}\\
\frac{|x|}{\rho}a^-\left(\frac{x\rho}{|x|}\right),& x\in \bar B_{\rho}
\end{cases}$.\\
Then evidently $\|a_\infty^-\|_{ C(\R^n)}=\|a^-\|_{ C(\R^n\backslash B_{\rho})}\to 0$ as $\rho\to\infty$.
Moreover, it is easy to check (see Appendix~\ref{Est}) that
$\|a_\infty^-\|_{ C^{\mu}(\R^n)}\le 3\|a^-\|_{ C^{\mu}(\R^n\backslash B_{\rho})}$ for any $\mu\in[0,1]$
and sufficiently large $\rho$, and
$\|a_\infty^-\|_{ C^{\mu}(\R^n)}\to 0$ as $\rho\to\infty$ if $a^-\in C^{\mu+\epsilon}(\Omega)$  for some $\varepsilon>0$.

Thus, since $\supp U_\eta\subset \R^n\backslash B_{\rho}$, we have from \eqref{Linfty} by Theorem \ref{GrL},
\begin{multline}\label{Lyuestinf}
\|A_{\infty}^-U_\eta\|_{H^{s_2}(\R^n)}
 \le C_4 \| a^-\nabla U_\eta\|_{H^{s_2+1}(\R^n)}=C_4 \| a_\infty^-\nabla U_\eta\|_{H^{s_2+1}(\R^n)}\\
\le
 CC_4 \| a_\infty^-\|_{C^{|s_2+1|+\ve_2/2}(\R^n)}\| \nabla U_\eta\|_{H^{s_2+1}(\R^n)}
 \le CC_4 \| a_\infty^-\|_{C^{|s_2+1|+\ve_2/2}(\R^n)}C_5\| U_\eta\|_{H^{s_2+2}(\R^n)}
 \end{multline}
for  any $\ve_2$ such that
\be\label{ve2inf}
\ve_2\in\R_+(s_2),\quad |s_2+1|+\ve_2/2< 1.
\ee

Applying estimate \eqref{estu} to equation \eqref{Lx0inf} and taking into account estimates \eqref{Lrho} and  \eqref{Lyuestinf}, we then have under conditions \eqref{range2+3inf} and  \eqref{ve2inf},
\begin{eqnarray}\label{estu21inf}
 C_6(\rho)\|U_\eta\|_{H^{s_2+2}(\R^n)}
 &\le& C_7(\rho)\|f\|_{H^{s_2}(\R^n\setminus \bar B_{\rho})}
 +C_8(\rho)\|u \|_{H^{s_1}(\R^n\setminus \bar B_{\rho})},
\end{eqnarray}
 $$
 C_6(\rho):=1-C_1 CC_4C_5\| a_\infty^-\|_{C^{|s_2+1|+\ve_2/2}(\R^n)},\
 C_7(\rho):= C_1C\|\eta \|_{\bar C^{|s_2|+\ve_2}(\R^n\setminus B_{\rho})},\
 C_8(\rho):=C_1C_3(\rho)+  C_7(\rho).
 $$
 The parameter $C_7(\rho)$ and, due to the theorem hypotheses, also  $C_3(\rho)$ and thus $C_8(\rho)$ are  finite  for any $\rho\in (0,\infty)$.

Further in this step we prove that $C_6(\rho)$ is positive for sufficiently large $\rho$ under conditions
 $s_2+1\le s_1<s_2+2$,
$ |s_2+1|<1,$
which correspond to region (i) in Fig. \ref{zones}.

Let first  $s_2=-1$, and consider estimate \eqref{estu21inf} with
$s_2+1=\ve_2=0$.  Since  $a^-(\infty)=0$,  the norm
$\|a^-_\infty\|_{C^{|s_2+1|+\ve_2/2}(\R^n)}=\|a^-_\infty\|_{C(\R^n)}$ for  sufficiently large  $\rho<\infty$ becomes small enough for $C_6(\rho)$ in \eqref{estu21inf} to be positive.

Let now  $0<|s_2+1|<1$.  Due to the theorem hypothesis, there exists $\ve_2\in (0,1-|s_2+1|)$ such that $a^-\in C^{|s_2+1|+\ve_2}(\R^n\setminus B_{\rho})$, which implies
$\| a^-_\infty\|_{C^{|s_2+1|+\ve_2/2}(\R^n)}\to 0$ as $\rho\to 0$.
Thus again  for  sufficiently large  $\rho$, the norm $\| a^-_\infty\|_{C^{|s_2+1|+\ve_2/2}(\R^n)}$ becomes small enough for $C_6(\rho)$ in \eqref{estu21} to be positive.

This means that in the both cases $U_\eta\in H^{s_2+2}(\R^n)$ implying $u\in H^{s_2+2}(\Omega'\setminus B_{2\rho})$ for sufficiently large $\rho$. Taking into account that  $u\in H^{s_2+2}(\Omega'\bigcap B_{3\rho})$ for any $\rho$ by Theorem~\ref{RegTh}, we arrive at the present theorem claim in region (i).

\paragraph{Step (ii)}
Let us prove the theorem under conditions
$
s_2+1\le s_1<s_2+2,\quad -3<s_2\le-2,
$
 that are satisfied in region (ii) in Fig.~\ref{zones}.
 We proceed as in Step (i) but instead of estimate \eqref{Lyuestinf} for the term  $A_\infty^-U_\eta$ we split it into two parts
$$
A_\infty^-U_\eta=A_{\infty 1}^-U_\eta+A_{\infty 2}^-U_\eta, \quad
A_{\infty 1}^-U_\eta:=\sum_{i,j=1}^n \pa_i((\pa_ja_{ij}^-) U_\eta), \quad
A_{\infty 2}^-U_\eta:=-\sum_{i,j=1}^n \pa_i\pa_j(a_{ij}^- U_\eta)
$$
and estimate each of them as follows,
 \begin{multline*}
\|A_{\infty 1}^-U_\eta\|_{H^{s_2}(\R^n)}
\le C_4 \| (\nabla a^-)U_\eta)\|_{H^{s_2+1}(\R^n)}
\le C_4 \| (\nabla a^-)U_\eta\|_{H^{s_1}(\R^n)}\\
\le
CC_4 \|a^-\|_{W_\infty^{|s_1|+1+\ve_1/2}(\R^n)}\|U_\eta\|_{H^{s_1}(\R^n)}
=
CC_4 \|a^-\|_{W_\infty^{|s_1-1|+\ve_1/2}(\R^n)}\|U_\eta\|_{H^{s_1}(\R^n)}
 \end{multline*}
where we took into account that $s_1<0$ in region (ii), and
\begin{multline*}
\|A_{\infty 2}^-U_\eta\|_{H^{s_2}(\R^n)}
\le C_4 \| a^- U_\eta\|_{H^{s_2+2}(\R^n)}
\le CC_4 \| a_0^- U_\eta\|_{H^{s_2+2}(\R^n)}
\le CC_4 \| a^-\|_{C^{|s_2+2|+\ve_2/2}(\ov B_{\rho})}\| U_\eta\|_{H^{s_2+2}(\R^n)}
 \end{multline*}
Taking into account that  $\|A_{\infty 2}^-U_\eta\|_{H^{s_2}(\R^n)}$  can be made arbitrarily small by choosing sufficiently large $\rho$, as in Step (i), since $0\le |s_2+2|<1$, this proves the theorem for region (ii).

\paragraph{Steps (iii)-(v)}
The proofs of the theorem under condition  $s_2<-3$ and under condition  $s_2\ge 0$, in addition to condition \eqref{range2+3inf} coincide word-for-word with the proof in Steps (iii) and (iv), respectively, of Theorem~\ref{RegTh},  while for $s_2>s_1-1$ with the proof in Step (v) of the same theorem.

\end{proof}

\begin{remark}
Conditions on the PDE coefficients in  Theorem~\ref{RegThInf}  can be evidently relaxed to the corresponding conditions for all domains $\Omega'\Subset \Omega$ (implying that the coefficients are extendable from such $\Omega'$ to the whole $\R^n$ such that the conditions hold) supplemented with the continuity of the coefficient $a$ at infinity for the extensions.
\end{remark}

\begin{remark}
The Theorem~\ref{RegThInf} proof works also for domains $\Omega$ with a non-compact boundary and for an open set $\Omega'$ for which there exist another open set $\Omega''$ such that $\Omega'\Subset\Omega''\Subset\Omega$ and a cut-off function $\eta' \in C^\infty(\R^n)$ with sufficient number of bounded derivatives in $\R^n$ such that  $\eta' (x)=1$ in $\Omega'$ and $\eta' (x)=0$ in $\R^n\setminus\Omega''$. In the first paragraph of Step (i) we can chose then a cut-off function $\eta_\rho \in C^\infty(\R^n)$  such that $\eta_\rho (x)=1$ in $\R^n\setminus B_{2\rho}$ and $\eta_\rho (x)=0$ in $B_\rho$. Defining $\eta (x):=\eta'(x)\eta_\rho(x)$ we have $\eta(x)=1$ in $\Omega'\setminus B_{2\rho}$ and $\eta (x)=0$ in $(\R^n\setminus\Omega'')\bigcup B_\rho$. Then the support of
$U_\eta(x):=\eta(x)u(x)$ belongs to $\bar\Omega''\backslash B_{\rho}\subset\Omega$ and we can follow the proof of Theorem~\ref{RegThInf} as before.
\end{remark}

\section{Extensions and generalized co-normal
derivatives for PDE systems with non-smooth coefficients}\label{S3sm3}
\subsection{Extension of partial differential operators}
Let $\ha<s<\tha$ and $\{a,b,c\}\in {\cal C}_+^{s-1}(\overline\Omega)$ (which for the case $s=1$ means $a,b,c\in L_\infty(\Omega)$).
In addition to the operator $A$ defined by \eqref{LH1-H}, let us consider also the {\em aggregate} partial differential operator
 $\check{A}$,
defined as,
\begin{equation}\label{Ltil-H}
    \langle \check{A}u,v \rangle_\Omega:=\check\E(u,v)\quad \forall v\in
    {H}^{2-s}(\Omega),
\end{equation}
where
\be\label{Echeckdef-H}
\check\E(u,v)=\check\E_\Omega(u,v):=
  \sum\limits_{i,j=1}^{n} \left\langle \s{E}^{s-1}(a_{ij}\pa_j u) , \pa_i v\right\rangle_{\Omega}
 +\sum\limits_{j=1}^{n} \left\langle \s{E}^{s_b(s)}(b_{j}\pa_j u ) ,  v\right\rangle_{\Omega}
 +\left\langle \s{E}^{s_c(s)}(c u)  ,  v\right\rangle_{\Omega},
 \ee
$\s{E}^{s-1}: H^{s-1}(\Omega)\to \s H^{s-1}(\Omega)$, $\s{E}^{s_b(s)}: H^{s_b(s)}(\Omega)\to \s H^{s_b(s)}(\Omega)$, $\s{E}^{s_c(s)}: H^{s_c(s)}(\Omega)\to \s H^{s_c(s)}(\Omega)$ are bounded extension operators, which are unique by \cite[Theorem~2.16]{MikJMAA2011} (Theorem~\ref{ExtOper-H} in Appendix~\ref{S2.2-H}) since $-\ha<s-1<\ha$ and $-\ha<s_b(s)\le 0$, $s_c(s)=0$ by \eqref{sbsc}.
Then the bilinear form $\check\E(u,v): H^s(\Omega)\times H^{2-s}(\Omega)\to \C$ is bounded by Theorem \ref{GrL}, implying that the operator $\check{A}: H^{s}(\Omega)\to\s{H}^{s-2}(\Omega)=[{H}^{2-s} (\Omega)]^*$ is bounded, for $\ha<s<\tha$.

Note that by \eqref{dp2}-\eqref{dp3-H} one can rewrite \eqref{Ltil-H} also as
$
    (\check Au,v )_{\Omega}:=\Phi(u,v)\quad \forall v\in
    {H}^{2-s}(\Omega),
$
where $\Phi(u,v)=\ov{\check\E(u,\bar v)}$ is the sesquilinear form.

If $s=1$, i.e. $u,v\in H^1(\Omega)$, then evidently
\bes
\check\E(u,v)=\E(u,v)=
  \int_\Omega\left[\sum\limits_{i,j=1}^{n}  (a_{ij}\pa_j u) \cdot \pa_i v
 +\sum\limits_{j=1}^{n} (b_{j}\pa_j u ) \cdot  v
 + c u \cdot   v\right]\,dx .
 \ees

For $\ha<s<\tha$ and $\{a,b,c\}\in {\cal C}_+^{s-1}(\overline\Omega)$ let us consider also the {\em aggregate} operator
 $\check{A}^*: H^{2-s}(\Omega)\to\s{H}^{-s}(\Omega)=[{H}^{s} (\Omega)]^*$,
defined as,
\begin{equation}\label{Ltil*-H}
    \langle \check{A}^*v,u \rangle_\Omega:=\check\E^*(v,u)\quad \forall u\in
    {H}^{s}(\Omega),
\end{equation}
\begin{multline}\label{E*E-H}
    \check\E^*(v,u)=\ov{\check\E(\bar u,\bar v)}=\Phi(\bar u,v)
=
  \sum\limits_{i,j=1}^{n} \left\langle \bar a_{ij}\pa_j u , \s{E}^{1-s}\pa_i v\right\rangle_{\Omega}
 +\sum\limits_{j=1}^{n} \left\langle \bar b_{j}\pa_j u ,\s{E}^{-s_b(s)} v\right\rangle_{\Omega}
 +\left\langle \bar c u , \s{E}^{-s_c(s)}v\right\rangle_{\Omega}\quad
\end{multline}
by \eqref{Echeckdef-H} since $(\s{E}^{p})^*=\s{E}^{-p}$ for $-\ha<p<\ha$ by \cite[Theorem~2.16]{MikJMAA2011} (Theorem~\ref{ExtOper-H} in Appendix~\ref{S2.2-H}).

Due to Theorem \ref{GrL} and relations \eqref{Ltil*-H}, \eqref{E*E-H} and  \eqref{Ltil-H}, we have the following statement.
\begin{theorem}\label{T4.4c}
If $\ha<s<\tha$ and $\{a,b,c\}\in {\cal C}_+^{s-1}(\overline\Omega)$, then bilinear form \eqref{Echeckdef-H}, $\check\E: \{H^s(\Omega), H^{2-s}(\Omega)\}\to\C$ is bounded and expressions \eqref{Ltil-H},   \eqref{Ltil*-H} define  bounded linear operators
 $
 \check{A}: H^s(\Omega)\to\s{H}^{s-2}(\Omega)$,
 $
 \check{A}^*: H^{2-s}(\Omega)\to\s{H}^{-s}(\Omega),
 $
and the aggregate second Green's identity holds true in the following form,
\begin{equation}\label{A2GI-H}
\langle \check{A}u,\bar v \rangle_\Omega=\langle u, \ov{\check{A}^*v}\rangle_\Omega,\quad u\in
    {H}^{s}(\Omega),\quad v\in H^{2-s}(\Omega), \quad\ha<s<\tha.
\end{equation}
\end{theorem}

For any $u\in H^s(\Omega)$, $\ha<s<\tha$, the
functional $\check{A}u$ belongs to $\s{H}^{s-2}(\Omega)$ and is an
extension of the functional ${A}u\in {H}^{s-2}(\Omega)$ from the
domain of definition  $\s{H}^{2-s}(\Omega)$ to the domain of
definition ${H}^{2-s}(\Omega)$.  Similarly, for any $v\in H^{2-s}(\Omega)$, $\ha<s<\tha$, the
distribution $\check{A}^*v$ belongs to $\s{H}^{-s}(\Omega)$ and is  an extension of the functional ${A}^*v\in {H}^{-s}(\Omega)$ from the
domain of definition  $\s{H}^{s}(\Omega)$ to the domain of
definition ${H}^{s}(\Omega)$.

The distribution $\check{A}u$ is not the only possible extension of the functional ${A}u$, and any functional of the form
\begin{equation}\label{Ext-H}
     \check{A}u +g,\quad  g\in
    {H}^{s-2}_{\partial\Omega}
\end{equation}
gives another extension. On the other hand, any extension of the domain
of definition of the functional ${A}u$ from $\s{H}^{2-s}(\Omega)$ to
${H}^{2-s}(\Omega)$ has evidently form \eqref{Ext-H}. The existence of
such extensions is provided by \cite[Theorem~2.16]{MikJMAA2011} (Theorem~\ref{ExtOper-H} in Appendix~\ref{S2.2-H}).

\subsection{Generalized co-normal derivatives}

Let $\gamma^+: H^s(\Omega)\to H^{s-\frac{1}{2}}(\partial\Omega)$ denote the trace operator, which is bounded on Lipschitz domains for $\frac{1}{2}<s<\frac{3}{2}$.

For $u\in H^s(\Omega)$, $s>\tha$,  and $a\in C(\bar\Omega)$,
the strong (classical) co--normal derivative operator
 \begin{equation}\label{clCD-H}
T^+_cu(x) := \sum_{i,j=1}^n
a_{ij}(x)\,\gamma^+[\pa_j u(x)]\nu_i(x)
\end{equation}
is well defined on $\pO$ in the sense of traces. Here $\gamma^+[\pa_j u]\in H^{s-\tha}(\pO)\subset L_2(\pO)$ if $\tha<s<\frac{5}{2}$, while the outward (to $\Omega$) unit normal vector $\nu(x)$ at
the point $x\in \pO$ belongs to $L_\infty(\pO)$ for the Lipschitz boundary $\pO$, implying  $T^+_cu\in L_2(\pO)$.
Note that for Lipschitz domains, $T^+_c u$ does not generally belong to $H^s(\pO)$, $s>0$, even for infinitely smooth $u$.

A definition  of
the generalized co--normal derivative is given in \cite[Lemma 4.3]{McLean2000-H} for $s=1$ (cf. also \cite[Lemma 2.2]{Kohr-Pintea-Wendland2010-H} for the generalized co--normal derivative on a manifold boundary) and in \cite{MikJMAA2011} for $\ha< s<\tha$ and infinitely smooth coefficients. We can now extend the definition to the range of Sobolev spaces and non-smooth coefficients.

\begin{definition}\label{GCDd-H}
Let $\Omega$ be a Lipschitz domain, $\ha< s<\tha$,  $u\in H^s(\Omega)$, $\{a,b,c\}\in {\cal C}_{+}^{s-1}(\overline\Omega)$, and $Au=\tilde{f}|_{_\Omega}\in {H}^{s-2}(\Omega)$ in
$\Omega$ for some $\tilde{f}\in \s{H}^{s-2}(\Omega)$. Let us
define {\em the generalized co--normal derivative\ }
$T^+(\tilde{f},u) \in H^{s-\tha}(\pO)$  as
\begin{equation*}
\left\langle
 T^+(\tilde{f},u)\,,\, w
\right\rangle _{\pO}:=
\check\E(u,\gamma_{-1} w)-\langle \tilde{f},\gamma_{-1} w \rangle_\Omega
=\langle \check{A}u - \tilde{f},\gamma_{-1} w \rangle_\Omega
  \quad  \forall\ w\in H^{\tha-s} (\pO),
\end{equation*}
where $\gamma_{-1} : H^{\tha-s}(\pO)\to H^{2-s}(\Omega)$ is a bounded
right inverse to the trace operator.
\end{definition}

\begin{theorem}\label{GCDl-H}
Under the hypotheses of Definition \ref{GCDd-H}, the generalized
co--normal derivative
$T^+(\tilde{f},u)$ is independent of the operator $\gamma_{-1}$, the estimate
$
\|T^+(\tilde{f},u)\|_{H^{s-\tha}(\pO)}\le
C_1\|u\|_{H^s(\Omega)} + C_2\|\tilde{f}\|_{\s{H}^{s-2}(\Omega)}
$
takes place, and the first Green's identity holds in the following
form,
\begin{equation}
\label{Tgen-H} \left\langle
 T^+(\tilde{f},u)\,,\, \gamma^+v
\right\rangle _{\pO}
=\check\E(u,v)-\langle \tilde{f},v \rangle_\Omega
=\langle\check{A}u - \tilde{f},v \rangle_\Omega
  \quad
\forall\ v\in H^{2-s}(\Omega).
\end{equation}
\end{theorem}
\begin{proof} The proof of the theorem coincides word-for-word with the proof of its counterpart for infinitely smooth coefficients, Theorem~3.2  in \cite{MikJMAA2011} .
\end{proof}

Because of the involvement of $\tilde{f}$, the
generalized co-normal derivative $T^+(\tilde{f},u)$  is
generally {\em non-linear} in $u$. It becomes linear if a linear
relation is imposed between $u$ and $\tilde{f}$ (including
behaviour of the latter on the boundary $\partial\Omega$), thus
fixing an extension of
$\tilde{f}|_{_\Omega}=Au$ into
$\s{H}^{s-2}(\Omega)$. For example, $\tilde{f}|_{_\Omega}$ can be
extended as
 $\check{f}:=\check{A}u,$
which generally does not coincide with $\tilde{f}$. Then obviously,
$T^{+}(\check{f},u)=T^{+}(\check{A}u,u)=0$, meaning that the co-normal derivatives associated with any other possible extension $\tilde{f}$ appear to be aggregated in $\check{f}$ as
 \be
\label{ftot-H}
\langle \check{f},v \rangle_\Omega=\langle \tilde{f},v \rangle_\Omega+ \left\langle T^+(\tilde{f},u)\,,\, \gamma^+v \right\rangle _{\pO}
  \quad
\forall\ v\in H^{2-s}(\Omega)
 \ee
due to \eqref{Tgen-H}. This justifies the term {\em aggregate} for the extension $\check{f}$, and thus for the operator $\check{A}u$.

As follows from Definition~\ref{GCDd-H}, the generalized co-normal derivative is still linear with respect to the couple $(\tilde f, u)$, i.e.,\
$
T^+(\alpha_1\tilde{f}_1,\alpha_1u_1) + T^+(\alpha_2\tilde{f}_2,\alpha_2u_2)=
T^+(\alpha_1\tilde{f}_1+\alpha_2\tilde{f}_2,\alpha_1u_1+\alpha_2u_2)
$
for any complex numbers $\alpha_1,\alpha_2$.

In fact, for a given function $u\in H^s(\Omega)$, $\ha<s<\tha$, any
distribution $\tau\in H^{s-\tha}(\pO)$ may be nominated as a co-normal
derivative of $u$, by an appropriate extension $\tilde{f}$ of the
distribution $Au\in {H}^{s-2}(\Omega)$ into $\s{H}^{s-2}(\Omega)$. This extension is
again given by the second Green's identity \eqref{Tgen-H} re-written as
follows (cf. \cite[Section 2.2, item 4]{Agranovich2003RMS-H} for $s=1$),
\begin{equation}
\label{Luext-H}
 \langle \tilde{f},v \rangle_\Omega:=
 \check\E(u,v)-\left\langle \tau ,  \gamma^{+}v \right\rangle _{\pO}=
\langle \check{A}u-\gamma^{+*}\tau,v \rangle_\Omega
  \quad  
\forall\ v\in H^{2-s} (\Omega).
\end{equation}
Here the operator $\gamma^{+*} : H^{s-\tha}(\pO) \to
\s{H}^{s-2}(\Omega)$ is adjoint to the trace operator, $\langle
\gamma^{+*}\tau, v\rangle_\Omega :=\left\langle \tau , \gamma^{+}v
\right\rangle _{\pO}$ for all $\tau\in H^{s-\tha}(\pO)$ and $v\in
H^{2-s} (\Omega)$. Evidently, the distribution $\tilde{f}$ defined by
\eqref{Luext-H} belongs to $\s{H}^{s-2}(\Omega)$ and is an extension of
the distribution $Au$ into $\s{H}^{s-2}(\Omega)$ since $\gamma^{+}v=0$
for $v\in \s{H}^{2-s}(\Omega)$.

For $u\in C^1(\ov{\Omega})\subset H^1(\Omega)$,  one can take $\tau$ equal to the strong co-normal derivative, $T^+_c u \in L_\infty(\pO)$, and relation \eqref{Luext-H} can be considered as the {\em classical extension} of $f=Au\in H^{-1}(\Omega)$ to $\tilde f_c\in \s H^{-1}(\Omega)$, which is evidently linear.

For a sufficiently smooth function $v$ and $a,b\in C(\bar\Omega)$, let
 \begin{equation*}\label{clCDm-H}
 T^+_{*c}v(x) := \sum\limits_{i,j=1}^n
\bar a_{ji}^\top(x)\,\gamma^+[\pa_j v(x)]\nu_i(x)+\sum_{i=1}^n \, \bar b_{i}^\top(x)\gamma^+v(x)\nu_i
 \end{equation*}
be the strong (classical) modified co-normal derivative (it corresponds to  $\s{\mathfrak{B}}_\nu v$ in \cite{McLean2000-H}), associated with the  operator $A^*$.

If $v\in H^{2-s}(\Omega)$, $\{a,b,c\}\in {\cal C}_{+}^{s-1}(\overline\Omega)$, $\ha< s<\tha$,  and $A^*v=\tilde{f}_*|_{_\Omega}$ in
$\Omega$ for some $\tilde{f}_*\in \s{H}^{-s}(\Omega)$, we
define {\em the generalized modified co--normal derivative\ } $ T^+_*(\tilde{f}_*,v) \in H^{\ha-s}(\pO)$,  associated with the  operator $A^*$, similar to  Definition \ref{GCDd-H}, as
 $$
\left\langle
  T^+_*(\tilde{f}_*,v)\,,\, w
\right\rangle _{\pO}:=
\check\E^*(v,\gamma_{-1} w)-\langle \tilde{f}_*,\gamma_{-1} w \rangle_\Omega
  \quad  \forall\ w\in H^{s-\ha} (\pO).
 $$
As in Theorem \ref{GCDl-H}, this leads to the following first Green's identity for the function $v$,
\begin{equation}
\label{Tgen*-H} \left\langle
 T^+_*(\tilde{f}_*,v)\,,\, \gamma^{+}u 
\right\rangle _{\pO}
=\check\E^*(v,u)-\langle \tilde{f}_*,u \rangle_\Omega
  \quad  
\forall\ u\in H^{s}(\Omega),
\end{equation}
which by \eqref{E*E-H}  implies
 \begin{equation}
\label{Tgen*bar-H} \left\langle
\gamma^{+}u,\ov{ T^+_*(\tilde{f}_*,v)}
\right\rangle _{\pO}
=\check\E(u,\bar v)-\langle u, \ov{\tilde{f}}_* \rangle_\Omega
  \quad  
\forall\ u\in H^{s}(\Omega).
\end{equation}
If, in addition, $Au=\tilde{f}|_{_\Omega}$ in $\Omega$ with some
$\tilde{f}\in \s{H}^{s-2}(\Omega)$, then combining \eqref{Tgen*bar-H} and
the first Green's identity \eqref{Tgen-H} for $u$, we arrive at the following generalized second Green's identity,
\begin{equation} \label{2.5sg*-H}
\langle \tilde{f},\bar v \rangle_\Omega\
 -\langle u, \ov{\tilde{f}}_*\rangle _{_\Omega}=
 \left\langle \gamma^{+}u, \ov{ T^+_*(\tilde{f}_*, v)}\right\rangle _{_\pO}
 - \left\langle T^+(\tilde{f},u)\,,\, \ov{\gamma^{+}v}\right\rangle _{_\pO}
 .
\end{equation}
By \eqref{Tgen*-H}, \eqref{Ltil*-H} and \eqref{Tgen-H}, \eqref{Ltil-H}, this, of course, leads to the aggregate second Green's identity \eqref{A2GI-H}.

\subsection{Generalized weak settings of boundary value problems}\label{NewPVPs-H}
Similar to the case of infinitely smooth coefficients in \cite[Section 3.2]{MikJMAA2011}, let us consider the generalized BVP weak settings  for PDE system \eqref{2.1-H} on an interior Lipschitz domain for $\ha<s<\tha$ and $\{a,b,c\}\in {\cal C}_{+}^{s-1}(\overline\Omega)$.

{\em The Dirichlet problem:} for $f\in H^{s-2}(\Omega)$ and $\varphi_0\in H^{s-\ha}(\pO)$, find $u\in H^s(\Omega)$ such that
\begin{eqnarray}\label{wD-H}
    \langle Au,v \rangle_\Omega&=&\langle f,v \rangle_\Omega\quad \forall v\in
    \s{H}^{2-s}(\Omega),\\
    \gamma^+u&=&\varphi_0\quad\mbox{on}\ \pO,\label{bD-H}
\end{eqnarray}
where $Au$ is defined by \eqref{LH1-H}.

{\em The Neumann problem:} for $\check f\in \s{H}^{s-2}(\Omega)$, find $u\in H^s(\Omega)$ such that
\begin{eqnarray}\label{wN-H}
    \langle \check Au,v \rangle_\Omega&=&\langle \check f,v \rangle_\Omega\quad \forall v\in
    {H}^{2-s}(\Omega),
\end{eqnarray}
where $\check Au$ is defined by \eqref{Ltil-H}.

{\em The mixed (Dirichlet-Neumann) problem:} for $\check f_m\in [{H}^{2-s}_{0} (\Omega,{\partial_D\Omega })]^*$ and $\varphi_0\in H^{s-\ha}(\partial_D\Omega)$, find $u\in H^s(\Omega)$ such that
\begin{eqnarray}\label{wM-H}
    \langle \check A_{\partial_D\Omega }u,v \rangle_\Omega&=&\langle \check f_m,v \rangle_\Omega\quad \forall v\in
    {H}^{2-s}_{0} (\Omega,\partial_D\Omega ),\\
    \gamma^+u&=&\varphi_0\quad\mbox{on}\ \partial_D\Omega.\label{wM2-H}
\end{eqnarray}
Here  $\check{A}_{\partial_D\Omega }: H^s(\Omega) \to
[{H}^{2-s}_{0} (\Omega,{\partial_D\Omega })]^*$ is the {\em mixed aggregate} operator defined
as
 $$
    \langle \check{A}_{\partial_D\Omega }u,v \rangle_\Omega:=
    \langle \check{A}u,v \rangle_\Omega=
\check\E(u,v)\quad \forall\ v\in {H}^{2-s}_{0} (\Omega,\partial_D\Omega ).
 $$
where, respectively, the Dirichlet and Neumann parts of the boundary, $\partial_D\Omega $  and $\partial_N\Omega =\pO\backslash \overline{\partial_D\Omega}$ are nonempty, open
sub--manifolds of $\pO$, and ${H}^{s}_{0} (\Omega,\partial_D\Omega)=\{w\in {H}^{s}(\Omega): \gamma^+ w=0$ on $\partial_D\Omega\}$.
The operator $\check{A}_{\partial_D\Omega }$ is bounded by the same argument as the aggregate operator $\check{A}$. For any $u\in H^s(\Omega)$, the
distribution $\check{A}_{\partial_D\Omega }u$ belongs to $[{H}^{2-s}_{0} (\Omega,{\partial_D\Omega })]^*$ and is an
extension of the functional ${A}u\in {H}^{s-2}(\Omega)$ from the
domain of definition  $\s{H}^{2-s}(\Omega)={H}^{2-s}_0(\Omega)\subset {H}^{2-s}_{0} (\Omega,\partial_D\Omega)$ to the domain of
definition ${H}^{2-s}_{0} (\Omega,\partial_D\Omega )$, and a restriction of the functional $\check{A}u\in \s{H}^{s-2}(\Omega)$ from the
domain of definition  ${H}^{2-s}(\Omega)\supset {H}^{2-s}_{0} (\Omega,\partial_D\Omega)$ to the domain of
definition ${H}^{2-s}_{0} (\Omega,\partial_D\Omega)$.

Note that one can take $v=\bar w$ to make the settings \eqref{wD-H}-\eqref{bD-H}, \eqref{wN-H} and \eqref{wM-H}-\eqref{wM2-H} in terms of the sesquilinear inner product and look more like the usual variational formulations, cf. e.g. \cite{LiMa1-H}.

The Dirichlet problem setting \eqref{wD-H}-\eqref{bD-H} coincides with the usual one, c.f. \cite{McLean2000-H}, (i.e., does not need a generalization), and the co-normal derivative does not evidently participate in it.
The Neumann and mixed problems are formulated in terms of the aggregate right hand sides  $\check f$ and $\check f_m$, respectively, prescribed on their own, i.e., without necessary splitting them into the given right hand side of the PDE inside the domain $\Omega$ and the part related with the co-normal derivative prescribed on the boundary. If, however, a PDE right hand side extension $\tilde f$ and an associated non-zero generalized co-normal derivative $T^+(\tilde{f},u)=\tau$ are prescribed instead, then $\check f$  can be expressed through it by relation \eqref{ftot-H} and  $\check f_m$ by relation
\begin{eqnarray*}
&&\langle \check{f}_m,v \rangle_\Omega=\langle \tilde{f},v \rangle_\Omega+ \left\langle \tau\,,\, \gamma^+v \right\rangle _{\partial_N\Omega }=\langle \tilde{f}+\gamma^{+*}\tau,v \rangle_\Omega
\quad \forall\ v\in {H}^{2-s}_{0} (\Omega,\partial_D\Omega ),
\end{eqnarray*}
also obtained from \eqref{ftot-H}, where it is taken into account that the trace $\gamma^+v$ belongs to $\s H^{s-\ha}(\partial_N\Omega )$ for $v\in {H}^{2-s}_{0} (\Omega,\partial_D\Omega )$ and $\gamma^{+*}: H^{s-\tha}(\partial_N\Omega)\to [{H}^{2-s}_{0} (\Omega,\partial_D\Omega )]^*$ is a continuous operator adjoint to the operator $\gamma^+$.

Thus the co-normal derivative does not enter, in fact,  the generalized weak settings of the Dirichlet, Neumann or mixed problem, implying that the non-uniqueness of  $T^+(\tilde{f},u)$ for a given function $u\in H^s(\Omega)$, $\ha<s<\tha$, does not influence the BVP weak settings, (cf. \cite[Section 2.2, item 4]{Agranovich2003RMS-H} for $s=1$).
On the other hand, for a given $u\in H^s(\Omega)$ the aggregate right hand sides $\check f$ and $\check f_m $ are uniquely determined by $u$ from \eqref{wN-H}, \eqref{wM-H}, as are, of course, $f$ and $\varphi_0$ by \eqref{wD-H},  \eqref{bD-H}/\eqref{wM2-H}.

Remark also that the formulation of the Neumann and mixed BVPs in terms of the aggregate right hand side can be also illustrated by a physical interpretation. For the Neumann problem, for example, if $A$ is a partial differential operator of the Lam\'e system of linear elasticity in a body $\Omega\subset\R^3$ for the displacement vector $u\in H^{1}(\Omega)$, then $\tilde f\in \tilde H^{-1}(\Omega)$ is the distributed volume force vector acting on the body and $T^+(\tilde{f},u)=\tau\in H^{-\ha}(\partial\Omega)$ is the prescribed traction vector on the boundary. Then $\tau\in H^{-\ha}(\partial\Omega)$ from the mechanical point of view is indistinguishable from the corresponding volume force $\gamma^{+*}\tau\in\tilde H^{-1}(\Omega)$  concentrated on the boundary surface and thus can be summed up with $\tilde f$ to produce the aggregate right hand side $\check{f}$.

\section{Canonical co-normal derivative for PDE systems with non-smooth coefficients} \label{S-CCD}
\subsection{Canonical operator extension and co-normal derivative}
As we have seen above, for an arbitrary $u\in H^s(\Omega)$, $\ha<s<\tha$, the co-normal derivative $T^+(\tilde{f},u)$ is generally  non-uniquely determined by $u$. An exception is $T^{+}(\check{A}u,u)\equiv 0$, which was in fact implemented in the revised weak  setting of the boundary value problems in Section~\ref{NewPVPs-H}.  But such zero co-normal derivative evidently differs from the strong co-normal derivative ${T}^{+}_c u$, given by \eqref{clCD-H} for sufficiently smooth $u$.
Another one way of making the generalized co-normal derivative unique for $u\in H^1(\Omega)$ was presented in \cite[Lemma 5.1.1]{Hsiao-Wendland2008-H} and is in fact associated with an extension of
$Au\in H^{-1}(\Omega)$ to $\tilde f\in\widetilde H^{-1}(\Omega)$, such that $\tilde f$ is orthogonal in $H^{-1}(\R^n)$ to $H^{-1}_\pO\subset H^{-1}(\R^n)$. However it appears (see \cite[Lemma A.1]{MikJMAA2011}), that even for infinitely smooth functions $f$ such extension $\tilde f$ does not generally belong to $L_2(\R^n)$, which implies that  the so-defined co-normal derivative operator from \cite[Lemma 5.1.1]{Hsiao-Wendland2008-H} is not a bounded extension of the strong co-normal derivative operator.

Nevertheless, we can  point out some subspaces of $H^s(\Omega)$, $\ha<s<\tha$, where a unique definition of the co-normal derivative by $u$ is still possible and leads to the strong co-normal derivative for sufficiently smooth $u$. Following \cite{MikJMAA2011}, we define below one such sufficiently wide subspace.

\begin{definition}\label{Hst-H}
Let $s\in\R$ and $A_*:H^s(\Omega)\to {\cal D}^*(\Omega)$ be a
linear operator. For $t\ge -\ha$, we introduce a space
 $
 H^{s,t} (\Omega;A_*):=\{g:\;g\in H^s (\Omega),\ A_*g|_\Omega=\tilde{f}_g|_\Omega,\
\tilde{f}_g\in
\s{H}^{t}(\Omega)\}
 $
equipped with the graphic norm,
 $ \|g\|_{ H^{s,t} (\Omega;A_*)}^2:= \|g\|_{ H^s (\Omega)}^2+ \|\tilde{f}_g\|_{\s{H}^{t}(\Omega)}^2
 $.
 \end{definition}

If $s_1\le s_2$ and $t_1\le t_2$, then we have the
embedding, $ H^{s_2,t_2} (\Omega;A_*)\subset H^{s_1,t_1}
(\Omega;A_*)$. Some other properties of the space $H^{s,t} (\Omega;A_*)$ studied  in \cite[Section 3.2]{MikJMAA2011} are provided in Appendix~\ref{S2.2-H}.

We will further use the space $H^{s,t} (\Omega;A_*)$ for the case when the operator $A_*$ is the operator $A$ from \eqref{Ldist-H} or the formally adjoint operator $A^*$ from \eqref{LH1*-H}.

\begin{definition}\label{Dce-H} Let $s\in\R$, $t\ge -\ha$. The operator $\tilde A$ mapping functions $u\in H^{s,t}(\Omega;A)$ to the extension of the distribution $Au\in {H}^t(\Omega)$ to $\s{H}^t(\Omega)$ will be called {\em the canonical extension} of the operator $A$.
\end{definition}

\begin{remark}
If $s\in\R$, $t\ge -\ha$, then  $\|\tilde Au\|_{\s{H}^t(\Omega)}\le \|u\|_{H^{s,t}(\Omega;A)}$ by the definition of the space $H^{s,t}(\Omega;A)$, i.e., the linear operator
 $
\tilde A: H^{s,t}(\Omega;A)\to \s{H}^t(\Omega)
 $
is continuous.
Moreover, if $\ -\ha<t<\ha$, then by Theorem~\ref{ExtOper-H} and uniqueness of the extension of ${H}^t(\Omega)$ to $\s{H}^t(\Omega)$, we have the representation $\tilde A:=\s E^tA$.
\end{remark}

\begin{remark}\label{NSD}Note that in the case of non-smooth coefficients of the operator $A$, the inclusion $u\in H^s(\Omega)$, $s>3/2$, does not generally imply that $u\in H^{s,t}(\Omega;A)$ for some $t\ge-\ha$, unlike the case of infinitely (or at least sufficiently) smooth coefficients. Particularly, even $u\in\D(\bar\Omega)$ does not generally belong to $H^{1,-\ha}(\Omega;A)$  unless   $a\in \bar C^\mu(\Omega)$ for some $\mu>1/2$ (and $b,c\in L_\infty(\Omega)$) by Theorem~\ref{GrL}, i.e., the usual assumption $a\in L_\infty(\Omega)$ is not generally sufficient for this.
\end{remark}

As in \cite[Definition 3]{MikhailovIMSE06-H} for scalar PDEs, let us define the {\em canonical} co-normal
derivative operator. This extends  \cite[Theorem 1.5.3.10]{Grisvard1985-H} and \cite[Lemma 3.2]{Costabel1988-H} where
co-normal derivative operators acting on functions from $H^{1,0}_p(\Omega;\Delta)$ and
$H^{1,0}(\Omega;A)$, respectively, were defined.

\begin{definition}\label{Dccd-H}
For $u\in H^{s,-\ha}(\Omega;A)$, $\{a,b,c\}\in {\cal C}_{+}^{s-1}(\overline\Omega)$, $\ha<s<\tha$,   we define the {\em canonical
co-normal derivative} as $T^+u:= T^+(\tilde A u,u)\in H^{s-\tha}(\pO)$, i.e.,
\bes
\label{Tcandef-H}
 \left\langle  T^+u\,,\, w\right\rangle _{\pO}
:= \check\E(u,\gamma_{-1} w)-\langle \tilde A u,\gamma_{-1} w \rangle_\Omega
=\langle\check{A}u- \tilde A u,\gamma_{-1} w \rangle_\Omega
  \quad  
\forall\ w\in H^{\tha-s} (\pO),
\ees
where $\gamma_{-1} : H^{s-\ha}(\pO)\to H^{s}(\Omega)$ is a bounded right inverse to the trace
operator.
\end{definition}
Thus, unlike the generalized co-normal derivative, the canonical
co-normal derivative is uniquely defined by the function $u$ and
the operator $A$ only, uniquely fixing an extension of the latter on
the boundary, and is linear in $u$.

Theorem \ref{GCDl-H} for the generalized co-normal derivative and
Definition \ref{Hst-H} imply the following assertion.
\begin{theorem}\label{GCDlc-H}
Under hypotheses of Definition \ref{Dccd-H},  the canonical
co-normal derivative $T^+u$ is independent of the
operator $\gamma_{-1}$, the operator $T^+ : H^{s,-\ha}(\Omega;A)\to
H^{s-\tha}(\pO)$ is continuous, and the first Green's identity holds
in the following form,
\begin{eqnarray*} \label{Tcan-H}
 \left\langle T^{+}u\,,\, \gamma^+v\right\rangle _{_\pO}
 =\left\langle T^+(\tilde A u,u)\,,\, \gamma^+v\right\rangle _{_\pO}
 =\check\E(u,v)-\langle \tilde A u,v \rangle_\Omega
 =\langle \check{A}u-\tilde A u,v \rangle_\Omega
  \quad 
\forall\ v\in H^{2-s} (\Omega) .
\end{eqnarray*}
\end{theorem}

Definitions \ref{GCDd-H} and \ref{Dccd-H} imply that the generalized
co-normal derivative of $u\in H^{s,-\ha}(\Omega;A)$, $\ha<s<\tha$,
for any other extension $\tilde{f}\in \s{H}^{s-2}(\Omega)$ of the
distribution $Au|_\Omega\in {H}^{-\ha}(\Omega)$ can be expressed as
\begin{equation*}
\label{Tgentil-H} \left\langle
 T^+(\tilde{f},u)\,,\, w
\right\rangle _{_\pO}=\left\langle T^{+}u\,,\, w \right\rangle
_{_\pO}+\langle \tilde A u - \tilde{f},\gamma_{-1} w \rangle_\Omega
  \quad  
\forall\ w\in H^{\tha-s} (\pO) .
\end{equation*}

Note that the distributions $\check{A}u-\tilde{f}$, $\check{A}u-\tilde A u$
and $\tilde A-\tilde{f}$ belong to ${H}^{2-s}_{\pO}$ since $\tilde Au$,
$\check{A}u$, $\tilde{f}$ belong to $\s{H}^{2-s}(\Omega)$, while
$\tilde Au|_\Omega=\check{A}u|_\Omega=\tilde{f}|_\Omega=Au|_\Omega\in
{H}^{s-2}(\Omega)$.

Since by Theorem~\ref{GCDlc-H} the canonical co-normal derivative does not depend on the  extension operator $\gamma_{-1}$, the latter can be always chosen such that $\gamma_{-1}w$ has a support only near the boundary, which means that the co-normal derivative $T^+u$ is determined by the behaviour of $u$ near the boundary. We can formalize this in the following statement.
\begin{theorem}\label{GCDlc1-H}
Let $\Omega$ and $\Omega'\subset\Omega$ be interior or exterior open Lipschitz domains, $\pO\subset\pO'$, $u\in H^{s,-\ha}(\Omega;A)$, $u\in H^{s,-\ha}(\Omega';A)$, $\{a,b,c\}\in {\cal C}_{+}^{s-1}(\overline\Omega)$,  $\ha<s<\tha$, while $T^+u$ and $T^{\prime+}u$ be the canonical
co-normal derivatives on $\pO$ and $\pO'$ respectively. Then $T^+u=r_{_\pO}T^{\prime+}u$.
\end{theorem}
\begin{proof}
The proof is word-for-word the proof of the counterpart for infinitely smooth coefficients, Theorem~3.10 in \cite{{MikJMAA2011}}
\end{proof}

Theorem~\ref{GCDlc1-H} can be considered as an alternative definition of the canonical co-normal derivative on $\partial\Omega$, where the domain $\Omega'$ can be chosen arbitrarily small, and particularly can be taken interior when $\Omega$ is exterior (with compact boundary). Note that a similar reasoning holds also for the generalized co-normal derivative.

If  $\ha<s<\tha$, $\{a,b,c\}\in {\cal C}_{+}^{s-1}(\overline\Omega)$ and $v \in H^{2-s,-\ha}(\Omega;A^*)$, then similar to Definitions \ref{Dce-H} and \ref{Dccd-H} we can introduce the {\em canonical extension} $\s A^*$ of the operator $A^{*}$, and the {\em canonical modified
co-normal derivative} $ T^+_*v:=  T^+_*(\s A^* v,v)\in H^{\ha-s}(\pO)$, i.e.,
$$
\left\langle
  T^+_*v\,,\, w
\right\rangle _{\pO}:=
\check\E^*(v,\gamma_{-1} w)-\langle \s A^*v,\gamma_{-1} w \rangle_\Omega
  \quad  \forall\ w\in H^{s-\ha} (\pO).
$$
Then the first Green's identity \eqref{Tgen*bar-H} becomes,
\begin{equation*}\label{Tcon*bar-H}
\left\langle
\gamma^{+}u,\ov{ T^+_* v}
\right\rangle _{\pO}
=\check\E(u,\bar v)-\langle u, \ov{\s A^*v} \rangle_\Omega
  \quad 
\forall\ u\in H^{s}(\Omega).
\end{equation*}
For $v \in H^{2-s,-\ha}(\Omega;A^*)$ and $u\in H^s(\Omega)$, $Au=\tilde{f}|_{_\Omega}$ in $\Omega$, where
$\tilde{f}\in \s{H}^{s-2}(\Omega)$,  the second Green's identity \eqref{2.5sg*-H} takes form,
\begin{equation} \label{2.5s*-H}
\langle \tilde{f},\bar v \rangle_\Omega\
 -\left\langle u,\ov{\s A^*v}\right\rangle _{_\Omega}=
 \left\langle \gamma^{+}u,\ov{ T^+_*v}\right\rangle _{_\pO}
 - \left\langle T^+(\tilde{f},u),\ov{\gamma^{+}v}\right\rangle _{_\pO}.
\end{equation}
 This form was a starting point in formulation and analysis of the extended boundary-domain integral equations in \cite{Liverpool2005UKBIM-H}.

If, moreover, $u \in H^{s,-\ha}(\Omega;A)$,  we obtain from
\eqref{2.5s*-H} the second Green's identity for the canonical extensions and
canonical co-normal derivatives,
\begin{equation} \label{GreenCan*-H}
\left\langle \tilde Au,\bar v\right\rangle _{_\Omega}
-\left\langle u,\ov{\s A^*v}\right\rangle _{_\Omega}=
\left\langle\gamma^{+}u ,\ov{ T^+_*v}\,\right\rangle _{_\pO}
-\left\langle T^{+}u\,,\, \ov{\gamma^{+}v}\right\rangle _{_\pO}
 .
\end{equation}
Particularly, if $u \in H^{1,0}(\Omega;A)$, $v \in H^{1,0}(\Omega;A^*)$, with $a,b,c\in L_\infty(\Omega)$, then \eqref{GreenCan*-H} takes the familiar form, cf. \cite[Lemma 3.4]{Costabel1988-H},
 $$
\int_{\Omega}[\,\ov{v(x)}Au(x)- u(x)\ov{A^*v(x)}\,]dx=
\left\langle \gamma^{+}u,\ov{ T^+_*v}\,\right\rangle _{_\pO}-
 \left\langle T^{+}u\,,\, \ov{\gamma^{+}v}\right\rangle _{_\pO}
.
 $$

\subsection{Classical verses canonical co-normal derivatives}\label{S3}

In this section we generalize to the case when the PDE coefficients are not infinitely smooth, the results of \cite{MikJMAA2011} on conditions when the canonical co-normal derivative $T^+ u$ coincides with the strong co-normal derivative $T^+_c u$, if the latter does exist in the trace sense.
To do this, we will need higher smoothness of the coefficients than necessary for continuity of the PDEs in Theorems \ref{T4.4} and \ref{T4.4c}. First of all, we make the following observation, c.f. Remark~\ref{NSD}.
\begin{remark}\label{NSD2}
Theorem~\ref{GrL} and Definition~\ref{DC^sigma} imply that if $\{a,b,c\}\in{\cal C}_{+}^{t+1}(\overline\Omega)$, $t\ge-\ha$, then $\D(\ov \Omega)\subset H^{s,t}(\Omega;A)$ (and moreover, $\D(\ov \Omega)\subset H^{s,t+\epsilon}(\Omega;A)$ for some $\epsilon\in\R_+(t)$) for any $s\in\R$.
\end{remark}
Now we are in the position to generalize the density theorem from \cite[Theorem 3.12]{MikJMAA2011} to non-smooth coefficients and exterior domains.

\begin{theorem}\label{densL-H}
Let $\Omega$ be an interior or exterior Lipschitz domain and $s\in\R$,  $-\ha\le t<\ha$. Let $\{a,b,c\}\in{\cal C}_{+}^{s-1}(\overline\Omega)\bigcap{\cal C}_{+}^{t+1}(\overline\Omega)$,  the operator $A$ be elliptic (in the sense of Petrovsky) on $\ov{\Omega}$ and, if $\Omega$ is exterior, there exists a finite $a(\infty):=\lim_{x\to\infty}a(x)$, which also satisfies the ellipticity condition. Then
$\D(\ov{\Omega})$ is dense in $H^{s,t}(\Omega;A)$.
\end{theorem}
\begin{proof}
We adopt here for the non-smooth coefficients and exterior domains the proof from \cite[Theorem 3.12]{MikJMAA2011}.

For every continuous linear functional $l$ on $H^{s,t}(\Omega;A)$
there exist distributions $\tilde{h}\in \s{H}^{-s}(\Omega)$ and
$g\in H^{-t}(\Omega)$ such that
 $
l(u)=\langle \tilde{h},u\rangle_\Omega +\langle g,
\tilde Au\rangle_\Omega\quad \forall\ u\in
H^{s,t}(\Omega;A).
 $

Remark~\ref{NSD2} and the theorem hypothesis on the coefficients imply that $\D(\ov{\Omega})\subset H^{s,t}(\Omega;A)$.
To prove the lemma claim, it suffices to show that any $l$, which
vanishes on $\D(\ov{\Omega})$, will vanish on any $u\in
H^{s,t}(\Omega;A)$.

If $l(\phi)=0$ for any
$\phi\in\D(\ov{\Omega})$, then
 \begin{equation}\label{j=0-H}
\langle \tilde{h},\phi\rangle_\Omega +\langle g,\tilde A\phi\rangle_\Omega=0.
 \end{equation}

Let us consider the case $-\ha<
t<\ha$ first and extend $g$ outside ${\Omega}$ to
$\tilde{g}=\s E^{-t}g\in \s{H}^{-t}(\Omega)$, cf. Theorem~\ref{ExtOper-H}.
Let $\Omega'\supset \ov{\Omega}$ be some domain, where the operator $A$ is still  elliptic. Such domain exists since the coefficients $a(x)$ are continuous and and the ellipticity condition holds in the closed domain $\ov{\Omega}$.
Then equation \eqref{j=0-H} gives
 \begin{multline}\label{j=0Rn}
\langle \tilde{h},\phi\rangle_{\Omega'} +\langle \tilde{g},A\phi\rangle_{\Omega'}=
\langle \tilde{h},\phi\rangle_{\Omega} +\langle \tilde{g},A\phi\rangle_{\Omega}=
\langle \tilde{h},\phi\rangle_{\Omega} +\langle\s E^{-t}g,A\phi\rangle_{\Omega}\\
=\langle \tilde{h},\phi\rangle_{\Omega} +\langle {g},\s E^tA\phi\rangle_{\Omega}=
\langle \tilde{h},\phi\rangle_{\Omega} +\langle {g},\tilde A\phi\rangle_{\Omega}=0
 \end{multline}
for any $\phi\in\D(\Omega')$.  This means
 \be\label{A*g-H} A^*\tilde{g}=-\tilde{h} \quad \mbox{in } \Omega'\ee
in the sense of distributions, where $A^*$ is the operator formally adjoint to $A$.
If $t\le s-2$, then evidently $\tilde{g}\in \s{H}^{2-s}(\Omega)$. If $t> s-2$, then \eqref{A*g-H} and the local regularity Theorem~\ref{RegThInf}(b) implies
$\tilde{g}\in H^{2-s}(\Omega'')$ for any $\Omega''$ such that $\Omega\Subset\Omega''\Subset\Omega'$ and consequently
$\tilde{g}\in \s{H}^{2-s}(\Omega)$.

In the case $t=-\ha$,
one can extend $g\in H^\ha(\Omega)$ outside $\ov{\Omega}$ by zero to $\tilde{g}\in
\s{H}^{\ha-\epsilon}(\Omega)$, $0<\epsilon$, and prove as
in the previous paragraph that $\tilde{g}\in \s{H}^{2-s}(\Omega)$.

If $-\ha<t<\ha$ or [$t=-\ha$, $s\le \tha$]
then for any $u\in H^{s,t}(\Omega;A)$, we have,
 \bes
l(u)=\langle -A^*\tilde g,u\rangle_\Omega +\langle g ,\tilde Au\rangle_\Omega =-\langle\tilde g,Au\rangle_\Omega
+\langle\tilde g,Au\rangle_\Omega =0.
 \ees
Thus $l$ is identically zero.

On the other hand, if $t=-\ha$, $s> \tha$, let $\{\tilde g_k\}\in\D(\Omega)$ be a sequence converging, as $k\to\infty$, to $g$
in $H^{\ha}_0(\Omega)=H^{\ha}(\Omega)$, cf. Theorem~\ref{H0=H-H}, and thus to $\tilde g$ in $\s{H}^{2-s}(\Omega)$. Then for any $u\in H^{s,\ha}(\Omega;A)$, we have,
 \begin{multline*}
l(u)=\langle -A^*\tilde g,u\rangle_\Omega +\langle g ,\tilde Au\rangle_\Omega =\lim_{k\to\infty}\left\{\langle -A^*\tilde g_k,u\rangle_\Omega
+\langle \tilde g_k,\tilde Au\rangle_\Omega\right\}
=\lim_{k\to\infty}\left\{-\langle
\tilde g_k,Au\rangle_\Omega +\langle \tilde g_k, Au\rangle_\Omega\right\}=0,
 \end{multline*}
which completes the proof.
\end{proof}
 
Let us prove an analogue of Lemma~3.13 from \cite{MikJMAA2011}.
\begin{lemma}\label{CCCDLDH}
Let $\Omega$ be a Lipschitz domain, $\ha<s<\tha$,
$\{a,b,c\}\in{\cal C}_{+}^{\ha}(\overline\Omega)$,
$u\in H^{s,-\ha}(\Omega;A)$, and $\{u_k\}\in \D(\ov{\Omega})$
be a sequence such that
\be\label{u-uk-H}
 \|u_k-u\|_{H^{s,-\ha}(\Omega;A)}\to 0\quad \text{as}\
k\to \infty.
 \ee
Then
 $\|T^+_c u_k-T^+ u\|_{H^{s-\tha}(\pO)}\to 0$ as $k\to \infty$.
\end{lemma}
\begin{proof}
By the lemma hypothesis on the coefficients, there exists $\epsilon\in(0,s-\ha)$ such that $\sum_{j=1}^n a_{ij}\partial_j u_k\in H^{\ha+\epsilon}(\Omega)$. Then for any $\gamma_{-1}w\in H^{2-s}(\Omega)$ and $u_k\in \D(\ov{\Omega})$ there exist sequences $\{W_p\}_{p=1}^\infty, \{U_{qi}\}_{q=1}^\infty\in\D(\ov\Omega)$ such that
$$\lim_{p\to\infty}\|\gamma_{-1} w-W_p\|_{H^{2-s}(\Omega)}=0,\quad
\lim_{q\to\infty}\|\sum_{j=1}^n a_{ij}\partial_ju_k-U_{qi}\|_{H^{\ha+\epsilon}(\Omega)}=0$$
and we have
\begin{multline*}
 \check\E(u_k,\gamma_{-1}w)
 -\sum\limits_{j=1}^{n} \left\langle \s{E}^{s_b(s)}(b_{j}\pa_j u_k),\gamma_{-1}w\right\rangle_{\Omega}
 -\left\langle \s{E}^{s_c(s)}(c u_k),  \gamma_{-1}w\right\rangle
 =\sum\limits_{i,j=1}^{n} \left\langle \s{E}^{s-1}(a_{ij}\pa_j u_k) , \pa_i\gamma_{-1}w\right\rangle_{\Omega}
 \\
 =\lim_{p,q\to\infty}\sum\limits_{i=1}^{n} \left\langle \s{E}^{s-1}U_{qi} , \pa_iW_p\right\rangle_{\Omega}
 =\lim_{p,q\to\infty}\sum\limits_{i=1}^{n}\left\{\int_\pO U_{qi}\nu_iW_p\ d\Gamma-
\int_\Omega(\pa_i U_{qi})W_p\ d\Omega\right\} \\
=\sum\limits_{i,j=1}^{n}\left\{\int_\pO( a_{ij}\partial_ju_k)\nu_i w\ d\Gamma-
\left\langle\s{E}^{-\ha+\epsilon}\pa_i (a_{ij}\partial_ju_k),\gamma_{-1} w\right\rangle_\Omega\right\}
 \\
= \left\langle T^+_c u_k, w\right\rangle_\pO
+ \langle \tilde A u_k,\gamma_{-1}w \rangle_\Omega
-\sum\limits_{j=1}^{n} \left\langle \s{E}^{s_b(s)}(b_{j}\pa_j u_k),\gamma_{-1}w\right\rangle_{\Omega}
 -\left\langle \s{E}^{s_c(s)}(c u_k),  \gamma_{-1}w\right\rangle,
\end{multline*}
that is, the first Green's identity holds for the classical co-normal derivative,\\
$
 \E(u_k,\gamma_{-1}w)=\left\langle T^+_c u_k, w\right\rangle_\pO
+ \langle \tilde A u_k,\gamma_{-1}w \rangle_\Omega.
$
Thus we have for any $w\in
H^{\tha-s}(\pO)$,
\begin{multline*}
\left|\left\langle T^{+} u -T^+_c u_k, w\right\rangle _{_\pO}\right|
 =\left|\check\E(u-u_k,\gamma_{-1} w)- \langle \tilde A (u-u_k), \gamma_{-1} w\rangle_\Omega\right|\le
 C\|u-u_k\|_{H^{s,-\ha}(\Omega;A)}\|w\|_{H^{\tha-s}(\pO)},
  \end{multline*}
which implies
$\|T^+_c u_k-T^+ u\|_{H^{s-\tha}(\pO)}\le C\|u-u_k\|_{H^{s,-\ha}(\Omega;A)}\to 0\quad \mbox{as}\quad k\to \infty.$
\end{proof}
Note that a sequence satisfying \eqref{u-uk-H} does always exist for Lipschitz domains by Theorem~\ref{densL-H} since  Definition~\ref{DC^sigma} implies ${\cal C}_{+}^{\ha}(\overline\Omega)\subset{\cal C}_{+}^{s-1}(\overline\Omega)$  if $\ha<s<\tha$.

The following statement gives the equivalence of the classical co-normal derivative (in the trace
sense) and the canonical co-normal derivative, for functions from
$H^s(\Omega)$, $s>\tha$.

\begin{corollary}\label{CCCDHs}
If $\Omega$ is an interior or exterior Lipschitz domain, $\{a,b,c\}\in{\cal C}_{+}^{\ha}(\overline\Omega)$  and $u\in
H^s(\Omega)$, $s>\tha$, then $T^+ u=T^+_c u$.
\end{corollary}
\begin{proof}
The proof coincides with the proof of \cite[Corollary~3.14]{MikJMAA2011} if we remark that $\gamma^+[\partial_j u]\in H^{s-\tha}(\pO)$ for $\tha<s<\frac{5}{2}$ and the condition $\{a,b,c\}\in{\cal C}_{+}^{\ha}(\overline\Omega)$ implies that  $T^+_c u\in L_2(\pO)$ and  $u\in  H^{s,
-\ha}(\Omega;A)\subset H^{1,-\ha}(\Omega;A)$.
\end{proof}

Similar to \cite[page 85]{Necas1967-H} we introduce the following definition.
\begin{definition}\label{DOktoO-H}
Let $\Omega_k$, $\Omega$ be Lipschitz domains. We say that $\Omega_k\to \Omega$ as $k\to\infty$ if $\pO_k$ are represented using the same system of covering charts $\omega_j$ as $\pO$ for all sufficiently large $k$, and
\\
$
\lim_{k\to\infty}|\zeta_{jk}-\zeta_j|_{C^{0,1}(\bar\omega_j)}=0,
$
where $\zeta_{jk}$ and $\zeta_j$ are the corresponding Lipschitz functions for the boundary representation.
\end{definition}

\begin{lemma}\label{T+=limT+cHs}
Let $\Omega$ and $\Omega_k\Subset\Omega$ be Lipschitz domains such that $\Omega_k\to\Omega$ as $k\to \infty$ (cf. Definition~\ref{DOktoO-H}). If  $u\in H^{s,t}(\Omega;A)$  for some $s\in (\ha,\tha)$ and $t\in(-\ha,\ha)$  and $\{a,b,c,\}\in  {\cal C}_{+}^{\ha}(\overline\Omega)$, then
$\langle T^+ u,\gamma^{+}v \rangle_\pO =\lim_{k\to\infty}\langle T^+_c u,\gamma^{+}v \rangle_{\pO_k}$ for any $v\in H^{2-s}(\Omega)$.
\end{lemma}
\begin{proof}
By Theorem~\ref{GCDlc1-H} it suffices to consider only an interior domain $\Omega$.
Let $\Omega'_k:=\Omega\setminus\ov{\Omega_k}$ be the layer between $\pO$ and $\partial\Omega_k$.  By the solution regularity Theorem~\ref{RegTh}, $u\in H^{t'+2}(\Omega_k)$  for some $t'>-\ha$.  On $\pO_k$ then
$T^+ u=T^+_c u\in L_2(\pO_k)$  by Corollary~\ref{CCCDHs}. Then
\begin{multline}\label{3.28a-H}
 \langle T^+ u,\gamma^{+}v \rangle_\pO - \langle T^+_c u,\gamma^{+}v \rangle_{\pO_k}=
 \langle T^+ u,\gamma^{+}v \rangle_\pO - \langle T^+ u,\gamma^{+}v\rangle_{\pO_k}=
 \langle T^+ u,\gamma^{+}v \rangle_{\pO'_k}=\\
  \check\E_{\Omega'_k}(u,v)- \langle \tilde A_{\Omega'_k} u,v \rangle_{\Omega'_k}
  =\check\E_{\Omega'_k}(u,v)- \langle A u,\tilde v_{\Omega'_k} \rangle_{\Omega'_k},
 \end{multline}
where
$\tilde A_{\Omega'_k} u=\s E_{\Omega'_k}^t r_{\Omega'_k}A u\in \s H^t(\Omega'_k)$
and
$\tilde v_{\Omega'_k}=\s E_{\Omega'_k}^{-t}r_{\Omega'_k}v\in\s H^{-t}(\Omega'_k)$
are the unique extensions of $r_{\Omega'_k}A u\in H^t(\Omega'_k)$ and $r_{\Omega'_k}v\in H^{2-s}(\Omega'_k)\subset H^{-t}(\Omega'_k)$, respectively.

By \eqref{Echeckdef-H} and Theorem~\ref{ExtOper-H} we have for the first term on the right hand side of \eqref{3.28a-H}, for $\ha<s\le 1$ and any $\epsilon\in\R_+(s)$,
 \begin{multline}
 |\check\E_{\Omega'_k}(u,v)|\le{C\sum\limits_{i,j=1}^{n}\|a_{ij}\|_{\bar C^{|s-1|+\epsilon}(\bar\Omega)} \|\pa_j u\|_{H^{s-1}(\Omega'_k)} \|\pa_i v\|_{H^{1-s}(\Omega)}\
 +}\\
  C\sum\limits_{j=1}^{n} \|b_{j}\|_{\bar C^{|s-1|+\epsilon}(\bar\Omega)}\|\pa_j u \|_{H^{s-1}(\Omega'_k)}  \|v\|_{H^{1-s}(\Omega)}
 + C\|c\|_{L_\infty(\Omega)} \| u \|_{H^{0}(\Omega'_k)}\|v\|_{H^{0}(\Omega)}\\
 \le
 \{C_1\|\nabla u \|_{H^{s-1}(\Omega'_k)} + C_2\| \nabla u \|_{H^{s-1}(\Omega'_k)}\} \|v\|_{H^{2-s}(\Omega)}
 + C_3\| u \|_{H^{0}(\Omega'_k)}\|v\|_{H^{0}(\Omega)}
 \to 0,\quad k\to\infty \label{Eest1}
\end{multline}
 by Lemma~\ref{Hsto1-H} since the Lebesgue measure of $\Omega'_k$ tends to zero.
For $1<s< \tha$ similarly,
 \begin{multline}
 |\check\E_{\Omega'_k}(u,v)|\le C\sum\limits_{i,j=1}^{n}\|a_{ij}\|_{\bar C^{|s-1|+\epsilon}(\bar\Omega)} \|\pa_j u\|_{H^{s-1}(\Omega)} \|\pa_i v\|_{H^{1-s}(\Omega'_k)}\
 +\\
 C\sum\limits_{j=1}^{n} \|b_{j}\|_{L_\infty(\Omega)}\|\pa_j u \|_{H^{s-1}(\Omega)}  \|v\|_{H^{1-s}(\Omega'_k)}
 + C\|c\|_{L_\infty(\Omega)} \| u \|_{H^{s-1}(\Omega)}\|v\|_{H^{0}(\Omega'_k)}\\
 \le
 \{C_4\|\nabla v \|_{H^{1-s}(\Omega'_k)} + C_5\|v \|_{H^{0}(\Omega'_k)}\} \|u\|_{H^{s}(\Omega)}
+C_6 \| u \|_{H^{s-1}(\Omega)}\|v\|_{H^{0}(\Omega'_k)}
 \to 0,\quad k\to\infty. \label{Eest2}
\end{multline}
The norms of the coefficients $a,b,c$ in \eqref{Eest1} and \eqref{Eest2} are bounded due to the lemma hypothesis.

For the last term in \eqref{3.28a-H} we have by Lemmas \ref{Hsto3-H} and \ref{Hsto1-H},
 \begin{multline*}
|\langle A u,\tilde v_{\Omega'_k} \rangle_{\Omega'_k}|\le
\|A u\|_{H^t(\Omega'_k)}\|\tilde v_{\Omega'_k}\|_{\s H^{-t}(\Omega'_k)}\le
C\|A u\|_{H^t(\Omega'_k)}\|v\|_{ H^{-t}(\Omega)}\le
C\|A u\|_{H^t(\Omega'_k)}\|v\|_{ H^{2-s}(\Omega)}\to 0,\quad k\to\infty,\qquad
 \end{multline*}
if $-\ha<t\le 0$. On the other hand, if $0<t<\ha$, then again by Lemmas \ref{Hsto3-H} and \ref{Hsto1-H},
 \begin{multline*}
|\langle A u,\tilde v_{\Omega'_k} \rangle_{\Omega'_k}|=
|\langle \tilde A_{\Omega'_k} u, v\rangle_{\Omega'_k}|
\le
\|\tilde A_{\Omega'_k} u\|_{\s H^t(\Omega'_k)}\|v\|_{H^{-t}(\Omega'_k)}\le
C\|A u\|_{H^t(\Omega)}\|v\|_{ H^{-t}(\Omega'_k)}\to 0,\quad k\to\infty.
 \end{multline*}
\end{proof}

Lemma~\ref{T+=limT+cHs} allows to show that the classical and canonical co-normal derivatives coincide also in another case (apart from the one in Corollary~\ref{CCCDHs}). First note that $C^1(\ov{\Omega})\subset H^{1}(\Omega)$ for any interior domain $\Omega$ and $C^1(\ov{\Omega'})\subset H^{1}(\Omega')$ for any interior subdomain $\Omega'$ of exterior domain $\Omega$, but  $C^1(\ov{\Omega})$ is not a subset of $H^{1,-\ha}(\Omega;A)$. For $u\in C^1(\ov{\Omega})$, evidently, $\lim_{k\to\infty}\langle T^+_c u,\gamma^{+}v \rangle_{\pO_k}=\langle T^+_c u,\gamma^{+}v \rangle_\pO$ for any $v\in H^{2-s}(\Omega^+)$ if $\Omega_k\to\Omega$ as $k\to \infty$, $\ov{\Omega}_k\subset\Omega$. Then taking into account that ${\cal C}_{+}^{s-1}(\overline\Omega) \subset{\cal C}_{+}^{\ha}(\overline\Omega)$ for $s=1$, Lemma~\ref{T+=limT+cHs} immediately implies the following assertion.

\begin{corollary}\label{CCCDHsgH}
If $\Omega$ is a Lipschitz domain, $\{a,b,c\}\in{\cal C}_{+}^{\ha}(\overline\Omega)$  and  $u\in C^1(\ov{\Omega})\bigcap H^{1,t}_{\mathrm{loc}}(\overline\Omega;A)$ for some $t\in(-\ha,\ha)$, then $T^+ u=T^+_c u$.
\end{corollary}

\section*{APPENDICES}

\section{Some estimates \label{Est}}
We will prove here some estimates used in Step (i) of the proof of Theorem~\ref{RegThInf}.
Let $0\le \mu\le 1$ and
$$
|a_\infty^-|_{C^\mu(B_\rho)}:=\sup_{x',x''\in B_\rho \atop x'\not=x''} \frac{|a_\infty^-(x'')-a_\infty^-(x')|}{|x''-x'|^\mu}\ .
$$
Let $x',x''\in B_\rho$ and $|x''|\ge|x'|$ for definiteness. Then
$$
\frac{|a_\infty^-(x'')-a_\infty^-(x')|}{|x''-x'|^\mu}=
\frac{\left|\frac{|x''|}{\rho}a^-\left(\frac{x''\rho}{|x''|}\right)
-\frac{|x'|}{\rho}a^-\left(\frac{x'\rho}{|x'|}\right)\right|}{|x''-x'|^\mu}\le A+B,
$$
where
$$
A:=\frac{|x''|}{\rho}\frac{\left|a^-\left(\frac{x''\rho}{|x''|}\right)
-a^-\left(\frac{x'\rho}{|x'|}\right)\right|}{|x''-x'|^\mu},\quad
B:=\frac{\left|\ |x''|-|x'|\ \right|}{\rho|x''-x'|^\mu}
\left| a^-\left(\frac{x'\rho}{|x'|}\right)\right|.
$$
The term $A$ can be expressed as
$$
A=\left(\frac{|x''|}{\rho}\right)^{1-\mu}
\frac{\left|a^-\left(\frac{x''\rho}{|x''|}\right)
-a^-\left(\frac{x'\rho}{|x'|}\right)\right|}{|\tilde\Delta|^\mu},
\quad
\tilde\Delta:=\frac{x''\rho}{|x''|}-\frac{x'\rho}{|x'|}\frac{|x'|}{|x''|}\ .
$$
Let $\Delta:=\frac{x''\rho}{|x''|}-\frac{x'\rho}{|x'|}$. Then $|\tilde\Delta|\ge\rho\ge{|\Delta|}/{2}$ if $x'\cdot x''\le 0$, while
$|\tilde\Delta|\ge|\Delta|\ |\sin(\widehat{x',\Delta})|
\ge|\Delta|\sin(\widehat{\pi/4})=|\Delta|/\sqrt2$ if $x'\cdot x''> 0$.
 Thus in the both cases,
$$
A\le 2\left(\frac{|x''|}{\rho}\right)^{1-\mu}
\frac{\left|a^-\left(\frac{x''\rho}{|x''|}\right)
-a^-\left(\frac{x'\rho}{|x'|}\right)\right|}{|\Delta|^\mu}\le 2|a^-|_{C^\mu(\partial B_\rho)}
\le 2|a^-|_{C^\mu(\R^n\setminus B_\rho)}\quad
\mbox{
if $\mu\in[0,1]$.}
$$
On the other hand,
$$
B\le\frac{\big(\ |x''|-|x'|\ \big)^{1-\mu}}{\rho}
\| a^-\|_{C(\partial B_\rho)}\le  \| a^-\|_{C(\R^n\setminus B_\rho)}
$$
for $\mu\in[0,1]$ and $\rho\ge 1$. This implies
$
\|a_\infty^-\|_{C^\mu(B_\rho)}\le 2\|a^-\|_{C^\mu(\R^n\setminus B_\rho)}
$
and considering also the case \mbox{$x'\in\R^n\setminus B_\rho$}, $x''\in B_\rho$ and the case $x',x''\in\R^n\setminus B_\rho$, we arrive at the desired estimate
$
\|a_\infty^-\|_{C^\mu(\R^n)}\le 3\|a^-\|_{C^\mu(\R^n\setminus B_\rho)}.
$

If $a^-\in C^{\mu_1}(\R^n)$  for some $\mu_1$ such that $0\le\mu<\mu_1\le 1$, then
\begin{multline*}
\frac{1}{3}\|a_\infty^-\|_{ C^{\mu}(\R^n)}\le \|a^-\|_{ C^{\mu}(\R^n\backslash B_{\rho})}
\le \|a^-\|_{ C(\R^n\backslash B_{\rho})}+|a^-|_{ C^{\mu}(\R^n\backslash B_{\rho})}\\
=\|a^-\|_{ C(\R^n\backslash B_{\rho})}
+ \sup_{|x'-x''|\le r,\  x'\not=x''\atop x',x''\in \R^n\backslash B_{\rho}} \frac{|a^-(x'')-a^-(x')|}{|x''-x'|^\mu}
+\sup_{|x'-x''|> r\atop x',x''\in \R^n\backslash B_{\rho}} \frac{|a^-(x'')-a^-(x')|}{|x''-x'|^\mu}\\
\le \|a^-\|_{ C(\R^n\backslash B_{\rho})}
+r^{\mu_1-\mu} \sup_{|x'-x''|\le r,\  x'\not=x''\atop x',x''\in \R^n\backslash B_{\rho}} \frac{|a^-(x'')-a^-(x')|}{|x''-x'|^{\mu_1}}
+2r^{-\mu}\sup_{ x\in \R^n\backslash B_{\rho}} |a^-(x)|\\
\le (1+2r^{-\mu})\|a^-\|_{ C(\R^n\backslash B_{\rho})}+r^{\mu_1-\mu}\|a^-\|_{ C^{\mu_1}(\R^n)}\ .
\end{multline*}
Thus for any $\varepsilon>0$ we can chose sufficiently small $r>0$ so that the last term on the right hand side is less than $\varepsilon/2$ and then chose $\rho$ sufficiently large so that the first term on the right hand side is less than $\varepsilon/2$
since $a^-(x)\to 0$ as $x\to\infty$. This means $\|a_\infty^-\|_{ C^{\mu}(\R^n)}\to 0$ as $\rho\to\infty$.

\section{On Sobolev spaces characterization, traces and extensions \label{S2.2-H}}
To make this paper more self-contained we provide here some assertions from \cite{MikJMAA2011} about Sobolev spaces characterization, traces and extensions.

\begin{theorem}\label{H_F=0-H}\cite[Theorem 2.10]{MikJMAA2011}
Let $\Omega$ be a Lipschitz domain  in $\R^n$.

(i)  If $ t\ge -\ha$, then $H^t_{\pO}=\{0\}$.

(ii) If $-\tha< t< -\ha$, then $g\in H^t_{\pO}$ if and only if $g=\gamma^*v$, i.e.,
$
  \langle g, W \rangle_{\R^n}=\langle v, \gamma W\rangle_\pO\quad \forall\
  W\in H^{-t}(\R^n),
$
with $v=\gamma_{-1}^*g\in H^{t+\ha}(\pO)$, i.e.,
$
 \langle v, w\rangle_\pO = \langle g, \gamma_{-1}w \rangle_{\R^n}\quad\forall\
  w\in H^{-t-\ha}(\pO),
$
where $v$ is independent of the choice of the non-unique operators $\gamma_{-1}$, $\gamma_{-1}^*$, and the estimate
$\|v\|_{H^{t+\ha}(\pO)}\le C \|g\|_{H^t(\R^n)}$ holds with $C$
independent of $t$.
\end{theorem}
\begin{theorem}
\label{H0=H-H} \cite[Theorem 2.12]{MikJMAA2011}
Let $\Omega$ be a Lipschitz domain  in $\R^n$ and $s\le\ha$. Then
$\D(\Omega)$ is dense in $H^s(\Omega)$, i.e.,
$H^s(\Omega)=H_0^s(\Omega)$.
\end{theorem}
\begin{theorem}\label{ExtOper-H}\cite[Theorem~2.16]{MikJMAA2011}
Let $\Omega$ be a Lipschitz domain and  $-\tha< s<\ha$, $s\not=-\ha$. There exists a bounded linear extension operator
$\s{E}^s : H^s(\Omega) \to \s{H}^s(\Omega)$, such that
$\s{E}^sg|_\Omega=g$, $\forall\ g\in H^s(\Omega)$. For $-\ha< s<\ha$ the extension operator is unique,  $(\s E^s)^*=\s E^{-s}$ and
 $
 \|\s E^s g\|_{\s H^{s}(\Omega)}\le C \|g\|_{H^{s}(\Omega)},
 $
where $C$ depends only on $s$ and on the maximum of the Lipschitz constants of the representation functions $\zeta_j$ for the boundary $\pO$, see Definition~\ref{DOktoO-H}.
\end{theorem}

\begin{lemma}\label{Hsto1-H}\cite[Lemma~2.17]{MikJMAA2011}
Let $\Omega$ and $\Omega'\subset\Omega$ be open sets, and $s\le 0$. If $u\in H^s(\Omega)$, then $\|u\|_{H^s(\Omega')}\to 0$ as the Lebesgue measure of $\Omega'$ tends to zero.
\end{lemma}

\begin{lemma}\label{Hsto3-H}\cite[Lemma~2.18]{MikJMAA2011}
Let  $\Omega_k\subset\Omega$ be a sequence of Lipschitz domains converging to a Lipschitz domain $\Omega$ and $-\ha< s<\ha$.
If $u\in H^s(\Omega)$ and $\tilde u_k=\s E^s u|_{\Omega_k}$, then  there exists a constant $C$ independent of $u$ and $k$ such that $\|\tilde u_k\|_{\s H^s(\Omega_k)}\le C\|u\|_{H^s(\Omega)}$ for all sufficiently large $k$.
\end{lemma}
\begin{remark}\label{R3.4-H} \cite[Remark~3.14]{MikJMAA2011}
If $s\in\R$, $-\ha<t<\ha$, and $A_*:H^s(\Omega)\to {H}^{t}(\Omega)$ is a linear continuous operator, then
$H^{s,t} (\Omega;A_*)=H^s(\Omega)$ by Theorem~\ref{ExtOper-H}.
\end{remark}

\begin{lemma}\cite[Lemma~3.5]{MikJMAA2011} Let $s\in\R$.
If a linear operator $A_*: H^s(\Omega)\to \D^*(\Omega)$ is
continuous,
then the space $H^{s,t} (\Omega;A_*)$ is complete for any $t\ge
-\ha$.
\end{lemma}

\noindent
{\bf Acknowledgement}\\
This research was supported by the  grant
EP/H020497/1: "Mathematical Analysis of Localized Boundary-Domain Integral Equations
 for Variable-Coefficient Boundary Value Problems" from the EPSRC, UK.



\end{document}